\documentclass[a4paper,11pt]{amsart}
\usepackage{mathrsfs}
\usepackage{cases}
\usepackage{amsfonts}
\usepackage{graphicx}
\usepackage{amsmath}
\usepackage{amssymb}
\usepackage[all]{xypic}
\usepackage[all]{xy}
\usepackage{color}
\usepackage{colordvi}
\usepackage{multicol}
\usepackage[linktocpage=true]{hyperref}
\hypersetup{colorlinks,linkcolor=blue,urlcolor=cyan,citecolor=blue}

\begin{document}
\input xy
\xyoption{all}

\numberwithin{equation}{section}
\allowdisplaybreaks
\renewcommand{\mod}{\operatorname{mod}\nolimits}
\newcommand{\proj}{\operatorname{proj.}\nolimits}
\newcommand{\rad}{\operatorname{rad}\nolimits}
\newcommand{\soc}{\operatorname{soc}\nolimits}
\newcommand{\ind}{\operatorname{inj.dim}\nolimits}
\newcommand{\id}{\operatorname{id}\nolimits}
\newcommand{\Mod}{\operatorname{Mod}\nolimits}
\newcommand{\R}{\operatorname{R}\nolimits}
\newcommand{\End}{\operatorname{End}\nolimits}
\newcommand{\ob}{\operatorname{Ob}\nolimits}
\newcommand{\Ht}{\operatorname{Ht}\nolimits}
\newcommand{\cone}{\operatorname{cone}\nolimits}
\newcommand{\rep}{\operatorname{rep}\nolimits}
\newcommand{\Ext}{\operatorname{Ext}\nolimits}
\newcommand{\Tor}{\operatorname{Tor}\nolimits}
\newcommand{\Hom}{\operatorname{Hom}\nolimits}
\newcommand{\Pic}{\operatorname{Pic}\nolimits}
\newcommand{\aut}{\operatorname{Aut}\nolimits}
\newcommand{\Fac}{\operatorname{Fac}\nolimits}
\newcommand{\Div}{\operatorname{Div}\nolimits}
\newcommand{\rank}{\operatorname{rank}\nolimits}
\newcommand{\Len}{\operatorname{Length}\nolimits}
\newcommand{\RHom}{\operatorname{RHom}\nolimits}
\renewcommand{\deg}{\operatorname{deg}\nolimits}
\renewcommand{\Im}{\operatorname{Im}\nolimits}
\newcommand{\Ker}{\operatorname{ker}\nolimits}
\newcommand{\Iso}{\operatorname{Iso}\nolimits}
\newcommand{\Coh}{\operatorname{coh}\nolimits}
\newcommand{\Qcoh}{\operatorname{Qch}\nolimits}
\newcommand{\inj}{\operatorname{inj.dim}\nolimits}

\newcommand{\Cp}{\operatorname{Cp}\nolimits}
\newcommand{\coker}{\operatorname{Coker}\nolimits}
\renewcommand{\dim}{\operatorname{dim}\nolimits}
\renewcommand{\ker}{\operatorname{Ker}\nolimits}
\renewcommand{\div}{\operatorname{div}\nolimits}
\newcommand{\Ab}{{\operatorname{Ab}\nolimits}}
\newcommand{\Cone}{{\operatorname{Cone}\nolimits}}
\renewcommand{\Vec}{{\operatorname{Vec}\nolimits}}
\newcommand{\pd}{\operatorname{proj.dim}\nolimits}
\newcommand{\sdim}{\operatorname{sdim}\nolimits}
\newcommand{\add}{\operatorname{add}\nolimits}
\newcommand{\pr}{\operatorname{pr}\nolimits}
\newcommand{\oR}{\operatorname{R}\nolimits}
\newcommand{\oL}{\operatorname{L}\nolimits}
\newcommand{\Perf}{{\mathfrak Perf}}
\newcommand{\cc}{{\mathcal C}}
\newcommand{\ce}{{\mathcal E}}
\newcommand{\cs}{{\mathcal S}}
\newcommand{\cf}{{\mathcal F}}
\newcommand{\cx}{{\mathcal X}}
\newcommand{\cy}{{\mathcal Y}}
\newcommand{\cl}{{\mathcal L}}
\newcommand{\ct}{{\mathcal T}}
\newcommand{\cu}{{\mathcal U}}
\newcommand{\cm}{{\mathcal M}}
\newcommand{\cv}{{\mathcal V}}
\newcommand{\ch}{{\mathcal H}}
\newcommand{\ca}{{\mathcal A}}
\newcommand{\mcr}{{\mathcal R}}
\newcommand{\cb}{{\mathcal B}}
\newcommand{\ci}{{\mathcal I}}
\newcommand{\cj}{{\mathcal J}}
\newcommand{\cp}{{\mathcal P}}
\newcommand{\cg}{{\mathcal G}}
\newcommand{\cw}{{\mathcal W}}
\newcommand{\co}{{\mathcal O}}
\newcommand{\cd}{{\mathcal D}}
\newcommand{\ck}{{\mathcal K}}
\newcommand{\calr}{{\mathcal R}}

\def \fg{{\mathfrak g}}
\newcommand{\ol}{\overline}
\newcommand{\ul}{\underline}
\newcommand{\cz}{{\mathcal Z}}
\newcommand{\st}{[1]}
\newcommand{\ow}{\widetilde}
\renewcommand{\P}{\mathbf{P}}
\newcommand{\pic}{\operatorname{Pic}\nolimits}
\newcommand{\Spec}{\operatorname{Spec}\nolimits}
\newtheorem{theorem}{Theorem}[section]
\newtheorem{acknowledgement}[theorem]{Acknowledgement}
\newtheorem{algorithm}[theorem]{Algorithm}
\newtheorem{axiom}[theorem]{Axiom}
\newtheorem{case}[theorem]{Case}
\newtheorem{claim}[theorem]{Claim}
\newtheorem{conclusion}[theorem]{Conclusion}
\newtheorem{condition}[theorem]{Condition}
\newtheorem{conjecture}[theorem]{Conjecture}
\newtheorem{construction}[theorem]{Construction}
\newtheorem{corollary}[theorem]{Corollary}
\newtheorem{criterion}[theorem]{Criterion}
\newtheorem{definition}[theorem]{Definition}
\newtheorem{example}[theorem]{Example}
\newtheorem{exercise}[theorem]{Exercise}
\newtheorem{lemma}[theorem]{Lemma}
\newtheorem{notation}[theorem]{Notation}
\newtheorem{problem}[theorem]{Problem}
\newtheorem{proposition}[theorem]{Proposition}
\newtheorem{solution}[theorem]{Solution}
\newtheorem{summary}[theorem]{Summary}
\newtheorem*{thm}{Theorem}
\newcommand{\qbinom}[2]{\begin{bmatrix} #1\\#2 \end{bmatrix} }

\theoremstyle{remark}
\newtheorem{remark}[theorem]{Remark}

\def \bp{{\mathbf p}}
\def \bA{{\mathbf A}}
\def \bL{{\mathbf L}}
\def \bF{{\mathbf F}}
\def \bS{{\mathbf S}}
\def \bC{{\mathbf C}}
\def \bD{{\mathbf D}}

\def \Z{{\Bbb Z}}
\def \F{{\Bbb F}}
\def \C{{\Bbb C}}
\def \N{{\Bbb N}}
\def \Q{{\Bbb Q}}
\def \G{{\Bbb G}}
\def \X{{\Bbb X}}
\def \P{{\Bbb P}}
\def \K{{\Bbb K}}
\def \E{{\Bbb E}}
\def \A{{\Bbb A}}
\def \BH{{\Bbb H}}
\def \T{{\Bbb T}}
\newcommand{\bluetext}[1]{\textcolor{blue}{#1}}
\newcommand{\redtext}[1]{\textcolor{red}{#1}}
\newcommand{\red}[1]{\redtext{ #1}}
\newcommand{\blue}[1]{\bluetext{ #1}}

\title[Semi-derived Ringel-Hall algebras and Drinfeld double]{Semi-derived Ringel-Hall algebras and Drinfeld double}

\author[M. Lu]{Ming Lu}
\address{Department of Mathematics, Sichuan University, Chengdu 610064, P.R.China}
\email{luming@scu.edu.cn}
\author[L. Peng]{Liangang Peng}
\address{Department of Mathematics, Sichuan University, Chengdu 610064, P.R.China}
\email{penglg@scu.edu.cn}

\thanks{Both authors are supported by the National Natural Science Foundations of China with Grant Nos. 11401401 and 11821001, respectively}

\begin{abstract}
Let $\ca$ be an arbitrary hereditary abelian category that may not have enough projective objects. For example, $\ca$ can be the category of finite-dimensional representations of a quiver or the category of coherent sheaves on a smooth projective curve or on a weighted projective line. Inspired by the works of Bridgeland and Gorsky, we define the semi-derived Ringel-Hall algebra of $\ca$,  denoted by $\cs\cd\ch_{\Z/2}(\ca)$, to be the localization of a quotient algebra of the Ringel-Hall algebra of the category of $\Z/2$-graded complexes over $\ca$.

We obtain the following three main results. The semi-derived Ringel-Hall algebra has a natural basis. 
A twisted version of the semi-derived Ringel-Hall algebra of $\ca$ is isomorphic to the Drinfeld double of the twisted extended Ringel-Hall algebra $\ch_{tw}^e(\ca)$ of $\ca$. If $\ca$ has a tilting object $T$, then its semi-derived Ringel-Hall algebra is isomorphic to the $\Z/2$-graded semi-derived Hall algebra $\cs\cd\ch_{\Z/2}(\add T)$ of the exact category $\add T$ defined by Gorsky,  and so is isomorphic to Bridgeland's Hall algebra of $\mod (\End(T)^{op})$.
\end{abstract}

\maketitle
 \tableofcontents
\section{Introduction}
The Hall algebra (also called Ringel-Hall algebra) of an associative algebra $A$ (or more generally, of an exact category) over a finite field $\K=\F_q$ is, by definition, an associative algebra over $\Q$ (or $\C$) with a basis indexed by the isomorphism classes of finite-dimensional $A$-modules. The structure coefficients, known as Hall numbers, count in a way the numbers of short exact sequences in $\mod A$. 
For $A$ hereditary and of finite representation type, a celebrated result by C.M. Ringel \cite{R2} shows that the twisted Ringel-Hall algebra of $A$ is isomorphic to the positive part of the corresponding quantum group. This result was generalized to the Kac-Moody type by J.A. Green \cite{Gr}. It means that the positive part of a quantum group 
and, moreover, of a Kac-Moody Lie algebra can be described in terms of the module category of a hereditary algebra.

There have been many attempts to construct the Ringel-Hall algebras from suitable categories to realize the whole quantum groups or the whole  Kac-Moody Lie algebras. 
Among them, three typical constructions should be mentioned. The first one was that of the Drinfeld double of the twisted extended Ringel-Hall algebras from  hereditary categories to realize the whole quantum groups (see \cite{Gr,X,SV}). The second one was that of  the  Ringel-Hall Lie algebras from root categories (see \cite{PX1,PX2,Hub}) to realize the Kac-Moody Lie algebras. The third one was that of the derived Ringel-Hall algebras from derived categories or more general triangulated categories with certain conditions (see \cite{T,XX}, for a different version see \cite{Kap}). However, the first two constructions are not so perfect, since each of them is built of  two natural pieces and the multiplication between the two pieces has to be put in by hand. In addition, the relationship between derived Ringel-Hall algebras and quantum groups is also unclear.

Recently, T. Bridgeland \cite{Br} found an ingenious way to solve this problem. He considered the Hall algebra of the exact category of $\Z/2$-graded complexes with projective components over an abelian category which has enough projective objects, and then localized it at the isomorphism classes of acyclic complexes. When the abelian category is hereditary, this localized Hall algebra can be used to realize the whole quantum group. 

Inspired by Bridgeland's construction, M. Gorsky \cite{Gor13} constructed the so-called semi-derived Hall algebra of $\Z$-graded or $\Z/2$-graded complexes over an exact category satisfying certain finiteness conditions. In particular, his $\Z/2$-graded version requires the condition that the exact category has enough projective objects. He proved that in this case, the semi-derived Hall algebra of  $\Z/2$-graded complexes is isomorphic to the corresponding Bridgeland's Hall algebra. Gorsky \cite{Gor18} also defined a version of semi-derived Hall algebras for Frobenius categories. 

The multiplicative structure  of the semi-derived Hall algebra in \cite{Gor13} is slightly similar to the usual Ringel-Hall multiplication but not completely. However, when the exact category is hereditary and has enough projective objects, Gorsky \cite{Gor13} proved that the multiplication of the semi-derived Hall algebra is the same as the usual Hall product, and the semi-derived Hall algebra is isomorphic to a certain quotient of the Ringel-Hall algebra of  complexes localized at the  isomorphism classes of acyclic complexes.

Note that there are many hereditary abelian categories without enough projective objects, whose Ringel-Hall algebras also have deep connections with quantum groups. For example, for the category of nilpotent finite-dimensional representations of a quiver with loops, the Ringel-Hall algebra is used by S.-J. Kang and O. Schiffmann \cite{KS} to realize one  half of a quantum generalized Kac-Moody algebra \cite{Bor,K}; for the category of coherent sheaves on a smooth projective curve or on a weighted projective line, the Ringel-Hall algebra or its Drinfeld double realizes Drinfeld's presentation of the quantum loop algebra (see \cite{Sch,Sch2,BS1,DJX,BS}).

In this paper, inspired by the works of Bridgeland and Gorsky on constructing Ringel-Hall algebras from the category of $\Z/2$-graded complexes and especially by the aforementioned result of Gorsky in \cite{Gor13}, we consider arbitrary hereditary abelian categories, which may not have enough projective objects. 
Our main results are as follows.

Assume that $\ca$ is a hereditary abelian category. After obtaining some homological properties of $\Z/2$-graded acyclic complexes in Section \ref{sec2}, we define in Section \ref{sec:semi} the so-called semi-derived Ringel-Hall algebra of $\ca$, denoted by $\cs\cd\ch_{\Z/2}(\ca)$ or $\cs\cd\ch(\ca)$, to be the localization of a quotient algebra of the Ringel-Hall algebra of $\cc_{\Z/2}(\ca)$, the category of $\Z/2$-graded complexes over $\ca$. We prove that the underlying vector space of such a Ringel-Hall algebra has a natural basis. As a consequence, $\cs\cd\ch(\ca)$ is a free module over a suitably defined quantum torus of acyclic complexes, with a basis given by the isomorphism classes of objects in the derived category of $\Z/2$-graded complexes ({see Theorem~ \ref{theorem basis of modified hall algebra}, Corollary~ \ref{corollary basis of semi-derived hall algebra2})}. This structure is similar to Gorsky's semi-derived Hall algebra. In addition, we also obtain another basis of $\cs\cd\ch(\ca)$, which provides a triangular decomposition of  $\cs\cd\ch(\ca)$
({see Theorem \ref{lemma basis of semi-derived hall algebra of A}).

In Section \ref{sec:Drinfeld}, we prove that the twisted version of $\cs\cd\ch(\ca)$ is isomorphic to the Drinfeld double of $\ch_{tw}^e(\ca)$, where $\ch_{tw}^e(\ca)$ denotes the twisted extended Ringel-Hall algebra of $\ca$ ({see Theorem \ref{theorem semi-derived hall algebra isomorphic to Drinfeld double}}). As a consequence, we obtain that the structure of the Drinfeld double of  $\ch_{tw}^e(\ca)$ can de defined on the tensor product $\ch_{tw}^e(\ca)\otimes \ch_{tw}^e(\ca)$ (instead of its completion). 
This result generalizes the corresponding results of \cite{BS1,Cr}.
Recently, S. Yanagida \cite{Y} proved the same result for Bridgeland's Hall algebra under the assumption that $\ca$ has enough projective objects.
Our result is more general and it shows that in particular for the category of finite-dimensional representations of a quiver possibly with oriented cycles and for the category of coherent sheaves on a smooth projective curve or on a weighted projective line, their twisted semi-derived Ringel-Hall algebras are isomorphic to the corresponding Drinfeld double of the twisted extended Ringel-Hall algebras.

In Section \ref{sec:tilting}, if $\ca$ has a tilting object $T$, then its semi-derived Ringel-Hall algebra is shown to be isomorphic to the $\Z/2$-graded semi-derived Hall algebra 
of the exact category $\add T$ defined by Gorsky, and so isomorphic to Bridgeland's Hall algebra of $\mod (\End(T)^{op})$ (see Theorem~ \ref{proposition isomorphism of semi derived hall algebras }, Corollary~ \ref{corollary isomorphic of hall algebras}).


Let us make some remarks about the terminology here. The semi-derived Ringel-Hall algebras defined in Definition \ref{definition of the modified Ringel-hall algebra} were called the ($\Z/2$-graded) {\em modified Ringel-Hall algebras} in a previous version of this paper (available on arXiv; see \cite{LuP}) to emphasize the difference between our definition and Gorsky's. 
   However, as pointed out by some experts, that terminology may cause some  confusion with the {\em modified quantum groups} (see \cite{Lu}). 
On the other hand, as mentioned above, the algebra we define in this paper and Gorsky's  in \cite{Gor13} have similar structures. 
So we have changed the name and notation to  the semi-derived Ringel-Hall algebra $\cs\cd\ch_{\Z/2}(\ca)$. 


The construction of semi-derived Ringel-Hall algebras in this paper can be adapted to other categories. 
For example, J.~ Lin and the second author \cite{LinP19} studied the relations between the $\Z$-graded semi-derived Ringel-Hall algebras and the derived Hall algebras; W.~ Wang and the first author \cite{LW19} used a version of semi-derived Ringel-Hall algebras for $1$-Gorenstein algebras to realize the $\imath$quantum groups arising from quantum symmetric pairs.

\vspace{2mm}
\noindent{\bf Acknowledgments.}
We thank the referee for his/her very careful reading, very helpful comments and suggestions, which greatly improved the exposition of this paper.

\section{$\Z/2$-graded complexes and homological properties of acyclic complexes}
\label{sec2}
\subsection{Categories of $\Z/2$-graded complexes}
We assume that $\ce$  is an exact category in this subsection. For the basics of exact categories, we refer to \cite{Q,Buh,Ke1}.

Let $\cc_{\Z/2}(\ce)$ be the exact category of $\Z/2$-graded complexes over $\ce$. Namely, an \emph{object} $M$ of this category is a diagram with objects and morphisms in $\ce$:
$$\xymatrix{ M^0 \ar@<0.5ex>[r]^{d^0}& M^1 \ar@<0.5ex>[l]^{d^1}  },\quad d^1d^0=d^0d^1=0.$$
All indices of components of $\Z/2$-graded objects will be understood modulo $2$.
A \emph{morphism} $s=(s^0,s^1):M\rightarrow N$ is a diagram
\[\xymatrix{  M^0 \ar@<0.5ex>[r]^{d^0} \ar[d]^{s^0}& M^1 \ar@<0.5ex>[l]^{d^1} \ar[d]^{s^1} \\
 N^0 \ar@<0.5ex>[r]^{e^0}& N^1 \ar@<0.5ex>[l]^{e^1} }  \]
with $s^{i+1}d^i=e^is^i$.

The shift functor on complexes is an involution
\begin{align}
\label{inv}
\xymatrix{\cc_{\Z/2}(\ce) \ar[r]^{*} & \cc_{\Z/2}(\ce)\ar[l],}
\end{align}
which shifts the grading and changes the sign of the differential as follows:
$$\xymatrix{ M^0 \ar@<0.5ex>[r]^{d^0}& M^1 \ar@<0.5ex>[l]^{d^1} \ar[r]^{*} & M^1 \ar@<0.5ex>[r]^{-d^1} \ar[l]& M^0 \ar@<0.5ex>[l]^{-d^0}  }.$$

For any object $X\in\ce$, we define
\begin{align*}
K_X:=&(\xymatrix{ X \ar@<0.5ex>[r]^{1}& X \ar@<0.5ex>[l]^{0}  }),\qquad \,\, K_X^*:=(\xymatrix{ X \ar@<0.5ex>[r]^{0}& X \ar@<0.5ex>[l]^{1}  }),
\\
C_X:=&(\xymatrix{ 0 \ar@<0.5ex>[r]& X \ar@<0.5ex>[l]  }),\qquad \quad C_X^*:=(\xymatrix{ X\ar@<0.5ex>[r]& 0 \ar@<0.5ex>[l]  })
\end{align*}
in $\cc_{\Z/2}(\ce)$.

Denote by $\cc^b(\ce)$ the category of bounded complexes over $\ce$. Then there exists a covering functor
$$\pi:\cc^b(\ce)\longrightarrow \cc_{\Z/2}(\ce),$$
which sends a complex $(M^i)_{i\in \Z}$ to the $\Z/2$-graded complex
$$\xymatrix{\bigoplus\limits_{i\in\Z}M^{2i}  \ar@<0.5ex>[r]& \bigoplus\limits_{i\in\Z}M^{2i+1} \ar@<0.5ex>[l]  }$$
with the naturally defined differentials. Note that $\pi$ is an exact functor.

\subsection{Homological dimensions of acyclic complexes}

Let $\cb$ be an exact category. For any $B\in\cb$, its {\em projective dimension} (denoted by $\pd_\cb (B)$) is defined to be the smallest number $i\in\N$ such that
$\Ext^{i+1}_{\cb}(B,-)=0$; dually one can define its {\em injective dimension} (denoted by $\inj_\cb (B)$). We say that $\cb$ is {\em hereditary} if $\pd_\cb (B)\leq 1$ and $\inj_\cb(B)\leq1$ for any $B\in\cb$.

Let $\ca$ be a hereditary abelian category, and $\cc_{\Z/2}(\ca)$ the category of $\Z/2$-graded complexes over $\ca$. 
In this subsection, we shall prove that the projective and injective dimensions of any acyclic complexes in $\cc_{\Z/2}(\ca)$ are at most one. 

\begin{lemma}
\label{Galois functor is dense for hereditary categories}
Let $\ca$ be a hereditary abelian category and $\pi:\cc^b(\ca)\rightarrow \cc_{\Z/2}(\ca)$ the covering functor. For any $M=\xymatrix{ M^0 \ar@<0.5ex>[r]^{f^0}& M^1 \ar@<0.5ex>[l]^{f^1}  }\in\cc_{\Z/2}(\ca)$,
there exist two short exact sequences of $\Z/2$-graded complexes:
\begin{align}
0\longrightarrow K_{\ker f^0} \longrightarrow \pi(U)\longrightarrow M\longrightarrow0 \text{ and } 0\longrightarrow M{\longrightarrow}  \pi(V)\longrightarrow K_{\coker f^0}\longrightarrow0
\end{align}
with $U,V\in\cc^b(\ca)$. Moreover, if $M$ is acyclic, then $U,V$ are acyclic complexes.
\end{lemma}
\begin{proof}
The proof is based on \cite[Propostion 3.2]{LP}. 

Denote by $i^0:\ker f^0{\rightarrow} M^0$ the kernel of $f^0$ and by ${p^1}:M^1{\rightarrow} \coker f^0$ the cokernel of $f^0$.
Since $f^1f^0=0$, there exists a morphism $h^0:\coker f^0\rightarrow \ker f^0$ such that $i^0 h^0 p^1=f^1$.
Since $\ca$ is hereditary, the action of $\Ext^1_{\ca}(\coker f^0,-)$ on the canonical epimorphism $M^0\rightarrow \Im f^0$ induces an epimorphism
$\Ext^1_{\ca}(\coker f^0,M^0)\rightarrow \Ext^1_{\ca}(\coker f^0,\Im f^0)$.
So we have the following commutative diagram such that the rows are exact:
\[\xymatrix{0\ar[r] &M^0\ar[r]^{l^0}\ar[d] &Z^0\ar[r]\ar[d]^{q^0} &\coker f^0\ar@{=}[d] \ar[r]&0\\
0\ar[r]&\Im f^0\ar[r] &M^1\ar[r]^{p^1} &\coker f^0 \ar[r]&0 }\]
Then \[\xymatrix{M^0\oplus \ker f^0 \ar@<0.5ex>[rr]^{\quad\quad(l^0,0)} &&Z^0\ar@<0.5ex>[ll]^{\quad\quad\quad\tiny \left(\begin{array}{cc}0\\ h^0p^1q^0 \end{array}\right)} }\]
is a $\Z/2$-graded complex.

So we have the following short exact sequence of $\Z/2$-graded complexes:
\[\xymatrix{0\ar[rr]&&\ker f^0 \ar@<-0.5ex>[dd]_1 \ar[rr]^{\tiny \left(\begin{array}{cc}-i^0\\1 \end{array}\right) \qquad}&&M^0\oplus \ker f^0 \ar@<-0.5ex>[dd]_{\quad\quad(l^0,0)} \ar[rr]^{\quad(1,i^0)}&&
M^0\ar@<-0.5ex>[dd]_{f^0}\ar[rr]&&0
\\
\\
0\ar[rr]&&\ker f^0\ar@<-0.5ex>[uu]_{0}\ar[rr]^{-l^0 i^0}&& Z^0\ar@<-0.5ex>[uu]_{\tiny \left(\begin{array}{cc}0\\ h^0p^1q^0 \end{array}\right)}\ar[rr]^{q^0}
&&
M^1\ar@<-0.5ex>[uu]_{f^1}\ar[rr]&&0}\]
Take $U$ to be the following complex:
$$\cdots \longrightarrow0\longrightarrow M^0\stackrel{l^0}{\longrightarrow} Z^0\xrightarrow{h^0p^1q^0} \ker f^0\longrightarrow0\longrightarrow\cdots.$$
The diagram above then gives the first desired exact sequence. One can construct the second
exact sequence dually.
\end{proof}


\begin{lemma}
\label{lemma extension 2 zero}
Let $\ca$ be a hereditary abelian category.
For any $M,K_X\in\cc_{\Z/2}(\ca)$ and $p\geq2$, we have
\begin{align*}
\Ext^{p}_{\cc_{\Z/2}(\ca)}(K_X,M)=0,\qquad \Ext^{p}_{\cc_{\Z/2}(\ca)}(K_X^*,M)=0,
\\
\Ext^{p}_{\cc_{\Z/2}(\ca)}(M,K_X)=0,\qquad \Ext^{p}_{\cc_{\Z/2}(\ca)}(M,K_X^*)=0.
\end{align*}
\end{lemma}

\begin{proof}
For any exact sequence
$$\xi: 0\longrightarrow M\longrightarrow V_p\longrightarrow 
\cdots \longrightarrow V_1\longrightarrow K_X\longrightarrow0$$
in $\cc_{\Z/2}(\ca)$,
one can obtain the following commutative diagram of exact sequences:
\[\xymatrix{ \xi':& 0\ar[r] & K_{M^0} \ar[r] \ar[d]^{g} & K_{V_p^0} \ar[r]  \ar[d]& \cdots \ar[r]&   K_{V_1^0} \ar[r] \ar[d]  & K_X\ar@{=}[d] \ar[r]& 0
\\
\xi:& 0\ar[r] & M\ar[r] &V_p \ar[r] & \cdots \ar[r] & V_1 \ar[r] & K_X \ar[r] &0 }\]
Here the exact sequence $\xi'$ is naturally determined by the $0$-th component of the exact sequence $\xi$.

Since $\ca$ is hereditary and $p\geq2$, it is clear that $[\xi']=0$, where $[\xi']$ is the equivalence class of the exact sequence $\xi'$ in $\Ext^p_{\cc_{\Z/2}(\ca)}(K_X,K_{M^0})$.
From the above commutative diagram, we have $\Ext^p_{\cc_{\Z/2}(\ca)}(K_X, g) ([\xi'])=[\xi]$ and so $[\xi]=0$. It follows that $\Ext^{p}_{\cc_{\Z/2}(\ca)}(K_X,M)=0$.

By applying the involution defined in \eqref{inv}, we have $\Ext^{p}_{\cc_{\Z/2}(\ca)}(K_X^*,M)=0$ for any $p\geq2$.

One can prove the remaining two formulas dually.
\end{proof}

Using the above lemmas, we can prove the main result of this section.

\begin{proposition}
\label{proposition extension 2 zero}
Let $\ca$ be a hereditary category.
For any $K\in\cc_{\Z/2}(\ca)$ with $K$ acyclic, we have
$$\pd_{\cc_{\Z/2}(\ca)} (K)\leq 1\text{ and }\inj_{\cc_{\Z/2}(\ca)} (K)\leq 1.$$
\end{proposition}

\begin{proof}
We denote by $\cc^{can}_{\Z/2,ac}(\ca)$ the finite extension closure of the acyclic $\Z/2$-graded complexes $K_X$ and $K_X^*$ for all $X\in\ca$. Then $\cc^{can}_{\Z/2,ac}(\ca)$ is closed under taking isomorphisms. By Lemma \ref{lemma extension 2 zero}, we can obtain that $\pd_{\cc_{\Z/2}(\ca)}(K)\leq 1$ and $\inj_{\cc_{\Z/2}(\ca)}(K)\leq 1$, for any $K\in \cc^{can}_{\Z/2,ac}(\ca)$. 

By induction on the width of bounded complexes, it is routine to prove that $\pi(X)\in\cc^{can}_{\Z/2,ac}(\ca)$ for any acyclic complex $X$ in $\cc^b(\ca)$. So $\pd_{\cc_{\Z/2}(\ca)}(\pi(X))\leq 1$ and $\inj_{\cc_{\Z/2}(\ca)}(\pi(X))\leq 1$.

For any acyclic complex $K=\xymatrix{K^0\ar@<0.5ex>[r]^{d^0} &  K^1\ar@<0.5ex>[l]^{d^1}}$, by  Lemma \ref{Galois functor is dense for hereditary categories}, we have the following
short exact sequences
\begin{equation*}
0\longrightarrow K_{\ker d^0}\longrightarrow \pi(U)\longrightarrow K\longrightarrow0 \text{ and }0\longrightarrow K\longrightarrow \pi(V) \longrightarrow K_{\coker d^0}\longrightarrow0
\end{equation*}
with $U,V\in\cc^{b}(\ca)$ acyclic. 
It follows that $\pd_{\cc_{\Z/2}(\ca)}(K)\leq 1$ and  $\inj_{\cc_{\Z/2}(\ca)}(K)\leq 1$.
\end{proof}

\subsection{Euler forms}

From now on, we take the field $\K=\mathbb F_q$, a finite field of $q$ elements.
In the following, we always assume that $\ca$ is a hereditary abelian $\K$-linear category which is essentially small with finite-dimensional homomorphism and extension spaces.

Let $\cc_{\Z/2,ac}({\ca})$ be the full subcategory of $\cc_{\Z/2}({\ca})$ consisting of acyclic complexes. Denote by $\Iso(\cc_{\Z/2}(\ca))$ the set of isomorphism classes $[M]$ of
$\cc_{\Z/2}(\ca)$.

By Proposition \ref{proposition extension 2 zero}, for any $[K],[M]\in \Iso(\cc_{\Z/2}(\ca))$ with $K$ acyclic, the following alternating products are well defined:
\begin{align*}
\langle [K],[M]\rangle =\prod_{p=0}^{+\infty} |\Ext^p_{\cc_{\Z/2}({\ca})}(K,M)|^{(-1)^p}=\frac{|\Hom_{\cc_{\Z/2}({\ca})}(K,M)|}{|\Ext^1_{\cc_{\Z/2}({\ca})}(K,M)|}
\end{align*}
and
\begin{align*}
\langle [M],[K]\rangle =\prod_{p=0}^{+\infty} |\Ext^p_{\cc_{\Z/2}({\ca})}(M,K)|^{(-1)^p}=\frac{|\Hom_{\cc_{\Z/2}({\ca})}(M,K)|}{|\Ext^1_{\cc_{\Z/2}({\ca})}(M,K)|}.
\end{align*}
We call them the {\em Euler forms}. They descend to bilinear forms on the Grothendieck groups $K_0(\cc_{\Z/2,ac}(\ca))$
and $K_0(\cc_{\Z/2}(\ca))$, denoted by the same symbol:
$$\langle\cdot,\cdot\rangle:K_0(\cc_{\Z/2,ac}({\ca}))\times K_0(\cc_{\Z/2}(\ca))\longrightarrow \Q^\times,$$
and
$$\langle\cdot,\cdot\rangle:K_0(\cc_{\Z/2}(\ca))\times K_0(\cc_{\Z/2,ac}({\ca}))\longrightarrow \Q^\times.$$
We can use the same symbol, since these two forms coincide on $K_0(\cc_{\Z/2,ac}({\ca}))\times K_0(\cc_{\Z/2,ac}({\ca}))$.

Let $K_0(\ca)$ be the Grothendieck group of $\ca$.
For any $A\in\ca$, we denote by $\widehat{A}$ the corresponding element in the Grothendieck group $K_0(\ca)$.
We also use $\langle \cdot,\cdot\rangle$ to denote the Euler form of $\ca$, i.e.,
\begin{align*}
\langle \widehat{A}, \widehat{B}\rangle=\frac{|\Hom_\ca(A,B)|}{|\Ext^1_\ca(A,B)|}, \text{ for any }A,B\in\ca.
\end{align*}

Recall that $C_X=(\xymatrix{ 0 \ar@<0.5ex>[r]& X \ar@<0.5ex>[l]  })$ and
$C_X^*=(\xymatrix{ X\ar@<0.5ex>[r]& 0 \ar@<0.5ex>[l]  })$ for any $X\in\ca$.

\begin{proposition}
\label{lema euler form}
For any $A ,B \in\ca$, we have the following.
\begin{align*}
\langle [C_A],[K_B]\rangle=\langle \widehat{A}, \widehat{B}\rangle,\quad \langle [C_A^*],[K_B]\rangle=1, \quad \langle [K_B],[C_A]\rangle=1,\quad
\langle [K_B],[C_A^*]\rangle=\langle \widehat{B},\widehat{A}\rangle.
\\
\langle [C_A],[K_B^*]\rangle=1, \quad \langle [C_A^*],[K_B^*]\rangle=\langle \widehat{A},\widehat{B}\rangle, \quad \langle [K_B^*],[C_A]\rangle=\langle \widehat{B},\widehat{A}\rangle,\quad
\langle [K_B^*],[C_A^*]\rangle=1.
\end{align*}
\end{proposition}

\begin{proof}
First, let us prove  $\langle [C_A],[K_B]\rangle=\langle \widehat{A}, \widehat{B}\rangle$.

Clearly, $\Hom_{\cc_{\Z/2}(\ca)} ( C_A, K_B )=\Hom_\ca(A,B)$.
For any short exact sequence in $\cc_{\Z/2}(\ca)$:
\[\xymatrix{ 0\ar[r]&B \ar@<-0.5ex>[d]_1 \ar[r]^{l_2}& X^0 \ar@<-0.5ex>[d]_f \ar[r] & 0 \ar@<-0.5ex>[d]\ar[r]&0
\\
 0\ar[r]&B\ar@<-0.5ex>[u]_0 \ar[r]^{l_1} &X^1\ar@<-0.5ex>[u]_g \ar[r] ^{\pi_1}&A\ar@<-0.5ex>[u]\ar[r]&0,
}\]
we have $l_2:B\rightarrow X^0$ is an isomorphism. It follows from $f=l_1l_2^{-1}$ that $f$ is injective,  and then $g=0$ since $fg=0$. So the short exact sequence is equivalent to
\[\xymatrix{ 0\ar[r]&B \ar@<-0.5ex>[d]_1 \ar[r]^{1}& B \ar@<-0.5ex>[d]_{l_1} \ar[r] & 0 \ar@<-0.5ex>[d]\ar[r]&0
\\
 0\ar[r]&B\ar@<-0.5ex>[u]_0 \ar[r]^{l_1} &X^1\ar@<-0.5ex>[u]_0 \ar[r] ^{\pi_1}&A\ar@<-0.5ex>[u]\ar[r]&0.
}\]
Then $\Ext^1_{\cc_{\Z/2}(\ca)}(C_A,K_B)=\Ext^1_\ca(A,B)$. Therefore, $\langle [C_A],[K_B]\rangle=\langle \widehat{A},\widehat{B}\rangle$.

Second, let us prove $\langle [C_A^*],[K_B]\rangle=1$.

Clearly, $\Hom_{\cc_{\Z/2}(\ca)} ( C_A^*, K_B )=0$.
For any short exact sequence in $\cc_{\Z/2}(\ca)$:
\[\xymatrix{0\ar[r]& B \ar@<-0.5ex>[d]_1 \ar[r]^{l_2}& X^0 \ar@<-0.5ex>[d]_f \ar[r]^{\pi_2} & A \ar@<-0.5ex>[d]\ar[r]&0
\\
0\ar[r]& B\ar@<-0.5ex>[u]_0 \ar[r]^{l_1} &X^1\ar@<-0.5ex>[u]_g \ar[r]&0\ar@<-0.5ex>[u]\ar[r]&0,
}\]
we have that $l_1:B\rightarrow X^1$ is an isomorphism. Then $g=0$ since $gl_1=0$. Furthermore, we have that $l_2$ is a section and $f$ is a retraction since $fl_2=l_1$. This implies $X^0\cong A\oplus B$.
So the short exact sequence is equivalent to
\[\xymatrix{ 0\ar[r]&B \ar@<-0.5ex>[d]_1 \ar[r]^{\tiny\left( \begin{array}{cc} 1\\0 \end{array}\right)}& B\oplus A \ar@<-0.5ex>[d]_{(1,0)} \ar[r]^{\quad(0,1)} & A \ar@<-0.5ex>[d]\ar[r]&0
\\
 0\ar[r]&B\ar@<-0.5ex>[u]_0 \ar[r]^{1} &B\ar@<-0.5ex>[u]_0 \ar[r]&0\ar@<-0.5ex>[u]\ar[r]&0,
}\]
which is split.
Then
we have $\Ext^1_{\cc_{\Z/2}(\ca)}(C_A^*,K_B)=0$. So $\langle [C_A^*],[K_B]\rangle=1$.

The remaining formulas are proved similarly.
\end{proof}

\begin{corollary}
\label{lemma coincide of Euler forms}
For any $X_1,X_2\in\ca$, we have
\begin{eqnarray*}
&&\langle [K_{X_1}],[ K_{X_2}]\rangle=\langle \widehat{X_1},\widehat{X_2}\rangle, \quad \langle [K_{X_1}^*], [K_{X_2}^*]\rangle=\langle \widehat{X_1},\widehat{X_2}\rangle,\\
&&\langle [K_{X_1}], [K_{X_2}^*]\rangle=\langle \widehat{X_1},\widehat{X_2}\rangle,\quad\langle [K_{X_1}^*], [K_{X_2}]\rangle=\langle \widehat{X_1},\widehat{X_2}\rangle.
\end{eqnarray*}
\end{corollary}
\begin{proof}
We only prove $\langle [K_{X_1}], [K_{X_2}^*]\rangle=\langle \widehat{X_1},\widehat{X_2}\rangle$. The others are proved similarly.

There is a short exact sequence
$$0\longrightarrow C_{X_1}\longrightarrow K_{X_1}\longrightarrow C_{X_1}^*\longrightarrow0.$$
Then Proposition \ref{lema euler form} shows that
\begin{align*}
\langle [K_{X_1}], [K_{X_2}^*]\rangle=&\langle[C_{X_1}\oplus C_{X_1}^*],[ K_{X_2}^*]\rangle\\
=&\langle[C_{X_1}], [K_{X_2}^*]\rangle\langle[C_{X_1}^*], [K_{X_2}^*]\rangle\\
=&\langle \widehat{X_1},\widehat{X_2}\rangle.
\end{align*}
\end{proof}

\section{Semi-derived Ringel-Hall algebras}
\label{sec:semi}
In this section, we define the semi-derived Ringel-Hall algebras for hereditary abelian categories.

\subsection{Ringel-Hall algebras}

Let $\ce$ be an essentially small exact category, linear over the finite field $\K=\F_q$.
Assume that $\ce$ has finite morphism and extension spaces:
$$|\Hom_\ce(A,B)|<\infty,\quad |\Ext^1_\ce(A,B)|<\infty,\,\,\forall A,B\in\ce.$$

Given objects $A,B,C\in\ce$, define $\Ext^1_\ce(A,C)_B\subseteq \Ext^1_\ce(A,C)$ to be the subset parameterising extensions with the middle term  isomorphic to $B$. We define the Hall algebra (also called Ringel-Hall algebra) $\ch(\ce)$ to be the $\Q$-vector space whose basis is formed by the isomorphism classes $[A]$ of objects $A$ of $\ce$, with the multiplication
defined by
$$[A]\diamond [C]=\sum_{[B]\in \Iso(\ce)}\frac{|\Ext_\ce^1(A,C)_B|}{|\Hom_\ce(A,C)|}[B].$$
It is well known that
the algebra $\ch(\ce)$ is associative and unital. The unit is given by $[0]$, where $0$ is the zero object of $\ce$; see \cite{R0} and also \cite{Rie,P,Hub,Br}.

\begin{remark}
Ringel's version of Hall algebra \cite{R0} uses a different Hall product, but these two versions of Hall algebras are isomorphic by rescaling the generators by the orders of automorphism groups of objects.
\end{remark}

\subsection{Semi-derived Ringel-Hall algebras}

In the following, we always assume that $\ca$ is a hereditary abelian $\K$-linear category which is essentially small with finite-dimensional homomorphism and extension spaces.

Let $\ch(\cc_{\Z/2}(\ca))$ be the Ringel-Hall algebra of $\cc_{\Z/2}(\ca)$. By definition, $\ch(\cc_{\Z/2}(\ca))$ has a basis formed by the isomorphism classes $[M]$ of objects $M$ of $\cc_{\Z/2}(\ca)$, with the product given by
\begin{equation*}
[L]\diamond [M]=\sum_{[X]\in \Iso(\cc_{\Z/2}(\ca))}\frac{|\Ext^1_{\cc_{\Z/2}(\ca)}(L,M)_X|}{|\Hom_{\cc_{\Z/2}(\ca)}(L,M)|}[X].
\end{equation*}
It is well known that
$\ch(\cc_{\Z/2}(\ca))$ is a $K_0(\cc_{\Z/2}(\ca))$-graded algebra, where $K_0(\cc_{\Z/2}(\ca))$ is the Grothendieck group of $\cc_{\Z/2}(\ca)$.

Let $I_{\Z/2}$ be the two-sided ideal of $\ch(\cc_{\Z/2}(\ca))$ generated by all differences $[L]-[K\oplus M]$ if there is a short exact sequence
\begin{equation}
  \label{eq:ideal}
 0 \longrightarrow K \longrightarrow L \longrightarrow M \longrightarrow 0
\end{equation}
with $K$ acyclic. Then $I_{\Z/2}$ is generated by $K_0(\cc_{\Z/2}(\ca))$-homogeneous elements.


Let $\ch(\cc_{\Z/2}(\ca))/I_{\Z/2}$ be the quotient algebra. Then $\ch(\cc_{\Z/2}(\ca))/I_{\Z/2}$ is also a $K_0(\cc_{\Z/2}(\ca))$-graded algebra.  We also denote by $\diamond$ the induced multiplication in $\ch(\cc_{\Z/2}(\ca))/I_{\Z/2}$. In the following, we shall use the same symbols both in $\ch(\cc_{\Z/2}(\ca))$ and $\ch(\cc_{\Z/2}(\ca))/I_{\Z/2}$.

\begin{lemma}
\label{lemma multiplcation in quotient algebra}
For any $K\in\cc_{\Z/2,ac}(\ca)$ and $M\in\cc_{\Z/2}(\ca)$, we have
\begin{equation}
\label{prod:MK}
 [M]\diamond [K]=\frac{1}{\langle [M],[K]\rangle}[M\oplus K]
\end{equation}
in $\ch(\cc_{\Z/2}(\ca))/I_{\Z/2}$.
In particular, for any $K_1,K_2\in\cc_{\Z/2,ac}(\ca)$, we have
\begin{equation}
\label{eq:prod KK}
[K_1]\diamond [K_2]=\frac{1}{\langle [K_1],[K_2]\rangle}[K_1\oplus K_2]
\end{equation}
in $\ch(\cc_{\Z/2}(\ca))/I_{\Z/2}$.
\end{lemma}
\begin{proof}
For a short exact sequence $0 \rightarrow K \rightarrow X \rightarrow M \rightarrow 0$, we obtain that $[X]-[K\oplus M]\in I_{\Z/2}$, which yields  in $\ch(\cc_{\Z/2}(\ca))/I_{\Z/2}$ that
\begin{eqnarray*}
[M]\diamond [K]&=&\sum_{[X]\in \Iso(\cc_{\Z/2}(\ca))}\frac{|\Ext^1_{\cc_{\Z/2}(\ca)}(M,K)_X|}{|\Hom_{\cc_{\Z/2}(\ca)}(M,K)|}[X]\\
&=&\frac{|\Ext^1_{\cc_{\Z/2}(\ca)}(M,K)|}{|\Hom_{\cc_{\Z/2}(\ca)}(M,K)|}[M\oplus K]\\
&=&\frac{1}{\langle [M],[K]\rangle}[M\oplus K].
\end{eqnarray*}
\end{proof}

In order to define the desired algebra, we need to take a localization of $\ch(\cc_{\Z/2}(\ca))/I_{\Z/2}$. We refer the readers to \cite{Lam} for the theory of localizations of noncommutative rings.
For convenience, we recall the (right) localization of rings in the following.
Let $A$ be a ring with identity $1$, $S$ a subset of $A$ closed under multiplication, and $1\in S$. Recall that a \emph{right localization of $A$ with respect to $S$} is a ring $R$ and a ring morphism $i:A\rightarrow R$ such that
\begin{itemize}
\item[(i)] $i(s)$ is a unit in $R$ for each $s\in S$,
\item[(ii)] every element of $R$ has the form $i(a)i(s)^{-1}$ for some $a\in A$, $s\in S$,
\item[(iii)] $i(a)i(s)^{-1}=i(b)i(s)^{-1}$ if and only if $at=bt$ for some $t\in S$.
\end{itemize}
Such $R$ is a \emph{universal $S$-inverting} ring, and so is unique. Thus we can safely denote $R$ by $AS^{-1}$ (or $A[S^{-1}]$) when it exists. We will suppress the map $i$ and write the elements of $AS^{-1}$ as $as^{-1}$.

We say that $S$ satisfies the \emph{right Ore condition} if for any $a\in A$ and $s\in S$, there exist $a_1\in A$ and $s_1\in S$ such that $sa_1=as_1$. We say that $S$ is \emph{right reversible} if for any $a\in A$, $s\in S$ such that $sa=0$ in $A$,  there exists a $t\in S$ such that $at=0$ in $A$.

{\O}re's localization Theorem states that the right localization $AS^{-1}$ exists if and only if $S$ is a right Ore, right reversible subset of $A$.

Returning to Ringel-Hall algebras, we consider the following subset of $\ch(\cc_{\Z/2}(\ca))/I_{\Z/2}$:
 \begin{align}
 \label{multiset}
S_{\Z/2}:=\{a[K] \in \ch(\cc_{\Z/2}(\ca))/I_{\Z/2} \mid  a \in \Q^\times, K\in\cc_{\Z/2,ac}(\ca)\},
 \end{align}
which is a multiplicatively closed subset with the identity $[0]\in S_{\Z/2}$.

For $N=\xymatrix{N^0\ar@<0.5ex>[r]^{d^0} & N^1\ar@<0.5ex>[l]^{d^1}}$, we denote by $\Cone(1_N)$ the complex
\[N^0\oplus N^1\xymatrix{\ar@<1ex>[rr]^{\tiny\left(\begin{array}{cc}d_0&1\\
0&-d_1 \end{array} \right)} && \ar@<1ex>[ll]^{\tiny\left(\begin{array}{cc}d_1&1\\
0&-d_0 \end{array}\right) } } N^1\oplus N^0.\]
It is not hard to see that $\Cone(1_N)$ is acyclic, and there exist the following two short exact sequences
\begin{align*}
0\longrightarrow N\longrightarrow \Cone(1_{N})\longrightarrow N^*\longrightarrow0
\text{ and }
0\longrightarrow N^*\longrightarrow \Cone(1_{N^*})\longrightarrow N\longrightarrow0.
\end{align*}

Denote by
\begin{align}
\label{set}
\Omega_{\Z/2}:= \{[L]-&[K\oplus M]\mid\exists \text{ a short exact sequence }
\\
\notag
&0\longrightarrow K \longrightarrow L\longrightarrow M \longrightarrow 0\mbox{ with }K\mbox{ acyclic} \}.
\end{align}
Note that
\begin{align}
I_{\Z/2}=(\Omega_{\Z/2}),
\end{align}
i.e., $I_{\Z/2}$ is the two-sided ideal generated by $\Omega_{\Z/2}$.

\begin{lemma}
\label{lem:existC}
For any $N\in\cc_{\Z/2}(\ca)$ such that $|\Ext^1_{\cc_{\Z/2}(\ca)}(K,M)_N|\neq0$, there exists a complex $A_N\in\cc_{\Z/2,ac}(\ca)$ such that
\begin{align*}
[A_N\oplus M]-[\Cone(1_{M^*})\oplus K\oplus M], \,\, [A_N\oplus M]-[\Cone(1_{M^*})\oplus N] \in \Omega_{\Z/2}.
\end{align*}
\end{lemma}

\begin{proof}
Since $|\Ext^1_{\cc_{\Z/2}(\ca)}(K,M)_N|\neq0$, there is a short exact sequence $0\rightarrow M\rightarrow N\rightarrow K\rightarrow0$. Note that $K$ is acyclic and so its projective dimension is less than or equal to 1. Then $\Ext^1_{\cc_{\Z/2}(\ca)}(K,-)$ sends the natural epimorphism $\Cone(1_{M^*})\rightarrow M$ 
to an epimorphism
$$\Ext^1_{\cc_{\Z/2}(\ca)}(K,\Cone(1_{M^*}))\rightarrow \Ext^1_{\cc_{\Z/2}(\ca)}(K,M),$$
namely we have
the following commutative diagram of short exact sequences:
\[\xymatrix{
 \Cone(1_{M^*}) \ar[r] \ar[d]  & A_N\ar[r] \ar[d] & K \ar@{=}[d] 
 \\
 M\ar[r] & N\ar[r] & K
  }\]
So we have a short exact sequence
$$0\longrightarrow \Cone(1_{M^*}) \longrightarrow A_N\oplus M\longrightarrow K\oplus M\longrightarrow0,$$
and then $[A_N\oplus M]-[\Cone(1_{M^*})\oplus K\oplus M]\in\Omega_{\Z/2}$ by noting that $\Cone(1_{M^*})$ is acyclic.

Since the above commutative diagram is a pushout and clearly also a pullback, we have
a short exact sequence
$$0\longrightarrow \Cone(1_{M^*}) \longrightarrow A_N\oplus M\longrightarrow N\longrightarrow0.$$
It follows that $[A_N\oplus M]-[\Cone(1_{M^*})\oplus N]\in \Omega_{\Z/2}$ since $\Cone(1_{M^*})$ is acyclic.
\end{proof}

\begin{lemma}
\label{lem:right Ore}
For any $K,M\in\cc_{\Z/2}(\ca)$ with $K$ acyclic,  
we have
\begin{align}
\label{equation proposition localizaition of Ringel-Hall algebra 1}
[K]\diamond [M]\diamond [\Cone(1_{M^*})]= \frac{ \langle [M],[ K]\rangle}{\langle [K],[M]\rangle  } [M]\diamond [K]\diamond [\Cone(1_{M^*})]
\end{align}
in $\ch(\cc_{\Z/2}(\ca))/I_{\Z/2}$.
\end{lemma}

\begin{proof}

By definition, we have
\begin{equation*}
[K]\diamond [M]=\sum_{[N]\in\Iso(\cc_{\Z/2}(\ca))} \frac{|\Ext^1_{\cc_{\Z/2}(\ca)}(K,M)_N|}{|\Hom_{\cc_{\Z/2}(\ca)}(K,M)|}[N].
\end{equation*}
For any $[N]\in\Iso(\cc_{\Z/2}(\ca))$ such that $|\Ext^1_{\cc_{\Z/2}(\ca)}(K,M)_N|\neq0$, by Lemma \ref{lem:existC}, we have
$$[N\oplus \Cone(1_{M^*})]-[\Cone(1_{M^*})\oplus K\oplus M]\in I_{\Z/2},$$
and so
\begin{align*}
&[K]\diamond [M]\diamond [\Cone(1_{M^*})]
\\
=&\sum_{[N]\in\Iso(\cc_{\Z/2}(\ca))} \frac{|\Ext^1_{\cc_{\Z/2}(\ca)}(K,M)_N|}{|\Hom_{\cc_{\Z/2}(\ca)}(K,M)|}[N]\diamond [\Cone(1_{M^*})]
\\
=&\sum_{[N]\in\Iso(\cc_{\Z/2}(\ca))} \frac{|\Ext^1_{\cc_{\Z/2}(\ca)}(K,M)_N|}{|\Hom_{\cc_{\Z/2}(\ca)}(K,M)|}\frac{1}{\langle [N], [\Cone(1_{M^*})] \rangle} [N\oplus \Cone(1_{M^*})]
\\
=& \sum_{[N]\in\Iso(\cc_{\Z/2}(\ca))} \frac{|\Ext^1_{\cc_{\Z/2}(\ca)}(K,M)_N|}{|\Hom_{\cc_{\Z/2}(\ca)}(K,M)|}\frac{1}{\langle [K\oplus M], [\Cone(1_{M^*})] \rangle} [\Cone(1_{M^*})\oplus K\oplus M]
\\
=& \frac{1}{\langle [K],[M]\rangle \langle [K\oplus M], [\Cone(1_{M^*})] \rangle} [\Cone(1_{M^*})\oplus K\oplus M]
\\
=& \frac{ \langle [M], [K\oplus \Cone(1_{M^*})]\rangle}{\langle [K],[M]\rangle\langle [K\oplus M], [\Cone(1_{M^*})] \rangle} [M]\diamond [K\oplus \Cone(1_{M^*})]
\\
=&\frac{ \langle [M],[ K]\rangle}{\langle [K],[M]\rangle  } [M]\diamond [K]\diamond [\Cone(1_{M^*})]
\end{align*}
in $\ch(\cc_{\Z/2}(\ca))/I_{\Z/2}$.
\end{proof}

\begin{proposition}
\label{proposition localizaition of Ringel-Hall algebra}
The multiplicatively closed subset $S_{\Z/2}$ is a right Ore, right reversible subset of $\ch(\cc_{\Z/2}(\ca))/I_{\Z/2}$. Equivalently, the right localization of $\ch(\cc_{\Z/2}(\ca))/I_{\Z/2}$ with respect to $S_{\Z/2}$ exists, denoted by
$(\ch(\cc_{\Z/2}(\ca))/I_{\Z/2})[S_{\Z/2}^{-1}]$.
\end{proposition}
\begin{proof}
From Lemma \ref{lem:right Ore}, it is clear that $S_{\Z/2}$ is a right Ore subset of $\ch(\cc_{\Z/2}(\ca))/I_{\Z/2}$.

Assume that $[K]\diamond( \sum_{i=1}^n a_i [M_i])=0$. Since $\ch(\cc_{\Z/2}(\ca))/I_{\Z/2}$ is a $K_0(\cc_{\Z/2}(\ca))$-graded algebra, we can assume that all $[M_i]\,\,(1\leq i\leq n)$ have the same $K_0(\cc_{\Z/2}(\ca))$-degree.
By the equality (\ref{equation proposition localizaition of Ringel-Hall algebra 1}), it is easy to see that there exists a complex $T\in\cc_{\Z/2,ac}(\ca)$ such that $[K]\diamond [M_i]\diamond [T]= \frac{ \langle [M_i],[ K]\rangle}{\langle [K],[M_i]\rangle  } [M_i]\diamond [K]\diamond [T]$ for all $1\leq i\leq n$. So we obtain that
$$0=[K]\diamond( \sum_{i=1}^n a_i [M_i])\diamond [T]= \sum_{i=1}^n a_i \frac{\langle [M_i],[K]\rangle }{\langle [K],[M_i]\rangle }[M_i] \diamond [K]\diamond [T]. $$
By our assumption,
$\frac{\langle [M_i],[K]\rangle }{\langle [K],[M_i]\rangle }$ are all equal for $1\leq i\leq n$. We denote this quotient by $b$. Note that $b\neq0$.
So $$ (\sum_{i=1}^n a_i [M_i]) \diamond( b [K\oplus T])=\sum_{i=1}^n a_ib\langle [K],[T]\rangle [M_i] \diamond [K]\diamond [T] =0.$$
Thus, $S_{\Z/2}$ is a right reversible subset.
\end{proof}

\begin{definition}
\label{definition of the modified Ringel-hall algebra}
For any hereditary abelian $\K$-linear category $\ca$ which is essentially small with finite-dimensional homomorphism and extension spaces, $(\ch(\cc_{\Z/2}(\ca))/I_{\Z/2})[S_{\Z/2}^{-1}]$ is called the ($\Z/2$-graded) semi-derived Ringel-Hall algebra of $\ca$, and denoted by $\cs\cd\ch_{\Z/2}(\ca)$ or $\cs\cd\ch(\ca)$.
\end{definition}

\begin{remark}
In a previous version of this paper \cite{LuP} (available on arXiv), the algebra in Definition \ref{definition of the modified Ringel-hall algebra} was called the ($\Z/2$-graded) {\em modified Ringel-Hall algebra} of $\ca$, and denoted by $\cm\ch_{\Z/2}(\ca)$ therein. To avoid confusion with the {\em modified quantum groups} (see \cite{Lu}), we changed the name to the present one. The present name comes from the paper \cite{Gor13} by Gorsky.
He defined in \cite{Gor13} the {\em semi-derived Hall algebra}  for an exact category  with enough projective objects (we shall recall his definition in Definition \ref{def:semiderivedGor} below). Our definition is quite different from his. However, for a hereditary abelian category with enough projective objects, he proved that the semi-derived Hall algebra defined in \cite{Gor13} has the same form as the algebra defined above.
\end{remark}

In the following, we also denote the multiplication in $\cs\cd\ch(\ca)$ (namely, $(\ch(\cc_{\Z/2}(\ca))/I_{\Z/2})[S_{\Z/2}^{-1}]$) by $\diamond$.

\begin{lemma}
\label{corollary hall multiplicatoin of acyclic complexes}
For any $K_1,K_2\in\cc_{\Z/2,ac}(\ca)$ and $M\in\cc_{\Z/2}(\ca)$,  we have
\begin{align}
\label{eq:KMMK}
&[M]\diamond [K]=\frac{1}{\langle [M],[K]\rangle}[K\oplus M],\qquad
[K]\diamond [M]=\frac{1}{\langle [K],[M]\rangle}[K\oplus M],
\\
&[K_1]^{-1}\diamond [K_2]^{-1}=\langle [K_2],[K_1]\rangle [K_1\oplus K_2]^{-1},\quad [K_1]^{-1}\diamond [K_2]=\frac{\langle [K_1], [K_2]\rangle}{\langle [K_2],[K_1]\rangle}[K_2]\diamond [K_1]^{-1}
\end{align}
in $\cs\cd\ch(\ca)$.
\end{lemma}

\begin{proof}
The first formula follows from \eqref{prod:MK}.
From (\ref{equation proposition localizaition of Ringel-Hall algebra 1}), we obtain that
$$[K]\diamond[M]\diamond[\Cone(1_{M^*})]= \frac{ \langle [M], [K]\rangle}{\langle [K],[M]\rangle  } [M]\diamond [K]\diamond [\Cone(1_{M^*})],$$ 
which implies that
$$[K]\diamond[M]=\frac{ \langle [M], [K]\rangle}{\langle [K],[M]\rangle  } [M]\diamond [K]$$
in $\cs\cd\ch(\ca)$. This proves the second formula.

The remaining two formulas follow from  \eqref{eq:prod KK}.
\end{proof}

\subsection{Quantum torus and the bimodule $\cm_{\Z/2}(\ca)$}
\label{subsec:quantum torus}
The existence of a good basis of $\cs\cd\ch(\ca)$ is crucial for our later study, but it is difficult to find it directly. So our strategy is to construct a $\mathbb{Q}$-linear quotient space of $\ch(\cc_{\Z/2}(\ca))$ and to find a good basis there. Then we prove that this space is naturally isomorphic to $\cs\cd\ch(\ca)$ as $\mathbb{Q}$-linear spaces.

Consider the set $\Iso(\cc_{\Z/2,ac}({\ca}))$ of isomorphism classes $[K]$ of acyclic $\Z/2$-complexes and its quotient by the following set of relations:
\begin{align}
\label{relat}
\{ [K_2]=[K_1\oplus K_3]\mid \exists \text{ a short exact sequence } 0\longrightarrow K_1\longrightarrow K_2\longrightarrow K_3\longrightarrow0\}.
\end{align}
If we endow $\Iso(\cc_{\Z/2,ac}({\ca}))$ with the addition given by direct sums, this quotient gives the \emph{Grothendieck monoid} $M_0(\cc_{\Z/2,ac}({\ca}))$.
Define the \emph{quantum affine space} $\A_{\Z/2,ac}({\ca})$ as the $\Q$-monoid algebra of the Grothendieck monoid $M_0(\cc_{\Z/2,ac}({\ca}))$, with the multiplication twisted by the inverse of the Euler form (see \eqref{eq:prod KK}),
i.e., the product of $[K_1],[K_2]\in\Iso(\cc_{\Z/2,ac}({\ca}))$ is defined as follows:
$$[K_1]\diamond[K_2]:=\frac{1}{\langle [K_1],[K_2]\rangle}[K_1\oplus K_2].$$

Define the \emph{quantum torus} $\T_{\Z/2,ac}({\ca})$ as the $\Q$-group algebra of $K_0(\cc_{\Z/2,ac}({\ca}))$, with the
multiplication twisted by the inverse of the Euler form as above. Namely,
$\T_{\Z/2,ac}(\ca)$ is the right and left localization of $\A_{\Z/2,ac}(\ca)$ with respect to the set $\Iso(\cc_{\Z/2,ac}({\ca}))$ (here $\A_{\Z/2,ac}(\ca)$ is considered with the twisted multiplication defined above). 

We define $J_{\Z/2}$ to be the {\em $\Q$-linear subspace} of $\ch(\cc_{\Z/2}(\ca))$ spanned by
$$\{[L]-[K\oplus M]\mid \exists \text{ a short exact sequence }0\longrightarrow K \longrightarrow L\longrightarrow M \longrightarrow 0\mbox{ with }K\mbox{ acyclic} \}.$$
Note that
$$J_{\Z/2}={\rm Span}_\Q \Omega_{\Z/2}.$$
Similar to \cite{Gor13}, we define an $\A_{\Z/2,ac}(\ca)$-bimodule structure on the quotient space $\ch(\cc_{\Z/2}(\ca))/J_{\Z/2}$ by setting
\begin{equation}
\label{definition of bimodule}
[K]\diamond[M]:=\frac{1}{\langle [K],[M]\rangle}[K\oplus M],\quad [M]\diamond[K]:=\frac{1}{\langle [M],[K]\rangle}[M\oplus K]
\end{equation}
for any $[K]\in \A_{\Z/2,ac}(\ca)$ and $[M]\in \ch(\cc_{\Z/2}(\ca))/J_{\Z/2}$. Our notation is justified by Lemma \ref{remark coincide of the structure of two bimodule} below.
Furthermore, we set
\begin{align}
\label{def:bimodule}
\cm_{\Z/2}(\ca):= \T_{\Z/2,ac}(\ca)\otimes_{\A_{\Z/2,ac}(\ca)}  \ch(\cc_{\Z/2}(\ca))/J_{\Z/2} \otimes_{\A_{\Z/2,ac}(\ca)} \T_{\Z/2,ac}(\ca),
\end{align}
which is a bimodule over the quantum torus $\T_{\Z/2,ac}(\ca)$. 

Let $\widehat{\Upsilon}: \ch(\cc_{\Z/2}(\ca))\rightarrow \ch(\cc_{\Z/2}(\ca))/J_{\Z/2}$ be the natural projection, and $\Upsilon: \ch(\cc_{\Z/2}(\ca))\rightarrow \cm_{\Z/2}(\ca)$ be the composition
\begin{align}
\label{eq:morUpsilon}
\ch(\cc_{\Z/2}(\ca))\stackrel{\widehat{\Upsilon}}{\longrightarrow} \ch(\cc_{\Z/2}(\ca))/J_{\Z/2} \stackrel{\rm nat.}{\longrightarrow} \cm_{\Z/2}(\ca).
\end{align}

\begin{remark}
\label{rem:concellation}
For any $a\in\ch(\cc_{\Z/2}(\ca))$, we have $\Upsilon(a)=0$ if and only if there exists $[K]\in\Iso(\cc_{\Z/2,ac}({\ca}))$  such that $\widehat{\Upsilon}(a)\diamond [K]=0$ in $\ch(\cc_{\Z/2}(\ca))/J_{\Z/2}$.
\end{remark}

\begin{lemma}
\label{remark coincide of the structure of two bimodule}
For any $M\in\cc_{\Z/2}(\ca)$, $K\in\cc_{\Z/2,ac}(\ca)$,  we have  
\begin{align}
\label{eq:bimoduleHall 1}
\widehat{\Upsilon}([M]\diamond[K])=[M]\diamond[K],\qquad
\widehat{\Upsilon} ([K]\diamond [M])\diamond[\Cone(1_{M^*})]=[K]\diamond [M]\diamond [\Cone(1_{M^*})]
\end{align}
in $\ch(\cc_{\Z/2}(\ca))/J_{\Z/2}$.

As a consequence, we have
\begin{align}
\label{eq:bimoduleHall 2}
\Upsilon([M]\diamond [K])= [M]\diamond [K],\qquad
\Upsilon([K]\diamond [M])=[K]\diamond [M]
\end{align}
in $\cm_{\Z/2}(\ca)$.
\end{lemma}

\begin{proof}
For $K\in\cc_{\Z/2,ac}(\ca)$ and $M\in\cc_{\Z/2}(\ca)$, we have
\begin{eqnarray*}
\widehat{\Upsilon}([M]\diamond[K])&=&\widehat{\Upsilon}\big(\sum_{[N]\in\Iso(\cc_{\Z/2}(\ca))} \frac{|\Ext^1_{\cc_{\Z/2}(\ca)}(M,K)_N|}{|\Hom_{\cc_{\Z/2}(\ca)}(M,K)|}[N]\big)\\
&=&\frac{1}{\langle [M],[K]\rangle}[M\oplus K]\\
&=& [M]\diamond[K],
\end{eqnarray*}
where the second equality follows from $\widehat{\Upsilon}([N])= [M\oplus K]$ in $\ch(\cc_{\Z/2}(\ca))/J_{\Z/2}$ for any $N$ such that $|\Ext^1_{\cc_{\Z/2}(\ca)}(M,K)_N|\neq0$.
So we get the first identity in \eqref{eq:bimoduleHall 1}.

To prove the second identity in \eqref{eq:bimoduleHall 1}, we note that $\Omega_{\Z/2}\subseteq J_{\Z/2}$. We use the same method as in the proof of \eqref{equation proposition localizaition of Ringel-Hall algebra 1} to show that
\begin{eqnarray*}
\widehat{\Upsilon}([K]\diamond [M]\diamond [\Cone(1_{M^*})])=\frac{\langle [M],[K]\rangle}{\langle [K],[M]\rangle}\widehat{\Upsilon}([M]\diamond [K]\diamond [\Cone(1_{M^*})]).
\end{eqnarray*}
In fact, for any $[N]\in\Iso(\cc_{\Z/2}(\ca))$ such that $|\Ext^1_{\cc_{\Z/2}(\ca)}(K,M)_N|\neq0$, by Lemma \ref{lem:existC}, we have
$$[N\oplus \Cone(1_{M^*})]-[\Cone(1_{M^*})\oplus K\oplus M]\in J_{\Z/2}.$$
So by the first identity in \eqref{eq:bimoduleHall 1}, we have
\begin{align*}
&\widehat{\Upsilon}([K]\diamond [M]\diamond [\Cone(1_{M^*})])
\\
=&\sum_{[N]\in\Iso(\cc_{\Z/2}(\ca))} \frac{|\Ext^1_{\cc_{\Z/2}(\ca)}(K,M)_N|}{|\Hom_{\cc_{\Z/2}(\ca)}(K,M)|}[N]\diamond [\Cone(1_{M^*})]
\\
=&\sum_{[N]\in\Iso(\cc_{\Z/2}(\ca))} \frac{|\Ext^1_{\cc_{\Z/2}(\ca)}(K,M)_N|}{|\Hom_{\cc_{\Z/2}(\ca)}(K,M)|}\frac{1}{\langle [N], [\Cone(1_{M^*})] \rangle} [N\oplus \Cone(1_{M^*})]
\\
=& \sum_{[N]\in\Iso(\cc_{\Z/2}(\ca))} \frac{|\Ext^1_{\cc_{\Z/2}(\ca)}(K,M)_N|}{|\Hom_{\cc_{\Z/2}(\ca)}(K,M)|}\frac{1}{\langle [K\oplus M], [\Cone(1_{M^*})] \rangle} [\Cone(1_{M^*})\oplus K\oplus M]
\\
=& \frac{1}{\langle [K],[M]\rangle \langle [K\oplus M], [\Cone(1_{M^*})] \rangle} [\Cone(1_{M^*})\oplus K\oplus M]
\\
=& \frac{ \langle [M], [K\oplus \Cone(1_{M^*})]\rangle}{\langle [K],[M]\rangle\langle [K\oplus M], [\Cone(1_{M^*})] \rangle} [M]\diamond [K\oplus \Cone(1_{M^*})]
\\
=&\frac{\langle [M],[K]\rangle}{\langle [K],[M]\rangle}\widehat{\Upsilon}([M]\diamond [K]\diamond [\Cone(1_{M^*})])
\end{align*}
in $\ch(\cc_{\Z/2}(\ca))/J_{\Z/2}$.
From this equality, we have
\begin{eqnarray*}
\widehat{\Upsilon}([K]\diamond [M])\diamond [\Cone(1_{M^*})]&=&\widehat{\Upsilon}([K]\diamond [M]\diamond [\Cone(1_{M^*})])\\
&=&\frac{\langle [M],[K]\rangle}{\langle [K],[M]\rangle}\widehat{\Upsilon}([M]\diamond [K]\diamond [\Cone(1_{M^*})])\\
&=&\frac{\langle [M],[K]\rangle}{\langle [K],[M]\rangle}\widehat{\Upsilon}([M]\diamond [K])\diamond [\Cone(1_{M^*})]\\
&=&\frac{1}{\langle [K],[M]\rangle}\widehat{\Upsilon}([M\oplus K])\diamond [\Cone(1_{M^*})]
\\
&=&[K]\diamond [M]\diamond [\Cone(1_{M^*})]
\end{eqnarray*}
in $\ch(\cc_{\Z/2}(\ca))/J_{\Z/2}$.

For the last statement, we only need to prove the second formula. From above, we have
$\widehat{\Upsilon}([K]\diamond [M])\diamond [\Cone(1_{M^*})]=[K]\diamond [M]\diamond [\Cone(1_{M^*})]$ in $\ch(\cc_{\Z/2}(\ca))/J_{\Z/2}$.
Since $[\Cone(1_{M^*})]$ is invertible in $\T_{\Z/2,ac}(\ca)$, it follows that
$$\Upsilon([K]\diamond [M])=[K]\diamond [M]$$
in
$\cm_{\Z/2}(\ca)$.
\end{proof}

\begin{remark}
From  \eqref{eq:bimoduleHall 2}, we obtain that the $\T_{\Z/2,ac}(\ca)$-bimodule structure  of
$\cm_{\Z/2}(\ca)$ is induced by the Hall product, namely the multiplication of the Ringel-Hall algebra $\ch(\cc_{\Z/2}(\ca))$. However, although the right $\A_{\Z/2,ac}(\ca)$-module structure on $\ch(\cc_{\Z/2}(\ca))/J_{\Z/2}$ is still induced by the Hall product, the left  $\A_{\Z/2,ac}(\ca)$-module structure on  $\ch(\cc_{\Z/2}(\ca))/J_{\Z/2}$ is not. It is not true that $\widehat{\Upsilon} ([K]\diamond [M])=[K]\diamond [M]$ for all $M\in\cc_{\Z/2}(\ca)$, $K\in\cc_{\Z/2,ac}(\ca)$, i.e., we cannot remove $[\Cone(1_{M^*})]$ form the second identity in \eqref{eq:bimoduleHall 1}.
\end{remark}

\subsection{A basis of $\cm_{\Z/2}(\ca)$}
\label{subsec:basis}

\begin{proposition}
\label{lemma canonical form of complex}
Let $M=(\xymatrix{ M^0 \ar@<0.5ex>[r]^{f^0}& M^1 \ar@<0.5ex>[l]^{f^1}  })$ and $H^0(M),H^1(M)$ be the homology groups. In  $\cm_{\Z/2}(\ca)$, we have
\begin{align*}
[M]=&\langle \widehat{\Im f^0} , \widehat{\Im f^1}\rangle \langle\widehat{\Im f^0}, \widehat{H^0(M)} \rangle \langle\widehat{ \Im f^1}, \widehat{H^1(M) }\rangle[K_{\Im f^0}]\diamond [K_{\Im f^1}^*]\diamond [C_{H^0(M)}^*\oplus C_{H^1(M)}].
\end{align*}
\end{proposition}
\begin{proof}
From Lemma \ref{Galois functor is dense for hereditary categories} and its proof,
there is a short exact sequence
$$0\longrightarrow K_{\ker f^0} \longrightarrow \pi(U)\longrightarrow M\longrightarrow0,$$
such that
$$U=\dots \longrightarrow0\longrightarrow M^0\stackrel{l^0}{\longrightarrow}Z^0\longrightarrow \ker f^0\longrightarrow0\longrightarrow\cdots$$
with $l^0$ injective.
Then
\begin{equation}\label{equation 1 in Lemma canonical form of complex}
[K_{\ker f^0}]\diamond [M]=\frac{1}{\langle [K_{\ker f^0}], [M]\rangle}[\pi(U)].
\end{equation}
Note that $l^0$ is injective. So there is a short exact sequence
$$0\longrightarrow K_{M^0}\longrightarrow \pi(U)\longrightarrow V:=(\xymatrix{ \ker f^0 \ar@<0.5ex>[r]^{0\quad}& M^1/\Im f^0 \ar@<0.5ex>[l]^{g^0\quad}  })\longrightarrow0,$$
where $g^0$ is induced by $f^1$.
Then
\begin{equation}\label{equation 2 in Lemma canonical form of complex}
[K_{M^0}]\diamond [V]=\frac{1}{\langle [K_{M^0} ],[V]\rangle} [\pi(U)].
\end{equation}

By applying $\Hom_\ca(H^0(M),-)$ to the short exact sequence $$0\longrightarrow \ker g^0\longrightarrow M^1/\Im f^0\longrightarrow \Im g^0\longrightarrow0,$$
we obtain the following commutative diagram:
\[\xymatrix{ \ker g^0 \ar[r] \ar@{=}[d]& M^1/\Im f^0 \ar[r] \ar@{.>}[d] & \Im g^0\cong \Im f^1 \ar[d] \\
\ker g^0\ar@{.>}[r]& X^0 \ar@{.>}[r] \ar@{.>}[d] & \ker f^0 \ar[d]\\
& H^0(M)\ar@{=}[r]  &H^0(M) }\]
So there is a short exact sequence
$$0\longrightarrow K_{\ker g^0}^* \longrightarrow W:=(\xymatrix{ X^0  \ar@<0.5ex>[r]^{0\quad\quad\quad}& \ker g^0\oplus M^1/\Im f^0 \ar@<0.5ex>[l]  })  \longrightarrow V\longrightarrow0,$$
and then
\begin{equation}\label{equation 3 in Lemma canonical form of complex}
[K^*_{\ker g^0}]\diamond [V] =\frac{1}{\langle [K^*_{\ker g^0}], [V]\rangle} [W].
\end{equation}

There is also a short exact sequence
$$0\longrightarrow K_{M^1/\Im f^0}^* \longrightarrow W \longrightarrow [C_{H^0(M)}^*\oplus C_{H^1(M)}] \longrightarrow0,$$
and then
\begin{equation}\label{equation 4 in Lemma canonical form of complex}
[K^*_{M^1/ \Im f^0}]\diamond [C_{H^0(M)}^*\oplus C_{H^1(M)}]  =\frac{1}{\langle [K^*_{M^1/ \Im f^0}],[C_{H^0(M)}^*\oplus C_{H^1(M)} ]\rangle} [W].
\end{equation}

From (\ref{equation 1 in Lemma canonical form of complex})-(\ref{equation 4 in Lemma canonical form of complex}), we obtain
\begin{align*}
[M]=& \frac{1}{\langle [K_{\ker f^0}], [M]\rangle}[K_{\ker f^0}]^{-1}\diamond[\pi(U)]\\
=&  \frac{\langle [K_{M^0}] ,[V]\rangle}{\langle [K_{\ker f^0}], [M]\rangle}[K_{\ker f^0}]^{-1}\diamond [K_{M^0}]\diamond [V]\\
=& \frac{\langle [K_{M^0}] ,[V]\rangle \langle [K_{\ker f^0}],[ K_{\Im f^0} ]\rangle}{\langle [K_{\ker f^0}], [M]\rangle}[K_{\Im f^0}]\diamond [V]\\
=& \frac{\langle [K_{M^0}] ,[V]\rangle \langle [K_{\ker f^0}], [K_{\Im f^0}] \rangle}{\langle [K_{\ker f^0}], [M]\rangle   \langle [K_{\ker g^0}^*], [V]\rangle}[K_{\Im f^0}]\diamond [K_{\ker g^0}^*]^{-1} \diamond [W]
\\
=&  \frac{\langle [K_{M^0}] ,[V]\rangle \langle [K_{\ker f^0}], [K_{\Im f^0} ]\rangle\langle [K^*_{M^1/ \Im f^0}],[C_{H^0(M)}^*\oplus C_{H^1(M)} ]\rangle}{\langle [K_{\ker f^0}], [M]\rangle   \langle [K_{\ker g^0}^*], [V]\rangle}\\
&[K_{\Im f^0}]\diamond [K_{\ker g^0}^*]^{-1} \diamond [K^*_{M^1/ \Im f^0}]\diamond [C_{H^0(M)}^*\oplus C_{H^1(M)}]\\
=&\frac{\langle [K_{M^0}] ,[V]\rangle \langle [K_{\ker f^0}], [K_{\Im f^0}] \rangle\langle [K^*_{M^1/ \Im f^0}],[C_{H^0(M)}^*\oplus C_{H^1(M)} ]\rangle\langle [K^*_{\ker g^0}], [ K^*_{\Im f^1}]\rangle}{\langle [K_{\ker f^0}], [M]\rangle   \langle [K_{\ker g^0}^*], [V]\rangle}\\
&[K_{\Im f^0}]\diamond [K^*_{\Im f^1}]\diamond [C_{H^0(M)}^*\oplus C_{H^1(M)}].
\end{align*}
Note that
\begin{eqnarray*}
&&\frac{\langle [K_{M^0}] ,[V]\rangle \langle [K_{\ker f^0}],[ K_{\Im f^0}] \rangle\langle[ K^*_{M^1/ \Im f^0}],[C_{H^0(M)}^*\oplus C_{H^1(M)}] \rangle\langle [K^*_{\ker g^0}],  [K^*_{\Im f^1}]\rangle}{\langle [K_{\ker f^0}], [M]\rangle   \langle [K_{\ker g^0}^*], [V]\rangle}\\
&=&\langle [K_{\Im f^0}],[V] \rangle \langle [K_{\Im f^1}^*], [C_{H^0(M)}^*\oplus C_{H^1(M)}]\rangle\\
&=& \langle [K_{\Im f^0}] , [K^*_{\Im f^1}]\rangle \langle [K_{\Im f^0}], [C_{H^0(M)}^*\oplus C_{H^1(M)} ]\rangle \langle [K_{\Im f^1}^*], [C_{H^0(M)}^*\oplus C_{H^1(M)}]\rangle.
\end{eqnarray*}
Then
\begin{eqnarray*}
[M]&=&\langle [K_{\Im f^0}] ,[ K^*_{\Im f^1}]\rangle \langle [K_{\Im f^0}], [C_{H^0(M)}^*\oplus C_{H^1(M)}] \rangle \langle [K_{\Im f^1}^*], [C_{H^0(M)}^*\oplus C_{H^1(M)} ]\rangle\\
&&[K_{\Im f^0}]\diamond [K_{\Im f^1}^*]\diamond [C_{H^0(M)}^*\oplus C_{H^1(M)}]\\
&=&\langle \widehat{\Im f^0} , \widehat{\Im f^1}\rangle \langle\widehat{\Im f^0}, \widehat{H^0(M)} \rangle \langle\widehat{ \Im f^1}, \widehat{H^1(M) }\rangle[K_{\Im f^0}]\diamond [K_{\Im f^1}^*]\diamond [C_{H^0(M)}^*\oplus C_{H^1(M)}],
\end{eqnarray*}
where the last equality follows from Proposition \ref{lema euler form} and Corollary \ref{lemma coincide of Euler forms}.
\end{proof}

We shall use the following general result later.

\begin{lemma}
\label{proposition equivalences of some exact sequences}
Let $\cb$ be an abelian category. For objects $$U=(\xymatrix{ U^0 \ar@<0.5ex>[r]^{k^0}& U^1 \ar@<0.5ex>[l]^{k^1}  }),V=(\xymatrix{ V^0 \ar@<0.5ex>[r]^{f^0}& V^1 \ar@<0.5ex>[l]^{f^1}  }),W=(\xymatrix{ W^0 \ar@<0.5ex>[r]^{e^0}& W^1 \ar@<0.5ex>[l]^{e^1}  })\in\cc_{\Z/2}(\cb),$$
if there is a short exact sequence $0\rightarrow U\xrightarrow{h_1} V\xrightarrow{h_2} W\rightarrow0$, then the following statements are equivalent.
\begin{itemize}
\item[(i).] $0\rightarrow \Im k^i\xrightarrow{t_1^i} \Im f^i\xrightarrow{t_2^i} \Im e^i\rightarrow 0$ is  exact for  $i=0,1$.

\item[(ii).] $0\rightarrow \ker k^i\xrightarrow{s_1^i} \ker f^i\xrightarrow{s_2^i} \ker e^i\rightarrow 0$ is  exact for  $i=0,1$.

\item[(iii).] $0\rightarrow H^i(U)\xrightarrow{r_1^i} H^i(V)\xrightarrow{r_2^i} H^i(W)\rightarrow 0$ is  exact for  $i=0,1$.
\end{itemize}
Here the morphisms are induced by $h_1$ and $h_2$.
In particular, if $U$ or $W$ is acyclic,
then
$\widehat{\Im f^0}=\widehat{\Im k^0}+ \widehat{\Im e^0}$ and $\widehat{\Im f^1}=\widehat{\Im k^1}+ \widehat{\Im e^1}$ in $K_0(\cb)$.
\end{lemma}

\begin{proof}
The analogous result for bounded complexes is known (see e.g.  \cite[Lemma 9.2.2]{EJ}), and its proof can be applied in the $\Z/2$-graded case. We give a proof here for convenience.

One can easily get the following commutative diagram:
\[\xymatrix{ 0\ar[r] & \ker k^i \ar[r]^{s_1^i}\ar@{>->}[d] & \ker f^i \ar[r]^{s_2^i}\ar@{>->}[d] & \ker e^i\ar@{>->}[d]&\\
 0\ar[r]& U^i\ar[r]\ar@{->>}[d] &V^i \ar[r]\ar@{->>}[d] & W^i\ar@{->>}[d]\ar[r]&0\\
 &\Im k^i\ar[r]^{t_1^i} &\Im f^i \ar[r]^{t_2^i}& \Im e^i\ar[r]&0}\]
with all rows and columns exact. Then the equivalence of (i) and (ii) follows from the Snake Lemma. 

We also have the following commutative diagram:
\[\xymatrix{
&\Im k^i\ar[r]^{t_1^i} \ar@{>->}[d] & \Im f^i \ar[r]^{t_2^i} \ar@{>->}[d] & \Im e^i \ar@{>->}[d] \ar[r]&0\\
0\ar[r]& \ker k^{i+1} \ar[r]^{s_1^i} \ar@{->>}[d] & \ker f^{i+1}\ar[r]^{s_2^i} \ar@{->>}[d] & \ker e^{i+1}\ar@{->>}[d] &\\
&H^{i+1}(U)\ar[r]^{r_1^i} & H^{i+1}(V)\ar[r]^{r_2^i} &H^{i+1}(W)& }\]
with all rows and columns exact.

(iii)$\Rightarrow$(ii): Since $r_2^i$ and $t_2^i$ are surjective, we obtain that $s_2^i$ is also surjective, and then
$0\rightarrow \ker k^i\rightarrow \ker f^i\rightarrow \ker e^i\rightarrow 0$ is short exact for $i=0,1$.

(ii)$\Rightarrow$(iii): Since $0\rightarrow \ker k^i\rightarrow \ker f^i\rightarrow \ker e^i\rightarrow 0$ is short exact for $i=0,1$, from the equivalence of (i) and (ii), we obtain that
$0\rightarrow \Im k^i\rightarrow \Im f^i\rightarrow \Im e^i\rightarrow 0$ is short exact for $i=0,1$. So $0\rightarrow H^i(U)\xrightarrow{r_1^i} H^i(V)\xrightarrow{r_2^i} H^i(W)\rightarrow 0$ is short exact for $i=0,1$.

In particular, if $U$ or $W$ is acyclic, then (iii) holds, which implies (i). So $\widehat{\Im f^i}=\widehat{\Im k^i}+ \widehat{\Im e^i}$ for $i=0,1$.
\end{proof}

Now we go back to the category $\ca$.
For any $A,B\in\ca$ with $\widehat{A}=\widehat{B}$, it is easy to see that $[K_A]=[K_B]$, $[K_A^*]=[K_B^*]$ in $\cm_{\Z/2}(\ca)$.
For any $\alpha\in K_0(\ca)$, there exist $A,B\in\ca$ such that $\alpha=\widehat{A}-\widehat{B}$, and we set
\begin{align*}
K_\alpha:=\frac{1}{\langle \alpha,\widehat{B}\rangle}[K_A]\diamond [K_B]^{-1}.
\end{align*}
Then $K_\alpha$ is well defined in  $\cm_{\Z/2}(\ca)$.
In fact, we only need to check that $K_\alpha$ does not depend on the choice of $[A],[B]\in\Iso(\ca)$ for $\alpha\in K_0(\ca)$.
If $\alpha=\widehat{A}-\widehat{B}=\widehat{A'}-\widehat{B'}$ for $A,B,A',B'\in\ca$,
then
\begin{align*}
&\big(\frac{1}{\langle \alpha,\widehat{B}\rangle}[K_A]\diamond [K_B]^{-1}-\frac{1}{\langle \alpha,\widehat{B'}\rangle}[K_{A'}]\diamond [K_{B'}]^{-1} \big)\diamond [K_B]\diamond [K_{B'}]\\
=&\frac{1}{\langle \alpha,\widehat{B}\rangle}[K_A]\diamond [K_{B'}]-\frac{1}{\langle \alpha,\widehat{B'}\rangle}\frac{\langle \widehat{B'},\widehat{B}\rangle}{\langle\widehat{B},\widehat{B'}\rangle}
[K_{A'}]\diamond [K_{B}]\\
=&\frac{1}{\langle \alpha,\widehat{B}\rangle}\frac{1}{\langle \widehat{A},\widehat{B'}\rangle}[K_A\oplus K_{B'}]-
\frac{1}{\langle \alpha,\widehat{B'}\rangle}\frac{\langle \widehat{B'},\widehat{B}\rangle}{\langle\widehat{B},\widehat{B'}\rangle}\frac{1}{\langle \widehat{A'},\widehat{B}\rangle} [K_{A'}\oplus K_B]\\
=&\frac{1}{\langle \alpha,\widehat{B}\rangle}\frac{1}{\langle \widehat{A},\widehat{B'}\rangle}[K_A\oplus K_{B'}]-\frac{1}{\langle \widehat{A},\widehat{B'}\rangle}\frac{1}{\langle \alpha,\widehat{B}\rangle}[K_{A'}\oplus K_B].
\end{align*}
Since $\widehat{A'}+\widehat{B}=\widehat{A}+\widehat{B'}$ in $K_0(\ca)$, we obtain that $[K_{A}\oplus K_{B'}]=[K_{A'}\oplus K_B]$ in  $\cm_{\Z/2}(\ca)$. So
\begin{align*}
\big(\frac{1}{\langle \alpha,\widehat{B}\rangle}[K_A]\diamond [K_B]^{-1}-\frac{1}{\langle \alpha,\widehat{B'}\rangle}[K_{A'}]\diamond [K_{B'}]^{-1} \big)\diamond [K_B]\diamond [K_{B'}]=0,
\end{align*}
and then
$$\frac{1}{\langle \alpha,\widehat{B}\rangle}[K_A]\diamond [K_B]^{-1}-\frac{1}{\langle \alpha,\widehat{B'}\rangle}[K_{A'}]\diamond [K_{B'}]^{-1}=0$$
in $\cm_{\Z/2}(\ca)$. So $K_\alpha$ is well defined.

Similarly, for any $\alpha\in K_0(\ca)$, we set
$K_\alpha^*:=\frac{1}{\langle \alpha,\widehat{B}\rangle}[K_A^*]\diamond [K_B^*]^{-1}$, where
$A,B\in\ca$ satisfy $\alpha=\widehat{A}-\widehat{B}$.


We set $\G$ to be the following set
$$\G:=\{(\alpha,\beta,[A],[B])\mid \alpha,\beta\in K_0(\ca), [A],[B]\in \Iso(\ca)\}.$$
For any $[A],[B]\in \Iso(\ca)$, we define $\G(A,B):=\{ (\alpha,\beta,[A],[B])\mid \alpha,\beta\in K_0(\ca)\}$. Then $\G(A,B)$ is an abelian group with its group structure induced by that of $K_0(\ca)$. It is natural to identify $\G(0,0)$ with the group $K_0(\ca)\times K_0(\ca)$.
Then
$\ch(\cc_{\Z/2}(\ca))$ is a $\G$-graded vector space, i.e.,
$$\ch(\cc_{\Z/2}(\ca))=\bigoplus_{(\alpha,\beta,[A],[B])\in \G} \Big(\bigoplus_{\stackrel{ \widehat{ \Im f^0 }=\alpha, \widehat{\Im f^1}=\beta,}{
H^0(M)\simeq A,H^1(M)\simeq B}} \Q[M]\Big),$$
where
$M= (\xymatrix{ M^0 \ar@<0.5ex>[r]^{f^0}& M^1 \ar@<0.5ex>[l]^{f^1}  })$.
In particular,
$$\ch(\cc_{\Z/2,ac}(\ca))=\bigoplus_{(\alpha,\beta)\in \G(0,0)}  \Big(\bigoplus_{ \widehat{ \Im d^0 }=\alpha, \widehat{\Im d^1}=\beta} \Q[K]\Big),$$
where $K= (\xymatrix{ K^0 \ar@<0.5ex>[r]^{d^0}& K^1 \ar@<0.5ex>[l]^{d^1}  })$ is acyclic.

\begin{lemma}
\label{lemma grade of Ringel-Hall algebra}
We have the following.
\begin{itemize}
\item[(i).] $\ch(\cc_{\Z/2,ac}(\ca))$ is a $\G(0,0)$-graded algebra.

\item[(ii).] Let $h\in \ch(\cc_{\Z/2}(\ca))$, $v\in\ch(\cc_{\Z/2,ac}(\ca))$ be homogeneous elements of degrees $(\alpha,\beta,[A],[B])\in \G$ and $(\gamma_1,\gamma_2)\in \G(0,0)$ respectively. Then $h\diamond v$ is homogeneous of degree $(\alpha+\gamma_1,\beta+\gamma_2,[A],[B])$.

\item[(iii).] $\cm_{\Z/2}(\ca)$ is a $\G$-graded vector space.
\end{itemize}
\end{lemma}
\begin{proof}
First, we prove (ii). Without loss of generality, we can assume $h=[M]\in \Iso(\cc_{\Z/2}(\ca))$, $v=[K]\in\Iso(\cc_{\Z/2,ac}(\ca))$.

For $ M=(\xymatrix{ M^0 \ar@<0.5ex>[r]^{e^0}& M^1 \ar@<0.5ex>[l]^{e^1}  })\in\cc_{\Z/2}(\ca)$ and $K=(\xymatrix{ K^0 \ar@<0.5ex>[r]^{f^0}& K^1 \ar@<0.5ex>[l]^{f^1}  }) \in \cc_{\Z/2,ac}(\ca)$, the degree of $[M]$ is $(\widehat{\Im e^0},\widehat{\Im e^1}, [H^0(M)],[H^1(M)])$ and the degree of $K$ is $(\widehat{\Im f^0},\widehat{\Im f^1})$. We have
\begin{equation*}
[M]\diamond[K]=\sum_{[N]\in\Iso(\cc_{\Z/2,ac}(\ca))} \frac{|\Ext^1_{\cc_{\Z/2}(\ca)}(M,K)_N|}{|\Hom_{\cc_{\Z/2}(\ca)}(M,K)|}[N],
\end{equation*}
where $N= (\xymatrix{ N^0 \ar@<0.5ex>[r]^{g_N^0}& N^1 \ar@<0.5ex>[l]^{g_N^1}})$.
If $|\Ext^1_{\cc_{\Z/2}(\ca)}(M,K)_N|\neq0$, then there exists a short exact sequence
$0\rightarrow K\rightarrow N\rightarrow M\rightarrow0$.
Lemma \ref{proposition equivalences of some exact sequences} shows that
$$\widehat{\Im g_N^0}=\widehat{\Im e^0} +\widehat{\Im f^0}, \qquad\widehat{\Im g_N^1}=\widehat{\Im e^1} +\widehat{\Im f^1},$$
which are the same for any $N$ such that $|\Ext^1_{\cc_{\Z/2}(\ca)}(M,K)_N|\neq0$. Thus, $[M]\diamond [K]$ is $\G$-homogeneous of degree $( \widehat{\Im e^0}+\widehat{\Im f^0},\widehat{\Im e^1}+\widehat{\Im f^1}, [H^0(M)],[H^1(M)])$.

(i) follows from (ii) as it is a special case of (ii).

(iii).
From Lemma \ref{proposition equivalences of some exact sequences}, for any short exact sequence $0\rightarrow K\rightarrow M\rightarrow N\rightarrow 0$ with $K$ acyclic, we obtain that $[M]-[K\oplus N]$ is $\G$-homogeneous, which implies that $J_{\Z/2}$ is a homogeneous subspace, and then
$\ch(\cc_{\Z/2}(\ca))/J_{\Z/2}$ is a $\G$-graded vector space. It follows that $\cm_{\Z/2}(\ca)$ is $\G$-graded.
\end{proof}

\begin{proposition}
\label{theorem basis of modified hall module}
$\cm_{\Z/2}(\ca)$ has a basis given by
\begin{align}
\label{basis}
\{[C_A\oplus C_B^*]\diamond K_\alpha\diamond K_\beta^*\mid  [A],[B]\in\Iso(\ca),\alpha,\beta\in K_0(\ca)\}.
\end{align}
\end{proposition}

\begin{proof}
Lemma \ref{lemma grade of Ringel-Hall algebra} shows that $\cm_{\Z/2}(\ca)$ is a $\G$-graded vector space, and the degree of $[C_A\oplus C_B^*]\diamond K_\alpha\diamond K_\beta^*$ is
$(\alpha,\beta,[A],[B])\in \G$, for any $\alpha,\beta\in K_0(\ca)$ and $[A],[B]\in \Iso(\ca)$.

We assume that
$$\sum_{\stackrel{\alpha,\beta\in K_0(\ca)}{ [A],[B]\in\Iso(\ca)}} a_{\alpha,\beta,[A],[B]} [C_A\oplus C_B^*]\diamond K_\alpha\diamond K_\beta^*=0$$
in $\cm_{\Z/2}(\ca)$,
where $a_{\alpha,\beta,[A],[B]} \in\Q$. Then $a_{\alpha,\beta,[A],[B]} [C_A\oplus C_B^*]\diamond K_\alpha\diamond K_\beta^*=0$ for any $\alpha,\beta\in K_0(\ca)$ and $[A],[B]\in \Iso(\ca)$, since
$\cm_{\Z/2}(\ca)$ is a $\G$-graded vector space.

For any $\alpha,\beta\in K_0(\ca)$, we choose the objects $A_1^\alpha,A_2^\alpha,B_1^\beta,B_2^\beta\in\ca$ 
such that
$\alpha=\widehat{A_1^\alpha}-\widehat{A_2^\alpha}$ and $\beta=\widehat{B_1^\beta}-\widehat{B_2^\beta}$.
By definition,
\begin{align*}K_\alpha=\frac{1}{\langle \alpha, \widehat{A_2^\alpha}\rangle}[K_{A_1^\alpha}]\diamond [K_{A_2^\alpha}]^{-1},\quad  K_\beta^*=\frac{1}{\langle \beta, \widehat{B_2^\beta}\rangle}[K_{B_1^\beta}^*]\diamond [K_{B_2^\beta}^*]^{-1}.
\end{align*}
So we compute in $\cm_{\Z/2}(\ca)$ that
\begin{eqnarray*}
&& a_{\alpha,\beta,[A],[B]} [C_A\oplus C_B^*]\diamond K_\alpha\diamond K_\beta^*\diamond [K_{A_2^\alpha}\oplus K_{B_2^\beta}^* ]\\
&=&a_{\alpha,\beta,[A],[B]} \frac{\langle  \widehat{A_2^\alpha},\widehat{B_2^\beta} \rangle }{\langle \alpha, \widehat{A_2^\alpha}\rangle\langle \beta, \widehat{B_2^\beta}\rangle}[C_A\oplus C_B^*]\diamond[K_{A_1^\alpha}]\diamond [K_{A_2^\alpha}]^{-1}\diamond [K_{B_1^\beta}^*]\diamond [K_{B_2^\beta}^*]^{-1}\diamond [K_{A_2^\alpha}]\diamond [K_{B_2^\beta}^* ]\\
&=&a_{\alpha,\beta,[A],[B]} \frac{\langle  \widehat{A_2^\alpha},\widehat{B_2^\beta} \rangle }{\langle \alpha, \widehat{A_2^\alpha}\rangle\langle \beta, \widehat{B_2^\beta}\rangle}\frac{\langle\widehat{A_2^\alpha},\beta \rangle}{\langle\beta,\widehat{A_2^\alpha}\rangle}[C_A\oplus C_B^*]\diamond[K_{A_1^\alpha}]\diamond [K_{B_1^\beta}^*].
\end{eqnarray*}
Set
\begin{align*}b_{\alpha,\beta}:=\frac{\langle  \widehat{A_2^\alpha},\widehat{B_2^\beta} \rangle}{\langle \alpha, \widehat{A_2^\alpha}\rangle\langle \beta, \widehat{B_2^\beta}\rangle } \frac{\langle\widehat{A_2^\alpha},\beta \rangle}{\langle\beta,\widehat{A_2^\alpha}\rangle}.
\end{align*}
Then we have
\begin{equation*}
a_{\alpha,\beta,[A],[B]}b_{\alpha,\beta}
[C_A\oplus C_B^*]\diamond [K_{A_1^\alpha}]\diamond
[K_{B_1^\beta}^*]=0
\end{equation*}
in $\cm_{\Z/2}(\ca)$.
The definition of $\cm_{\Z/2}(\ca)$ implies that there exists a complex $U\in \cc_{\Z/2,ac}(\ca)$ such that
\begin{align*}
a_{\alpha,\beta,[A],[B]}b_{\alpha,\beta}
[C_A\oplus C_B^*]\diamond [K_{A_1^\alpha}]\diamond
[K_{B_1^\beta}^*]\diamond[U]
=0,
\end{align*}
 and then
\begin{align*}
&a_{\alpha,\beta,[A],[B]}b_{\alpha,\beta}c_{\alpha,\beta,[A],[B]} [C_A\oplus C_B^*\oplus K_{A_1^\alpha}\oplus K_{B_1^\beta}^*\oplus U]
=0
\end{align*}
in $\ch(\cc_{\Z/2}(\ca))/J_{\Z/2}$, where
$$c_{\alpha,\beta,[A],[B]}= \frac{1}{\langle [C_A\oplus C_B^*],  [K_{A_1^\alpha}\oplus K_{B_1^\beta}^*\oplus U]\rangle\langle[K_{A_1^\alpha}],[K_{B_1^\beta}^*\oplus U] \rangle \langle
[K_{B_1^\beta}^*],[U]\rangle}.$$
Therefore,
\begin{eqnarray*}
a_{\alpha,\beta,[A],[B]}b_{\alpha,\beta}c_{\alpha,\beta,[A],[B]} [C_A\oplus C_B^*\oplus K_{A_1^\alpha}\oplus K_{B_1^\beta}^*\oplus U]
\in J_{\Z/2}
\end{eqnarray*}
in $\ch(\cc_{\Z/2}(\ca))$.
By definition of $ J_{\Z/2}$, we have the following equality in $\ch(\cc_{\Z/2}(\ca))$:
\begin{equation}
\label{eq:Mbasis}
a_{\alpha,\beta,[A],[B]}b_{\alpha,\beta}c_{\alpha,\beta,[A],[B]} [C_A\oplus C_B^*\oplus K_{A_1^\alpha}\oplus K_{B_1^\beta}^*\oplus U]=\sum_{i=1}^n c_i([M_i]-[K_i\oplus N_i]),
\end{equation}
where $c_i\in\Q$ and $K_i\in \cc_{\Z/2,ac}(\ca),M_i,N_i\in \cc_{\Z/2}(\ca)$ satisfy that there exist short exact sequences
$$0\longrightarrow K_i\longrightarrow M_i\longrightarrow N_i\longrightarrow0$$
for $1\leq i\leq n$.
Since the isomorphism classes form a $\Q$-basis of $\ch(\cc_{\Z/2}(\ca))$, there exists a $\Q$-linear map $\ch(\cc_{\Z/2}(\ca))\rightarrow \Q$, which sends each isomorphism class $[M]$ to $1$. Under this linear map, 
we have
\begin{eqnarray*}
a_{\alpha,\beta,[A],[B]}b_{\alpha,\beta}c_{\alpha,\beta,[A],[B]}=\sum_{i=1}^n c_i(1-1)=0.
\end{eqnarray*}
Note that
$b_{\alpha,\beta} >0$ and $ c_{\alpha,\beta,[A],[B]} >0$. Thus, $a_{\alpha,\beta,[A],[B]}=0$.

Therefore, \eqref{basis} is a linearly independent subset of $\cm_{\Z/2}(\ca)$.

On the other hand, from Proposition \ref{lemma canonical form of complex} and \eqref{definition of bimodule}, we obtain that
\eqref{basis}
spans the whole space,
and hence it is a basis of $\cm_{\Z/2}(\ca)$.
\end{proof}

\subsection{An isomorphism and a basis of $\cs\cd\ch(\ca)$}
\label{subsec:isomodules}

In this subsection, we prove that $\cs\cd\ch(\ca)$ and $\cm_{\Z/2}(\ca)$ are naturally isomorphic as $\mathbb{Q}$-linear spaces.
 Together with Proposition \ref{theorem basis of modified hall module}, one can get a basis of $\cs\cd\ch(\ca)$.

Recall
\begin{align*}
\Omega_{\Z/2}= \{[L]-&[K\oplus M]\mid\exists \text{ a short exact sequence }
\\
\notag
&0\longrightarrow K \longrightarrow L\longrightarrow M \longrightarrow 0\mbox{ with }K\mbox{ acyclic} \}
\end{align*}
defined in \eqref{set}, 
the $\mathbb{Q}$-linear subspace $J_{\Z/2}$ of $\ch(\cc_{\Z/2}(\ca))$ spanned by $\Omega_{\Z/2}$, and 
the two-sided ideal $I_{\Z/2}$ of $\ch(\cc_{\Z/2}(\ca))$ generated by $\Omega_{\Z/2}$.
Clearly $J_{\Z/2}\subset I_{\Z/2}$.

Note that $\cc_{\Z/2,ac}(\ca)$ is an extension-closed subcategory of $\cc_{\Z/2}(\ca)$ and so an exact category.
Let $\ch(\cc_{\Z/2,ac}(\ca))$ be the Ringel-Hall algebra of $\cc_{\Z/2,ac}(\ca)$. Let $I_{\Z/2,ac}$ be the two-sided ideal of $\ch(\cc_{\Z/2,ac}(\ca))$ generated by the following set
\begin{align*}
\{[K_2]-&[K_1\oplus K_3]\mid\exists \text{ a short exact sequence }
\\
\notag
&0\longrightarrow K_1 \longrightarrow K_2\longrightarrow K_3 \longrightarrow 0\mbox{ with }K_1, K_2, K_3\mbox{ acyclic} \}.
\end{align*}
 We consider the quotient algebra $\ch(\cc_{\Z/2,ac}(\ca))/I_{\Z/2,ac}$.
By \eqref{eq:prod KK},
one can get an algebra isomorphism
 \begin{align}
\ch(\cc_{\Z/2,ac}(\ca))/I_{\Z/2,ac}\cong \A_{\Z/2,ac}(\ca),
\end{align}
and we identify them in the following.
Then the natural embedding $\ch(\cc_{\Z/2,ac}(\ca))\rightarrow \ch(\cc_{\Z/2}(\ca))$ induces the following morphisms of algebras
\begin{align*}
\A_{\Z/2,ac}(\ca)\longrightarrow \ch(\cc_{\Z/2}(\ca))/I_{\Z/2},\qquad
\T_{\Z/2,ac}(\ca)\longrightarrow \cs\cd\ch(\ca).
\end{align*}
So 
$\cs\cd\ch(\ca)$ is a $\T_{\Z/2,ac}(\ca)$-bimodule (also an $\A_{\Z/2,ac}(\ca)$-bimodule), with its bimodule structure induced by the Hall product; see \eqref{eq:KMMK}. 

Since $J_{\Z/2}\subset I_{\Z/2}$, there is a natural linear map $\Psi: \ch(\cc_{\Z/2}(\ca))/J_{\Z/2}\rightarrow \ch(\cc_{\Z/2}(\ca))/I_{\Z/2}$.
Composed with the natural morphism $\ch(\cc_{\Z/2}(\ca))/I_{\Z/2}\rightarrow \cs\cd\ch(\ca)$,  $\Psi$
induces a linear map $\bar{\Psi}:\ch(\cc_{\Z/2}(\ca))/J_{\Z/2} \rightarrow\cs\cd\ch(\ca)$. By \eqref{definition of bimodule} and \eqref{eq:KMMK}, $\bar{\Psi}$ is a morphism of $\A_{\Z/2,ac}(\ca)$-bimodules. Then $\bar{\Psi}$ induces a morphism of $\T_{\Z/2,ac}(\ca)$-bimodules
\begin{align}
\label{eq:morPsi}
\widetilde{\Psi}:\cm_{\Z/2}(\ca)\longrightarrow  \cs\cd\ch(\ca),
\end{align}
which is the following composition:
\[\xymatrix{\cm_{\Z/2}(\ca)
\ar[rrr]^{\tiny\T_{\Z/2,ac}(\ca)\otimes \bar{\Psi}\otimes \T_{\Z/2,ac}(\ca)\qquad\qquad\qquad\qquad\qquad\qquad\,\,\,\,\,\,}&&&\T_{\Z/2,ac}(\ca)\otimes_{\A_{\Z/2,ac}(\ca)} \cs\cd\ch(\ca)\otimes_{\A_{\Z/2,ac}(\ca)} \T_{\Z/2,ac}(\ca)\ar[r]^{\qquad\qquad\qquad\qquad\quad\rm mult.} &\cs\cd\ch(\ca)
 }\]
by \eqref{def:bimodule}. 
In fact,
$\widetilde{\Psi}$ maps
$s_1^{-1}\otimes a\otimes s_2^{-1}$ to $s_1^{-1}as_2^{-1}$ for any $s_1^{-1}\otimes a\otimes s_2^{-1}\in \cm_{\Z/2}(\ca)$.

In the following, we shall prove that $\widetilde{\Psi}$ is in fact a linear bijection. For this purpose, we will then construct its inverse linear map.

As before, let $\Upsilon: \ch(\cc_{\Z/2}(\ca))\rightarrow \cm_{\Z/2}(\ca)$ be the composition
\begin{align}
\label{eq:morUpsilon}
\ch(\cc_{\Z/2}(\ca))\stackrel{\widehat{\Upsilon}}{\longrightarrow} \ch(\cc_{\Z/2}(\ca))/J_{\Z/2} \stackrel{\rm nat.}{\longrightarrow} \cm_{\Z/2}(\ca),
\end{align}
where $\widehat{\Upsilon}: \ch(\cc_{\Z/2}(\ca))\rightarrow \ch(\cc_{\Z/2}(\ca))/J_{\Z/2}$ is the natural projection.

First, we shall prove that $\Upsilon(I_{\Z/2})=0$ and so we can construct the inverse linear map of $\widetilde{\Psi}$. It is worth clarifying that $\widehat{\Upsilon}(I_{\Z/2})\neq 0$ in  $\ch(\cc_{\Z/2}(\ca))/J_{\Z/2}$.

\begin{lemma}\label{lemma left ideal}
Let $0\rightarrow K\xrightarrow{h_1} M\xrightarrow{h_2} N\rightarrow 0$ be a short exact sequence with $K$ acyclic. For any $L$, we have
$$ \Upsilon\big([L]\diamond([M]-[K\oplus N] )\big)=0.$$
\end{lemma}
\begin{proof}
Proposition \ref{lemma canonical form of complex} yields that
\begin{eqnarray*}
\Upsilon\big([L]\diamond[M]\big)&=&\sum_{[V]\in \Iso(\cc_{\Z/2}(\ca))}\frac{|\Ext^1_{\cc_{\Z/2}(\ca)}(L,M)_V|}{|\Hom_{\cc_{\Z/2}(\ca)}(L,M)|} \Upsilon([V])\\
&=&\sum_{\tiny\begin{array}{cc}[A],[B]\in \Iso(\ca),\\ \alpha,\beta\in K_0(\ca)\end{array}}\sum_{\tiny\begin{array}{cc} [V],V= (\xymatrix{ V^0  \ar@<0.5ex>[r]^{f^0}& V^1 \ar@<0.5ex>[l]^{f^1}  })\\
H^0(V)\simeq A,H^1(V)\simeq B\\
\widehat{\Im f^0}=\alpha,\widehat{\Im f^1}=\beta \end{array}} \langle \alpha , \beta\rangle \langle \alpha, \widehat{A} \rangle \langle \beta, \widehat{B}\rangle  \frac{|\Ext^1_{\cc_{\Z/2}(\ca)}(L,M)_V|}{|\Hom_{\cc_{\Z/2}(\ca)}(L,M)|}\\
&& [K_{\alpha}]\diamond [K_{\beta}^*]\diamond [C_{A}^*\oplus C_{B}].
\end{eqnarray*}
Similarly,
\begin{eqnarray*}
&&\Upsilon\big([L]\diamond [K\oplus N]\big)
\\
&=&\sum_{[U]\in \Iso(\cc_{\Z/2}(\ca))}\frac{|\Ext^1_{\cc_{\Z/2}(\ca)}(L,K\oplus N)_U|}{|\Hom_{\cc_{\Z/2}(\ca)}(L,K\oplus N)|} \Upsilon([U])\\
&=&\sum_{\tiny\begin{array}{cc}[A],[B]\in \Iso(\ca),\\ \alpha,\beta\in K_0(\ca)\end{array}}\sum_{\tiny\begin{array}{cc} [U],U= (\xymatrix{ U^0  \ar@<0.5ex>[r]^{d^0}& U^1 \ar@<0.5ex>[l]^{d^1}  })\\
H^0(U)\simeq A,H^1(U)\simeq B\\
\widehat{\Im d^0}=\alpha,\widehat{\Im d^1}=\beta \end{array}}  \langle \alpha , \beta\rangle \langle \alpha, \widehat{A} \rangle \langle \beta, \widehat{B}\rangle  \frac{|\Ext^1_{\cc_{\Z/2}(\ca)}(L,K\oplus N)_U|}{|\Hom_{\cc_{\Z/2}(\ca)}(L,K\oplus N)|} \\
&&  [K_{\alpha}]\diamond [K_{\beta}^*]\diamond [C_{A}^*\oplus C_{B}].
\end{eqnarray*}
It suffices to prove
\begin{equation}\label{equation left ideal 1}
\sum_{\tiny\begin{array}{cc} [V],V= (\xymatrix{ V^0  \ar@<0.5ex>[r]^{f^0}& V^1 \ar@<0.5ex>[l]^{f^1}  })\\
H^0(V)\simeq A,H^1(V)\simeq B\\
\widehat{\Im f^0}=\alpha,\widehat{\Im f^1}=\beta \end{array}}\frac{|\Ext^1_{\cc_{\Z/2}(\ca)}(L,M)_V|}{|\Hom_{\cc_{\Z/2}(\ca)}(L,M)|} =\sum_{\tiny\begin{array}{cc} [U],U= (\xymatrix{ U^0  \ar@<0.5ex>[r]^{d^0}& U^1 \ar@<0.5ex>[l]^{d^1}  })\\
H^0(U)\simeq A,H^1(U)\simeq B\\
\widehat{\Im d^0}=\alpha,\widehat{\Im d^1}=\beta \end{array}} \frac{|\Ext^1_{\cc_{\Z/2}(\ca)}(L,K\oplus N)_U|}{|\Hom_{\cc_{\Z/2}(\ca)}(L,K\oplus N)|}
\end{equation}
for any $[A],[B]\in \Iso(\ca)$ and $\alpha,\beta\in K_0(\ca)$.

By applying $\Hom_{\cc_{\Z/2}(\ca)}(L,-)$ to the short exact sequence $0\rightarrow K\xrightarrow{h_1} M\xrightarrow{h_2}N\rightarrow0$, we obtain a long exact sequence
\begin{eqnarray*}
&&0\longrightarrow \Hom_{\cc_{\Z/2}(\ca)}(L,K)\longrightarrow \Hom_{\cc_{\Z/2}(\ca)}(L,M)\longrightarrow \Hom_{\cc_{\Z/2}(\ca)}(L,N)\longrightarrow\Ext^1_{\cc_{\Z/2}(\ca)}(L,K)\\
&&\longrightarrow\Ext^1_{\cc_{\Z/2}(\ca)}(L,M)\xrightarrow{\varphi} \Ext^1_{\cc_{\Z/2}(\ca)}(L,N)\longrightarrow\Ext^2_{\cc_{\Z/2}(\ca)}(L,K)=0.
\end{eqnarray*}
So $\varphi$ is surjective.
More precisely, for any $[\xi]\in \Ext^1_{\cc_{\Z/2}(\ca)}(L,M)$, it is represented by $0\rightarrow M\stackrel{f_1}{\rightarrow} V\stackrel{f_2}{\rightarrow} L\rightarrow0$. Then $\varphi([\xi])$ is represented by
$0\rightarrow N \stackrel{g_1}{\longrightarrow} W\stackrel{g_2}{\longrightarrow} L\rightarrow 0$, which is constructed by the pushout:
\[\xymatrix{ K \ar[r]^{h_1} \ar@{=}[d] & M\ar[r]^{h_2} \ar[d]^{f_1} & N \ar[d]^{g_1} \\
K\ar[r] & V\ar[r]\ar[d]^{f_2} & W \ar[d]^{g_2}\\
&L\ar@{=}[r] & L}\]
Let $K= (\xymatrix{ K^0 \ar@<0.5ex>[r]^{k^0}& K^1 \ar@<0.5ex>[l]^{k^1}  })$, $V=(\xymatrix{ V^0  \ar@<0.5ex>[r]^{f^0}& V^1 \ar@<0.5ex>[l]^{f^1}  })$ and $W= (\xymatrix{ W^0  \ar@<0.5ex>[r]^{e^0}& W^1 \ar@<0.5ex>[l]^{e^1}  })$ be the objects in the second row of the above commutative diagram.
Then Lemma \ref{proposition equivalences of some exact sequences} shows that $H^i(V)\cong H^i(W)$  and $\widehat{\Im e^i}+\widehat{\Im k^i}=\widehat{\Im f^i}$ for
$i=0,1$.

For any $[A],[B]\in \Iso(\ca)$ and $\alpha,\beta\in K_0(\ca)$,
set $$\cv_{[A],[B],\alpha,\beta}:=\bigsqcup_ {\tiny\begin{array}{cc} [V],V= (\xymatrix{ V^0  \ar@<0.5ex>[r]^{f^0}& V^1 \ar@<0.5ex>[l]^{f^1}  })\\
H^0(V)\simeq A,H^1(V)\simeq B\\
\widehat{\Im f^0}=\alpha,\widehat{\Im f^1}=\beta \end{array}} \Ext^1_{\cc_{\Z/2}(\ca)}(L,M)_V$$
and
$$\cw_{[A],[B],\alpha,\beta}:=\bigsqcup_ {\tiny\begin{array}{cc} [W],W= (\xymatrix{ W^0  \ar@<0.5ex>[r]^{e^0}& W^1 \ar@<0.5ex>[l]^{e^1}  })\\
H^0(W)\simeq A,H^1(W)\simeq B\\
\widehat{\Im e^0}=\alpha-\widehat{\Im k^0},\widehat{\Im e^1}=\beta-\widehat{\Im k^1} \end{array}} \Ext^1_{\cc_{\Z/2}(\ca)}(L,N)_W.$$
Since $K$ is acyclic, from the above pushout diagram and the surjectivity of $\varphi$,  one can obtain that $\varphi^{-1}(\cw_{[A],[B],\alpha,\beta})=\cv_{[A],[B],\alpha,\beta}$.
From the long exact sequence $0\rightarrow \Hom_{\cc_{\Z/2}(\ca)}(L,K)\rightarrow \Hom_{\cc_{\Z/2}(\ca)}(L,M)\rightarrow \Hom_{\cc_{\Z/2}(\ca)}(L,N)\rightarrow\Ext^1_{\cc_{\Z/2}(\ca)}(L,K)\rightarrow \ker\varphi\rightarrow0$,
we have
\begin{equation}
|\ker\varphi|=\frac{|\Ext^1_{\cc_{\Z/2}(\ca)}(L,K)||\Hom_{\cc_{\Z/2}(\ca)}(L,M)|}{|\Hom_{\cc_{\Z/2}(\ca)}(L,K)||\Hom_{\cc_{\Z/2}(\ca)}(L,N)|},
\end{equation}
and so
\begin{align}
\label{VW}
|\cv_{[A],[B],\alpha,\beta}|=|\cw_{[A],[B],\alpha,\beta}| |\ker\varphi|=|\cw_{[A],[B],\alpha,\beta}| \frac{|\Ext^1_{\cc_{\Z/2}(\ca)}(L,K)||\Hom_{\cc_{\Z/2}(\ca)}(L,M)|}{|\Hom_{\cc_{\Z/2}(\ca)}(L,K)||\Hom_{\cc_{\Z/2}(\ca)}(L,N)|}.
\end{align}

On the other hand, we set
$$\cu_{[A],[B],\alpha,\beta}:=\bigsqcup_ {\tiny\begin{array}{cc} [U],U= (\xymatrix{ U^0  \ar@<0.5ex>[r]^{d^0}& U^1 \ar@<0.5ex>[l]^{d^1}  })\\
H^0(U)\simeq A,H^1(U)\simeq B\\
\widehat{\Im d^0}=\alpha,\widehat{\Im d^1}=\beta \end{array}} \Ext^1_{\cc_{\Z/2}(\ca)}(L,K\oplus N)_U.$$
Similarly, by applying $\Hom_{\cc_{\Z/2}(\ca)}(L,-)$ to the split exact sequence $0\rightarrow K\rightarrow K\oplus N\rightarrow N\rightarrow0$, we can obtain that
\begin{align}
\label{UW}
|\cu_{[A],[B],\alpha,\beta}|=|\cw_{[A],[B],\alpha,\beta}| |\Ext^1_{\cc_{\Z/2}(\ca)}(L,K)|.
\end{align}

Combining \eqref{VW} and \eqref{UW}, we have
\begin{eqnarray*}
\mbox{LHS of (\ref{equation left ideal 1}) }&=&\frac{|\cv_{[A],[B],\alpha,\beta}|}{|\Hom_{\cc_{\Z/2}(\ca)}(L,M)|}\\
 &=&\frac{|\cw_{[A],[B],\alpha,\beta}| |\Ext^1_{\cc_{\Z/2}(\ca)}(L,K)|}{|\Hom_{\cc_{\Z/2}(\ca)}(L,K)||\Hom_{\cc_{\Z/2}(\ca)}(L,N)| }\\
 &=&\frac{|\cu_{[A],[B],\alpha,\beta}|}{ |\Hom_{\cc_{\Z/2}(\ca)}(L,K\oplus N)|}\\
 &=&\mbox{RHS of (\ref{equation left ideal 1})}.
\end{eqnarray*}
The proof is completed.
\end{proof}

\begin{lemma}
\label{lemma right ideal}
Let $0\rightarrow K\xrightarrow{h_1} M\xrightarrow{h_2} N\rightarrow0 $ be a short exact sequence with $K$ acyclic. For any $L$, we have
$$\Upsilon\big(([M]-[K\oplus N] )\diamond [L]\big)=0.$$
\end{lemma}

\begin{proof}
Similar to Lemma \ref{lemma left ideal}, it is enough to prove
\begin{equation}
\label{equation right ideal 1}
\sum_{\tiny\begin{array}{cc} [V],V= (\xymatrix{ V^0  \ar@<0.5ex>[r]^{f^0}& V^1 \ar@<0.5ex>[l]^{f^1}  })\\
H^0(V)\simeq A,H^1(V)\simeq B\\
\widehat{\Im f^0}=\alpha,\widehat{\Im f^1}=\beta \end{array}}\frac{|\Ext^1_{\cc_{\Z/2}(\ca)}(M,L)_V|}{|\Hom_{\cc_{\Z/2}(\ca)}(M,L)|} =\sum_{\tiny\begin{array}{cc} [U],U= (\xymatrix{ U^0  \ar@<0.5ex>[r]^{d^0}& U^1 \ar@<0.5ex>[l]^{d^1}  })\\
H^0(U)\simeq A,H^1(U)\simeq B\\
\widehat{\Im d^0}=\alpha,\widehat{\Im d^1}=\beta \end{array}} \frac{|\Ext^1_{\cc_{\Z/2}(\ca)}(K\oplus N,L)_U|}{|\Hom_{\cc_{\Z/2}(\ca)}(K\oplus N,L)|}
\end{equation}
 for any $[A],[B]\in\Iso(\ca)$ and $\alpha,\beta\in K_0(\ca)$.

There is a natural injective morphism $s: M\rightarrow \Cone(1_M)$, which yields a quasi-isomorphism:
$$g_1:=\left(\begin{array}{cc} h_2\\ s \end{array}\right): M\rightarrow N\oplus \Cone(1_M).$$
So there is a short exact sequence
\begin{equation}\label{equation right ideal 2}
0\longrightarrow M\stackrel{g_1}{\longrightarrow}N\oplus \Cone(1_M)\stackrel{g_2}{\longrightarrow} C\longrightarrow 0
\end{equation}
with $C$ acyclic.

By the Snake Lemma, there is a commutative diagram of short exact sequences
\[\xymatrix{ K \ar@{.>}[r]^{s_1\quad\quad} \ar[d]^{h_1} & \Cone(1_M) \ar@{.>}[r]^{\quad s_2}\ar[d] & C\ar[d]^1\\
M\ar[r]^{g_1\quad\quad\,\,}\ar[d]^{h_2} &N\oplus \Cone(1_M) \ar[r]^{\quad\quad\,\,g_2} \ar[d]&C\ar[d]\\
  N\ar[r]^1& N\ar[r]&0
   }\]
From the short exact sequence in the first row of the above diagram, we have a short exact sequence
\begin{equation}
\label{equation right ideal 3}
0\longrightarrow K\oplus N \stackrel{t_1}{\longrightarrow} \Cone(1_M)\oplus N\stackrel{t_2}{\longrightarrow} C\longrightarrow0.
\end{equation}

We assume that $C=(\xymatrix{ C^0 \ar@<0.5ex>[r]^{c^0}& C^1 \ar@<0.5ex>[l]^{c^1}  }) $ and $K=(\xymatrix{ K^0 \ar@<0.5ex>[r]^{k^0}& K^1 \ar@<0.5ex>[l]^{k^1}  }) $.
Denote by $$\cw'_{[A],[B],\alpha,\beta}:=\bigsqcup_ {\tiny\begin{array}{cc} [W],W= (\xymatrix{ W^0  \ar@<0.5ex>[r]^{e^0}& V^1 \ar@<0.5ex>[l]^{e^1}  })\\
H^0(W)\simeq A,H^1(W)\simeq B\\
\widehat{\Im e^0}=\alpha+\widehat{\Im c^0},\widehat{\Im e^1}=\beta +\widehat{\Im c^1}\end{array}} \Ext^1_{\cc_{\Z/2}(\ca)}(N\oplus \Cone(1_M),L)_W,$$
and $$\cv'_{[A],[B],\alpha,\beta}:=\bigsqcup_ {\tiny\begin{array}{cc} [V],V= (\xymatrix{ V^0  \ar@<0.5ex>[r]^{f^0}& V^1 \ar@<0.5ex>[l]^{f^1}  })\\
H^0(V)\simeq A,H^1(V)\simeq B\\
\widehat{\Im f^0}=\alpha,\widehat{\Im f^1}=\beta \end{array}} \Ext^1_{\cc_{\Z/2}(\ca)}(M,L)_V.$$
Dual to the proof of Lemma \ref{lemma left ideal}, we obtain that
\begin{align}
\label{WV'}
|\cw'_{[A],[B],\alpha,\beta}|=|\cv'_{[A],[B],\alpha,\beta}| \frac{|\Hom_{\cc_{\Z/2}(\ca)}(N\oplus \Cone(1_M),L)| |\Ext^1_{\cc_{\Z/2}(\ca)}(C,L)|}{|\Hom_{\cc_{\Z/2}(\ca)}(C,L)||\Hom_{\cc_{\Z/2}(\ca)}(M,L)|}.
\end{align}
On the other hand, denote by
$$\cu'_{[A],[B],\alpha,\beta}:= \bigsqcup_ {\tiny\begin{array}{cc} [U],U= (\xymatrix{ U^0  \ar@<0.5ex>[r]^{d^0}& U^1 \ar@<0.5ex>[l]^{d^1}  })\\
H^0(U)\simeq A,H^1(U)\simeq B\\
\widehat{\Im d^0}=\alpha,\widehat{\Im d^1}=\beta \end{array}} \Ext^1_{\cc_{\Z/2}(\ca)}(K\oplus N,L)_U.$$
Then
\begin{align}
\label{WU'}
|\cw'_{[A],[B],\alpha,\beta}|=|\cu'_{[A],[B],\alpha,\beta}| \frac{|\Hom_{\cc_{\Z/2}(\ca)}(N\oplus \Cone(1_M),L)| |\Ext^1_{\cc_{\Z/2}(\ca)}(C,L)|}{|\Hom_{\cc_{\Z/2}(\ca)}(C,L)||\Hom_{\cc_{\Z/2}(\ca)}(K\oplus N,L)|}.
\end{align}
By combining \eqref{WV'} and \eqref{WU'}, we have
$$\frac{|\cv'_{[A],[B],\alpha,\beta}|}{|\Hom_{\cc_{\Z/2}(\ca)}(M,L)|}= \frac{|\cu'_{[A],[B],\alpha,\beta}|}{|\Hom_{\cc_{\Z/2}(\ca)}(K\oplus N,L)|},$$
and then (\ref{equation right ideal 1}) follows.
\end{proof}

\begin{lemma}
\label{proposition lem ideal}
The natural linear map $\Upsilon: \ch(\cc_{\Z/2}(\ca))\rightarrow \cm_{\Z/2}(\ca)$ induces a linear map
$\widetilde{\Upsilon}: \cs\cd\ch(\ca)\rightarrow \cm_{\Z/2}(\ca)$. Moreover $\widetilde{\Upsilon}$ is a morphism of $\T_{\Z/2,ac}(\ca)$-bimodules. 
\end{lemma}

\begin{proof}
First, we prove $\Upsilon(I_{\Z/2})=0$. By definition, it is enough to prove that
\begin{align}
\Upsilon\big([L_1]\diamond ([M]-[K\oplus N]) \diamond [L_2]\big)=0,
\end{align}
for any $\Z/2$-complexes $L_1,L_2$ and any short exact sequence
$0\rightarrow K\rightarrow M\rightarrow N\rightarrow0$ with $K$ acyclic.

For $L_2\in\cc_{\Z/2}(\ca)$ and a short exact sequence $0\rightarrow K\xrightarrow{} M\xrightarrow{}N\rightarrow0$,
Lemma \ref{lemma right ideal} and Remark \ref{rem:concellation} yield that there is an acyclic complex $K_2$ such that $\widehat{\Upsilon}\big(([M]-[K\oplus N] )\diamond [L_2]\big)\diamond [K_2]=0$ in $\ch(\cc_{\Z/2}(\ca))/J_{\Z/2}$.
By Lemma \ref{remark coincide of the structure of two bimodule}, we have $\widehat{\Upsilon}\big(([M]-[K\oplus N] )\diamond [L_2]\diamond [K_2]\big)=0$ in $\ch(\cc_{\Z/2}(\ca))/J_{\Z/2}$. 
By definition of $J_{\Z/2}$, there exist short exact sequences $0\rightarrow T_i \rightarrow M_i \rightarrow N_i\rightarrow0$ with $T_i\in\cc_{\Z/2,ac}(\ca)$, and $M_i,N_i\in\cc_{\Z/2}(\ca)$ such that
$$([M]-[K\oplus N] )\diamond [L_2]\diamond [K_2]=\sum_i a_i([M_i]-[T_i\oplus N_i])$$
in $\ch(\cc_{\Z/2}(\ca))$, where $a_i\in \Q$. 
So Lemma \ref{lemma left ideal} implies that
$$\Upsilon\big([L_1]\diamond([M]-[K\oplus N] )\diamond [L_2]\diamond [K_2]\big)=0$$
in $\cm_{\Z/2}(\ca)$, and then
$\Upsilon\big([L_1]\diamond([M]-[K\oplus N] )\diamond [L_2]\big)\diamond [K_2]=0$ by Lemma \ref{remark coincide of the structure of two bimodule}. It follows that $\Upsilon\big([L_1]\diamond([M]-[K\oplus N] )\diamond [L_2]\big)=0$ since $[K_2]$ is invertible in $\T_{\Z/2,ac}(\ca)$. So $\Upsilon(I_{\Z/2})=0$.

Therefore, $\Upsilon$ induces a linear map $\bar{\Upsilon}:\ch(\cc_{\Z/2}(\ca))/I_{\Z/2}\rightarrow \cm_{\Z/2}(\ca)$, which is a morphism of right $\A_{\Z/2,ac}(\ca)$-bimodules by \eqref{eq:bimoduleHall 2}. Note that $\ch(\cc_{\Z/2}(\ca))/I_{\Z/2}\otimes_{\A_{\Z/2,ac}(\ca)} \T_{\Z/2,ac}(\ca) \cong \cs\cd\ch(\ca)$ as right $\T_{\Z/2,ac}(\ca)$-modules.
Then $\bar{\Upsilon}$ induces the morphism of right $\T_{\Z/2,ac}(\ca)$-bimodules 
$$\widetilde{\Upsilon}:  \cs\cd\ch(\ca)\longrightarrow \cm_{\Z/2}(\ca),$$
which is the composition
\[\xymatrix{&\cs\cd\ch(\ca)\cong \ch(\cc_{\Z/2}(\ca))/I_{\Z/2}\otimes_{\A_{\Z/2,ac}(\ca)} \T_{\Z/2,ac}(\ca) \ar[rr]^{\quad\qquad\qquad\qquad\qquad\qquad\bar{\Upsilon}\otimes \T_{\Z/2,ac}(\ca)}&&}\]
\[\xymatrix{\cm_{\Z/2}(\ca)\otimes_{\A_{\Z/2,ac}(\ca)} \T_{\Z/2,ac}(\ca)\ar[r]^{\qquad\qquad\quad\rm mult.}& \cm_{\Z/2}(\ca). }\]
In fact, $\widetilde{\Upsilon}$ maps
$as^{-1}$ to $1\otimes a\otimes s^{-1}$ for any $as^{-1}\in \cs\cd\ch(\ca)$, where $1=[0]$. 
It follows from \eqref{eq:KMMK} and \eqref{definition of bimodule} that $\widetilde{\Upsilon}$ is a morphism of $\T_{\Z/2,ac}(\ca)$-bimodules. 
\end{proof}

\begin{proposition}
\label{proposition ideal} Let $\widetilde{\Upsilon}: \cs\cd\ch(\ca)\rightarrow \cm_{\Z/2}(\ca)$ be the linear map induced by
the natural linear map $\Upsilon: \ch(\cc_{\Z/2}(\ca))\rightarrow \cm_{\Z/2}(\ca)$. Then $\widetilde{\Upsilon}$ is an isomorphism as a linear map and also as a morphism of $\T_{\Z/2,ac}(\ca)$-bimodules.
\end{proposition}

\begin{proof}

From Lemma \ref{proposition lem ideal}, $\widetilde{\Upsilon}: \cs\cd\ch(\ca)\rightarrow \cm_{\Z/2}(\ca)$ is a morphism of $\T_{\Z/2,ac}(\ca)$-bimodules.
Consider the morphism of $\T_{\Z/2,ac}(\ca)$-bimodules $\widetilde{\Psi}: \cm_{\Z/2}(\ca) \rightarrow \cs\cd\ch(\ca)$  defined in \eqref{eq:morPsi}. Clearly, $\widetilde{\Psi}\circ\widetilde{\Upsilon}$ and $\widetilde{\Upsilon}\circ\widetilde{\Psi}$
are identity maps, and so $ \cs\cd\ch(\ca)\cong\cm_{\Z/2}(\ca)$
as vector spaces and also as $\T_{\Z/2,ac}(\ca)$-bimodules.
\end{proof}


The following is our main result in this section. 

\begin{theorem} We have the following.
\label{theorem basis of modified hall algebra}
\begin{itemize}
\item[(1).]
$\cs\cd\ch(\ca)$ has a basis given by
$$\{[C_A\oplus C_B^*]\diamond [K_\alpha]\diamond[K_\beta^*]\mid [A],[B]\in\Iso(\ca),\alpha,\beta\in K_0(\ca)\}$$
or by $$\{[K_\alpha]\diamond[K_\beta^*]\diamond [C_A\oplus C_B^*] \mid [A],[B]\in\Iso(\ca),\alpha,\beta\in K_0(\ca)\}.$$
\item[(2).] Let $M=(\xymatrix{ M^0 \ar@<0.5ex>[r]^{f^0}& M^1 \ar@<0.5ex>[l]^{f^1}  })\in \cc_{\Z/2}(\ca)$. Then we have
\begin{align*}
[M]=&\langle \widehat{\Im f^0} , \widehat{\Im f^1}\rangle \langle\widehat{\Im f^0}, \widehat{H^0(M)} \rangle \langle\widehat{ \Im f^1}, \widehat{H^1(M) }\rangle[K_{\widehat{\Im f^0}}]\diamond [K_{\widehat{\Im f^1}}^*]\diamond [C_{H^0(M)}^*\oplus C_{H^1(M)}]
\end{align*}
in  $\cs\cd\ch(\ca)$.
\end{itemize}
\end{theorem}

\begin{proof}
(1) follows from Proposition \ref{theorem basis of modified hall module}, Proposition \ref{proposition ideal} and (\ref{eq:KMMK}).

Since $K_A$ and $K_A^*$  only depend on $\widehat{A}\in K_0(\ca)$, (2) follows from Proposition \ref{lemma canonical form of complex} and Proposition \ref{proposition ideal}.
\end{proof}



\begin{corollary}
\label{corollary basis of semi-derived hall algebra}
 $\cs\cd\ch(\ca)$ is
a free right (respectively left) module over the quantum torus $\T_{\Z/2,ac}(\ca)$, with a basis given by $\{[C_A\oplus C_B^*] \mid [A],[B]\in\Iso(\ca)\}$.
\end{corollary}
\begin{proof}
From Theorem \ref{theorem basis of modified hall algebra}, we obtain that
$$\cs\cd\ch(\ca)=\bigoplus_{[A],[B]\in \Iso(\ca)} [C_A\oplus C_B^*]\diamond \T_{\Z/2,ac}(\ca),$$
and then it is
a free right module over the quantum torus $\T_{\Z/2,ac}(\ca)$.
\end{proof}

We denote by $\cd_{\Z/2}(\ca)$ the $\Z/2$-graded derived category, i.e., the localization of the homotopy category $\ck_{\Z/2}(\ca)$ with respect to quasi-isomorphisms.
For any $M,N\in\cc_{\Z/2}(\ca)$, we have $M\cong N$ in $\cd_{\Z/2}(\ca)$ if and only if $H^\bullet(M)\cong H^\bullet (N)$; see e.g. \cite{PX1} or the proof of Proposition \ref{lemma canonical form of complex}.
So Corollary \ref{corollary basis of semi-derived hall algebra} can be restated as follows.

\begin{corollary}
\label{corollary basis of semi-derived hall algebra2}
 $\cs\cd\ch(\ca)$ is
a free right (respectively left) module over the quantum torus $\T_{\Z/2,ac}(\ca)$, with a basis given by $\Iso (\cd_{\Z/2}(\ca))$.
\end{corollary}

\begin{remark}
By Corollary \ref{corollary basis of semi-derived hall algebra2}, the underlying vector space of $\cs\cd\ch(\ca)$ is the same as the one used by Gorsky \cite{Gor13} to define the $\Z/2$-graded semi-derived Hall algebra of $\ca$ in case $\ca$ has enough projective objects.
\end{remark}



\subsection{A triangular decomposition}

In this subsection, we give  another basis for $\cs\cd\ch(\ca)$, which gives a triangular decomposition of $\cs\cd\ch(\ca)$.

The following result might be known for experts; we include a proof here as we cannot find a suitable reference.

\begin{lemma}
\label{lemma:nonzero object}
For any nonzero object $A\in\ca$, the following holds.
\begin{itemize}
\item[(i).] There exists no monomorphism $t:A\hookrightarrow A$ with $\coker t\neq0$;
\item[(ii).] There exists no epimorphism $s:A\twoheadrightarrow A$ with $\ker s\neq0$.
\end{itemize}
\end{lemma}
\begin{proof}
We only need to prove (i) since (ii) is dual.

For any monomorphism $t:A\rightarrow A$, we have $t^i\in \Hom_\ca(A,A)$ for any $i\geq0$. Here $t^0=\id$. As $\dim\Hom_\ca(A,A)<\infty$, there exists $0\leq m<n$ such that $\sum_{i=m}^n a_i t^i=0$ with $a_m,a_n\neq0$. Then
$t^m\sum_{i=m}^n a_it^{i-m}=0$, which shows that $\sum_{i=m}^n a_it^{i-m}=0$ since $t^m$ is injective. It follows that $a_m\id=a_{m+1}t+\cdots +a_{n}t^{n-m}=t(a_{m+1}\id+\cdots +a_{n}t^{n-m-1})$,
and so $t$ is surjective. Therefore, $t$ is an isomorphism.
\end{proof}

\begin{theorem}
\label{lemma basis of semi-derived hall algebra of A}
$\cs\cd\ch(\ca)$ has a basis given by
\begin{equation}
\label{eqn:basis}
\{[C_A]\diamond[C_B^*]\diamond K_\alpha\diamond K_\beta^*\mid  [A],[B]\in\Iso(\ca),\alpha,\beta\in K_0(\ca)\}.
\end{equation}
\end{theorem}
\begin{proof}
From Theorem \ref{theorem basis of modified hall algebra}, $\{[C_A\oplus C_B^*]\diamond K_\alpha\diamond K_\beta^*\mid [A],[B]\in\Iso(\ca),\alpha,\beta\in K_0(\ca)\}$ is a basis of $\cs\cd\ch(\ca)$ .

For any  $[A],[B]\in\Iso(\ca)$ and $\alpha,\beta\in K_0(\ca)$,
we obtain that
\begin{eqnarray*}
[C_A]\diamond[C_B^*]\diamond K_\alpha\diamond K_\beta^*&=&\sum_{[X]\in\Iso(\cc_{Z/2}(\ca))} |\Ext^1(C_A,C_B^*)_X|[X]\diamond K_\alpha\diamond K_\beta^*,
\end{eqnarray*}
and $[X]$ is of the form $[\xymatrix{B\ar@<0.5ex>[r]^0& A\ar@<0.5ex>[l]^{f}}]$  if $|\Ext^1_{\cc_{\Z/2}(\ca)}(C_A,C_B^*)_X|\neq0$. In this case,
\begin{align*}
[X]=\langle \coker f,\Im f \rangle[C_{\ker f}\oplus C_{\coker{f}}^*]\diamond K_{\Im f}^*,
\end{align*}
and then
\begin{eqnarray*}
&&[C_A]*[C_B^*]\diamond K_\alpha\diamond K_\beta^*\\
&=&[C_A\oplus C_B^*]\diamond K_\alpha\diamond K_\beta^*+
\sum_{[\xymatrix{B\ar@<0.5ex>[r]^0& A\ar@<0.5ex>[l]^{f}}],f\neq0 }
 a_f [C_{\ker f}\oplus C_{\coker{f}}^*]\diamond K_{\Im f}^*\diamond K_\alpha\diamond K_\beta^*,
\end{eqnarray*}
for some $a_f\in\Q$. 

For each $\xymatrix{B\ar@<0.5ex>[r]^0& A\ar@<0.5ex>[l]^{f}}$, there exists a short exact sequence
$$0\longrightarrow C_B^*\longrightarrow (\xymatrix{B\ar@<0.5ex>[r]^0& A\ar@<0.5ex>[l]^{f}})\longrightarrow C_A\longrightarrow0,$$
which yields a triangle
\begin{align}
\label{eq:triangle}
C_B^*\stackrel{\omega}{\longrightarrow} C_{\ker f}\oplus C_{\coker{f}}^*\longrightarrow C_A\longrightarrow C_B
\end{align}
in $\cd_{\Z/2}(\ca)$ as $(\xymatrix{B\ar@<0.5ex>[r]^0& A\ar@<0.5ex>[l]^{f}})\cong C_{\ker f}\oplus C_{\coker{f}}^*$ in $\cd_{\Z/2}(\ca)$.

Claim ($\star$): we have
\begin{align}
\label{claim}
|\End_{\cd_{\Z/2}(\ca)}(C_{\ker f}\oplus C_{\coker{f}}^*) |\leq| \End_{\cd_{\Z/2}(\ca)}(C_A\oplus C_B^*)|,
\end{align}
and the equality holds if and only if $f=0$. The analogous result for abelian categories or exact categories is already known (see \cite[Lemma 2.1]{GP} and \cite[Proposition 4.8]{BG}). For the sake of self-containment, we give a proof here.

Let us prove the Claim. 
 By applying $\Hom_{\cd_{\Z/2}(\ca)}(C_{\ker f}\oplus C_{\coker f}^*,- )$ to \eqref{eq:triangle}, we have a long exact sequence
\begin{align*}
\cdots \longrightarrow&\Hom_{\cd_{\Z/2}(\ca)}(C_{\ker f}\oplus C_{\coker f}^*, C_B^*)\longrightarrow \End_{\cd_{\Z/2}(\ca)}(C_{\ker f}\oplus C_{\coker f}^*)
\\
\longrightarrow &\Hom_{\cd_{\Z/2}(\ca)}(C_{\ker f}\oplus C_{\coker f}^*,C_A )\longrightarrow\cdots.
\end{align*}
Then
\begin{align}
\label{eq:noneq1}
|\End_{\cd_{\Z/2}(\ca)}(C_{\ker f}\oplus C_{\coker f}^*)|\leq |\Hom_{\cd_{\Z/2}(\ca)}(C_{\ker f}\oplus C_{\coker f}^*, C_B^*\oplus C_A )|.
\end{align}
Dually, by applying $\Hom_{\cd_{\Z/2}(\ca)}(-, C_B^*\oplus C_A)$ to \eqref{eq:triangle}, one can prove that
\begin{align}
\label{eq:noneq2}
|\Hom_{\cd_{\Z/2}(\ca)}(C_{\ker f}\oplus C_{\coker f}^*, C_B^*\oplus C_A )|\leq | \End_{\cd_{\Z/2}(\ca)}(C_A\oplus C_B^*)|.
\end{align}
Then \eqref{claim} follows from \eqref{eq:noneq1}--\eqref{eq:noneq2}.
If the equality holds, then any morphism $C_B^*\rightarrow C_B^*\oplus C_A$ factors through $\omega$, and so \eqref{eq:triangle} is split. Then $C_A\oplus C_B^*\cong C_{\ker f}\oplus C_{\coker f}^*$ in $\cd_{\Z/2}(\ca)$, which implies that $A\cong \ker f$ and $B\cong\coker f$ by comparing their homology groups. So
$f=0$ by Lemma \ref{lemma:nonzero object}. Conversely, if $f=0$, it is clear that the equality holds. The Claim is proved.

For any $m\in\Z$, denote by $\cs\cd\ch(\ca)_m$  the subspace of $\cs\cd\ch(\ca)$ spanned by $[M]\diamond K_\alpha\diamond K_\beta^*$ for $\alpha,\beta\in K_0(\ca)$ and $[M]\in \Iso(\cc_{\Z/2}(\ca))$ with $|\End_{\cd_{\Z/2}(\ca)}(M)|= m$.
Clearly, this makes $\cs\cd\ch(\ca)$ to be a $\Z$-graded vector space
$$\cs\cd\ch(\ca)= \bigoplus_{m\in\Z}\cs\cd\ch(\ca)_m.$$
Then it is not hard to prove that the set (\ref{eqn:basis}) is linearly independent.


On the other hand, $$[C_A\oplus C_B^*]=[C_A]\diamond[C_B^*]-\sum_{[\xymatrix{B\ar@<0.5ex>[r]^0& A\ar@<0.5ex>[l]^{f}}], f\neq0 } a_f [C_{\ker f}\oplus C_{\coker{f}}^*]\diamond K_{\Im f}^*.$$
From Claim ($\star$), we obtain that $$|\End_{\cd_{\Z/2}(\ca)}(C_{\ker f}\oplus C_{\coker{f}}^*)| \leq| \End_{\cd_{\Z/2}(\ca)}(C_A\oplus C_B^*)|,$$
and the equality holds if and only if $f=0$.
As $|\End_{\cd_{\Z/2}(\ca)}(C_{\ker f}\oplus C_{\coker{f}}^*)|$ is finite,  one can obtain by induction that
$[C_{\ker f}\oplus C_{\coker{f}}^*] \diamond K_{\Im f}^*\diamond K_\alpha\diamond K_\beta^*$ is in the subspace spanned by  (\ref{eqn:basis}) for any $f\neq0$, and then so is $[C_A\oplus C_B^*]$.
Therefore, the set (\ref{eqn:basis}) is a basis of $\cs\cd\ch(\ca)$.
\end{proof}

\begin{remark}
Theorem \ref{lemma basis of semi-derived hall algebra of A} gives a triangular decomposition of $\cs\cd\ch(\ca)$. Moreover, this result does not require any additional condition on the hereditary abelian category $\ca$. In \cite[Lemma 4.7]{Br}, Bridgeland proved a similar result for Bridgeland's Hall algebras. However, the proof there requires the category $\ca$ to satisfy that a nonzero object defines a nonzero class in the Grothendieck group of $\ca$.
\end{remark}

\section{Drinfeld double}
\label{sec:Drinfeld}

As in Section \ref{sec:semi}, we shall assume in this section that $\ca$ is a hereditary abelian $\K$-linear category which is essentially small with finite-dimensional homomorphism and extension spaces.
We assume that Ringel-Hall algebras are defined over $\C$.

\subsection{Extended Ringel-Hall algebras}

Denote by $(\cdot,\cdot)$ the {\em symmetric Euler form}, i.e.,
\begin{align}
(\alpha,\beta)=\langle \alpha,\beta\rangle\langle \beta,\alpha\rangle,\,\,\forall \alpha,\beta\in K_0(\ca).
\end{align}

\begin{definition}[\cite{R2,Gr}]
(1) The twisted Ringel-Hall algebra $\ch_{tw}(\ca)$ is the same vector space as $\ch(\ca)$ equipped with the twisted multiplication
$$[A]* [B]=\sqrt{\langle \widehat{A},\widehat{B}\rangle}[A]\diamond [B]$$
for $[A],[B]\in\Iso(\ca)$.

(2) The twisted extended Ringel-Hall algebra $\ch^e_{tw}(\ca)$ is defined as an extension of $\ch_{tw}(\ca)$ by adjoining symbols $k_\alpha$ for classes $\alpha\in K_0(\ca)$, and imposing relations
$$k_\alpha* k_\beta=k_{\alpha+\beta},\quad k_\alpha* [B]=\sqrt{( \alpha,\widehat{B})}  [B]* k_\alpha$$
for $\alpha,\beta\in K_0(\ca)$ and $[B]\in\Iso(\ca)$. Note that $\ch^e_{tw}(\ca)$ has a basis consisting of the elements $k_\alpha* [B]$ for $\alpha\in K_0(\ca)$ and $[B]\in\Iso(\ca)$.
\end{definition}
In fact, there is an isomorphism of vector spaces given by the multiplication
$$\C[K_0(\ca)]\otimes\ch_{tw}(\ca)\longrightarrow \ch^e_{tw}(\ca).$$

We denote by $\ch(\ca)\widehat{\otimes} \ch(\ca)$ the space of formal linear combinations
$$\sum_{[A],[B]\in \Iso(\ca)} c_{A,B} [A]\otimes[B],$$
where $\widehat{\otimes}$ is the \emph{completed tensor product}. Similarly, we can define $\ch^e_{tw}(\ca)\widehat{\otimes} \ch^e_{tw}(\ca)$.

The coproduct and counit for $\ch^e_{tw}(\ca)$ are given by Green \cite{Gr} (see also \cite{X,Cr}):
\begin{eqnarray*}
&&\Delta:\ch^e_{tw}(\ca)\longrightarrow \ch^e_{tw}(\ca)\widehat{\otimes} \ch^e_{tw}(\ca),\qquad \epsilon: \ch^e_{tw}(\ca)\longrightarrow \C,\\
\end{eqnarray*}
\begin{align}
\Delta([A]*k_\alpha)=&\sum_{[B],[C]}\sqrt{\langle \widehat{B},\widehat{C}\rangle}\frac{|\Ext^1_\ca(B,C)_A|}{|\Hom_\ca(B,C)|}\frac{|\aut(A)|}{|\aut(B)||\aut(C)|}([B]*k_{\widehat{C}+\alpha})\otimes[C]*k_\alpha,\label{equation coproduct 1}\\
\Delta(k_\alpha)=&k_\alpha\otimes k_\alpha,\label{equation coproduct 2}\\
\epsilon([A]k_\alpha)=&\delta_{[A],0},
\end{align}
for $[A]\in\Iso(\ca)$ and $\alpha\in K_0(\ca)$. Then $(\ch_{tw}^e(\ca),*,[0],\Delta,\epsilon)$ is a \emph{topological bialgebra} defined over $\C$ (see \cite{Gr,X}).
Here topological means that everything should be considered in the completed space.
\begin{remark}
If $\ca$ is in particular a finite length hereditary category (for instance, the category of nilpotent finite-dimensional representations of a finite quiver), then $(\ch_{tw}^e(\ca),*,[0],\Delta,\epsilon)$ is a genuine bialgebra over $\C$. Moreover, J. Xiao \cite{X} defined the antipode that endowed $\ch_{tw}^e(\ca)$ with a natural Hopf algebra structure.
\end{remark}
\vspace{0.2cm}

\subsection{Drinfeld double}
The bilinear pairing $\varphi:\ch^e_{tw}(\ca)\times \ch^e_{tw}(\ca)\rightarrow\C$ given by
\begin{equation}
\varphi([M]*k_\alpha,[N]*k_\beta)=\sqrt{(\alpha,\beta)}\delta_{[M],[N]}|\aut(M)|
\end{equation}
is a \emph{Hopf pairing} on $\ch_{tw}^e(\ca)$ (see \cite{Gr,R5}), that is, for any $x,y,z\in\ch_{tw}^e(\ca)$, one has
$$\varphi(x*y,z)={\varphi}(x\otimes y,\Delta z).$$
Here we use the usual pairing on the product space:
$\varphi(a\otimes a',b\otimes b')=\varphi(a,b)\varphi(a',b')$. Moreover, this pairing is non-degenerate on $\ch_{tw}(\ca)$ and symmetric.

\vspace{0.2cm}

For the hereditary abelian category $\ca$, there is a unique algebra structure on $\ch^e_{tw}(\ca){\otimes}\ch^e_{tw}(\ca)$ satifying the following conditions, which is called the \emph{Drinfeld double} of $\ch^e_{tw}(\ca)$; see \cite{Dr,Jo,X}, and see also \cite{Gr,Cr,Sch,BS1,BS,DJX,Y}.

\begin{itemize}
\item[(D1)] The maps $$\ch^e_{tw}(\ca)\longrightarrow\ch^e_{tw}(\ca){\otimes}\ch^e_{tw}(\ca),\quad a\mapsto a\otimes1$$
and
$$\ch^e_{tw}(\ca)\longrightarrow\ch^e_{tw}(\ca){\otimes}\ch^e_{tw}(\ca),\quad a\mapsto 1\otimes a$$
are injective homomorphisms of $\C$-algebras. That is, for any $a,a',b,b'\in \ch^e_{tw}(\ca)$,
\begin{eqnarray*}
&&(a\otimes1)(a'\otimes1)=aa'\otimes1,\\
&&(1\otimes b)(1\otimes b')=1\otimes bb'.
\end{eqnarray*}
\item[(D2)] For any $a,b\in \ch^e_{tw}(\ca)$, one has
$$(a\otimes1)(1\otimes b)=(a\otimes b).$$
\item[(D3)] For any $a,b\in \ch^e_{tw}(\ca)$, one has
\begin{equation}\label{equation drinfeld double 2}
\sum \varphi(a_{(2)},b_{(1)})a_{(1)}\otimes b_{(2)}=\sum \varphi(a_{(1)},b_{(2)}) (1\otimes b_{(1)})(a_{(2)}\otimes1).
\end{equation}
\end{itemize}

It is worth noting that if $\ch^e_{tw}(\ca)$ is a topological bialgebra, then one should replace the tensor product $\otimes$
in the statement by the completed one
$\widehat{\otimes}$.

In particular, if $\ca$ is a finite length hereditary abelian category, then the Drinfeld double of $\ch^e_{tw}(\ca)$ is again a Hopf algebra.

\subsection{Twisted semi-derived Ringel-Hall algebras}


Similar to \cite{Br,Gor13}, we construct a twisted version of the algebra $\cs\cd\ch(\ca)$. Define \emph{the componentwise Euler form} on $\Iso(\cc_{\Z/2}(\ca))$ by
$$\langle\cdot,\cdot\rangle_{cw}:\Iso(\cc_{\Z/2}(\ca))\times \Iso(\cc_{\Z/2}(\ca))\longrightarrow\C^\times,\quad \langle [M],[N]\rangle_{cw}=\sqrt{\langle \widehat{M^0},\widehat{N^0}\rangle \langle \widehat{M^1},\widehat{N^1}\rangle},$$
where $M=(\xymatrix{M^0\ar@<0.5ex>[r]& M^1\ar@<0.5ex>[l]})$ and $N=(\xymatrix{N^0\ar@<0.5ex>[r]& N^1\ar@<0.5ex>[l]})$.
This form descends to a bilinear form
$$\langle\cdot,\cdot\rangle_{cw}:K_0(\cc_{\Z/2}(\ca))\times K_0(\cc_{\Z/2}(\ca))\longrightarrow\C^\times.$$

The multiplication in the \emph{twisted semi-derived Ringel-Hall algebra} $\cs\cd\ch_{tw}(\ca)$ over $\C$ is given by
\begin{align}
\label{eq:twistedprod}
[M_1]*[M_2]:=\langle [M_1],[M_2]\rangle_{cw}[M_1]\diamond[M_2],\quad \forall M_1,M_2\in\cc_{\Z/2}(\ca).
\end{align}

On the Grothendieck group of acyclic complexes, 
the bilinear form $\langle\cdot,\cdot\rangle_{cw}$ coincides with the usual Euler form:
\begin{align}
\label{eq:comparison Eulers}
\langle\alpha,\beta\rangle_{cw}=\langle\alpha,\beta\rangle,\,\,\forall \alpha,\beta\in K_0(\cc_{\Z/2,ac}(\ca)).
\end{align}
In fact, by Corollary \ref{lemma coincide of Euler forms}, \eqref{eq:comparison Eulers} holds for $\alpha,\beta\in\{[K_X],[K_Y^*] \mid X,Y\in\ca\}$, and the proof of Proposition \ref{lemma canonical form of complex}
implies that these classes generate the whole Grothendieck group $K_0(\cc_{\Z/2,ac}(\ca))$.

Clearly, $K_\alpha*K_\beta=K_{\alpha+\beta}$ and $K_{\alpha}^**K_\beta^*=K_{\alpha+\beta}^*$ for any $\alpha,\beta\in K_0(\ca)$.

\begin{lemma}
\label{corollary commutation equations}
In $\cs\cd\ch_{tw}(\ca)$, for any $\alpha,\beta\in K_0(\ca)$ and $M\in \cc_{\Z/2}(\ca)$, we have
\begin{align*}
\sqrt{(\alpha,\widehat{M})}[C_M]* &K_\alpha=K_\alpha *[C_M], \qquad \qquad[C_M^*]*K_\alpha=\sqrt{(\alpha,\widehat{M})} K_\alpha *[C_M^*],
\\
[C_M]*K_\alpha^*= &\sqrt{(\widehat{M},\alpha)}K_\alpha^**[C_M] ,\qquad\quad \sqrt{(\widehat{M},\alpha)}[C_M^*]*K_\alpha^*=K_\alpha^**[C_M^*],
\\
&[K_\alpha,K_\beta]=[K_\alpha,K_\beta^*]=[K_\alpha^*,K_\beta^*]=0.
\end{align*}
\end{lemma}

\begin{proof}
From Lemma \ref{corollary hall multiplicatoin of acyclic complexes} and the definitions, we obtain that
$$
[K ]* [M]=\frac{\langle [K] ,[M]\rangle_{cw}\langle [M],[K] \rangle}{\langle [M],[K ]\rangle_{cw}\langle [K] ,[M]\rangle}[M]*[K],
$$
for any $K\in\cc_{\Z/2,ac}(\ca)$, and $M\in\cc_{\Z/2}(\ca)$. 
Now all the formulas in the statement can be checked directly by applying this identity and Proposition \ref{lema euler form}.
\end{proof}

\begin{lemma}
\label{lemma 2 complexes to stalk complexes}
Let $M=\xymatrix{M^0\ar@<0.5ex>[r]^{f^0} & M^1 \ar@<0.5ex>[l]^{0} }$ be in $\cc_{\Z/2}(\ca)$. Then we have
\begin{align}
\label{eq:Mdec}
[M]=\frac{\sqrt{\langle \widehat{\coker f^0},\widehat{\Im f^0}\rangle}}{ \sqrt{\langle \widehat{\ker f^0},\widehat{\Im f^0}\rangle} } [C_{\ker f^0}^*\oplus C_{\coker f^0}]*[K_{\Im f^0}],
\\
\label{eq:Mdec2}
[M^*]=\frac{\sqrt{\langle \widehat{\coker f^0},\widehat{\Im f^0}\rangle}}{ \sqrt{\langle \widehat{\ker f^0},\widehat{\Im f^0}\rangle} } [C_{\ker f^0}\oplus C^*_{\coker f^0}]*[K^*_{\Im f^0}]
\end{align}
in $\cs\cd\ch_{tw}(\ca)$.
\end{lemma}

\begin{proof}
We only need to prove \eqref{eq:Mdec}.  From Proposition \ref{lemma canonical form of complex}, we obtain that
\begin{eqnarray*}
[M]&=&\langle \widehat{\Im f^0},\widehat{\ker f^0 } \rangle K_{\Im f^0} \diamond [C_{\ker f^0}^*\oplus C_{\coker f^0}]\\
&=& \langle \widehat{\Im f^0},\widehat{\ker f^0}  \rangle \frac{\langle [C_{\ker f^0}^*\oplus C_{\coker f^0}] ,[K_{\Im f^0}]\rangle }{\langle [K_{\Im f^0}], [C_{\ker f^0}^*\oplus C_{\coker f^0}]\rangle  } [C_{\ker f^0}^*\oplus C_{\coker f^0}] \diamond[ K_{\Im f^0}]\\
&=& \langle \widehat{\coker f^0} ,\widehat{\Im f^0}\rangle  [C_{\ker f^0}^*\oplus C_{\coker f^0}] \diamond [K_{\Im f^0}]\\
&=& \frac{\sqrt{\langle \widehat{\coker f^0},\widehat{\Im f^0}\rangle}}{ \sqrt{\langle \widehat{\ker f^0},\widehat{\Im f^0}\rangle} } [C_{\ker f^0}^*\oplus C_{\coker f^0}]*[K_{\Im f^0}].
\end{eqnarray*}
\end{proof}

The following result follows from Theorem \ref{lemma basis of semi-derived hall algebra of A}.

\begin{proposition}
\label{cor basis of twisted semi-derived hall algebra of A}
$\cs\cd\ch_{tw}(\ca)$ has a basis given by
\begin{equation}
\label{eqn:basistwist}
\{[C_A]*[C_B^*]* K_\alpha* K_\beta^*\mid  [A],[B]\in\Iso(\ca),\alpha,\beta\in K_0(\ca)\}.
\end{equation}
\end{proposition}

\subsection{An algebra isomorphism}
In this subsection, we prove that the twisted semi-derived Ringel-Hall algebra $\cs\cd\ch_{tw}(\ca)$ is isomorphic to the Drinfeld double of $\ch^e_{tw}(\ca)$. When $\ch^e_{tw}(\ca)$ is only a topological bialgebra, its Drinfeld double is, as a vector space, the completed tensor product $\ch^e_{tw}(\ca)\widehat{\otimes} \ch^e_{tw}(\ca)$. Accordingly, in this case, we need to complete the twisted semi-derived Ringel-Hall algebra  $\cs\cd\ch_{tw}(\ca)$ with respect to the basis
\eqref{eqn:basistwist}
obtained in Proposition \ref{cor basis of twisted semi-derived hall algebra of A}.

Unless otherwise specified below, we do not distinguish the Drinfeld double of the twisted extended Ringel-Hall algebra from the completed ones, and similarly for semi-derived Ringel-Hall algebras.
First, we give some lemmas.

\begin{lemma}
Let $\ca$ be a hereditary abelian $\K$-linear category. There is an embedding of algebras
$$I_+^e:\ch^e_{tw}(\ca)\hookrightarrow\cs\cd\ch_{tw}(\ca)$$
$$[A]\mapsto [C_A],\quad k_\alpha\mapsto K_\alpha,$$
where $A\in\ca,\alpha\in K_0(\ca)$. By composing $I_+^e$ and the involution $*$, we also have an embedding
$$I_{-}^e:\ch^e_{tw}(\ca)\hookrightarrow\cs\cd\ch_{tw}(\ca)$$
$$[A]\mapsto [C_A^*],\quad k_\alpha\mapsto K^*_\alpha.$$
\end{lemma}
\begin{proof}
We only need to check that $I_+^e$ and $I_-^e$ are embedding maps, which follows from Theorem \ref{theorem basis of modified hall algebra}.
\end{proof}

For any $U,V\in\ca$, define the group action of $\aut(U)\times \aut(V)$ on $\Hom_\ca(U,V)$ as follows: for any $(s,t)\in\aut(U)\times \aut(V)$ and $g\in\Hom_\ca(U,V)$,
$(s,t)\cdot g:=tgs^{-1}$.
We denote by \begin{align}
O=\{O_g\mid O_g\mbox{ is the orbit of }  g\in \Hom_\ca(U,V)\}
\end{align}
the set of orbits corresponding to this group action.
Then we have the following lemma.

\begin{lemma}
\label{lemma product of stalk complexes}
Retain the notations as above. For any $A_1,B_2,\widetilde{B}_1,\widetilde{A}_2\in\ca$, in $\cs\cd\ch_{tw}(\ca)$, we have
\begin{align}
\label{eq:A1B2}
[ C_{A_1}]*[C_{B_2}^*]
=&\sum_{O_g\in O:\,g\in \Hom_\ca(A_1,B_2)} |O_g|\frac{\sqrt{\langle \widehat{\coker g},\widehat{\Im g}\rangle}}{\sqrt{\langle \widehat{\ker g},\widehat{\Im g}\rangle}} [C_{\ker g}\oplus C^*_{\coker g}]*[K_{\Im g}^*];
\\
\label{eq:B1A2}
[ C_{\widetilde{B_1}}^*]*[ C_{\widetilde{A}_2}]
=&\sum_{O_f\in O:\,f\in \Hom_\ca(\widetilde{B}_1,\widetilde{A}_2)} |O_f|\frac{\sqrt{\langle \widehat{\coker f},\widehat{\Im f}\rangle}}{\sqrt{\langle \widehat{\ker f},\widehat{\Im f}\rangle}} [C^*_{\ker f}\oplus C_{\coker f}]*[K_{\Im f}].
\end{align}
\end{lemma}

\begin{proof}
We only need to prove \eqref{eq:A1B2} since \eqref{eq:B1A2} is similar.
By definition, we obtain that
\begin{eqnarray*}
[ C_{A_1}]*[C_{B_2}^*]&=& \sum_{[X]\in\Iso(\cc_{\Z/2}(\ca))} \frac{|\Ext^1_{\cc_{\Z/2}(\ca)}( C_{A_1},C_{B_2}^*)_X|}{|\Hom_{\cc_{\Z/2}(\ca)}(C_{A_1},C_{B_2}^*)|} [X].
\end{eqnarray*}
If $|\Ext^1_{\cc_{\Z/2}(\ca)}( C_{A_1},C_{B_2}^*)_X|\neq0$, then $X$ is of the form $\xymatrix{ B_2\ar@<0.5ex>[r]^0 & A_1\ar@<0.5ex>[l]^g}$ for some $g:A_1\rightarrow B_2$.

Let $X=\xymatrix{ B_2\ar@<0.5ex>[r]^0 & A_1\ar@<0.5ex>[l]^g}$.
Set
\begin{align*}\cs_1:=\{(a,b)\in&\Hom_{\cc_{\Z/2}(\ca)}(C_{B_2}^*,X)\times \Hom_{\cc_{\Z/2}(\ca)}(X,C_{A_1})\mid
 \\
 &\exists \text{ a short exact sequence }  0\longrightarrow C_{B_2}^*\xrightarrow{a} X\xrightarrow{b} C_{A_1} \longrightarrow 0\}.
 \end{align*}
Then $\cs_1$ can be identified with
$$\cs_2=\aut(B_2)\times\aut (A_1).$$
There exists a natural surjective map $\cs_1\rightarrow \Ext^1_{\cc_{\Z/2}(\ca)}(C_{A_1},C_{B_2}^*)_X$.
Now $\aut(X)$ acts on $\cs_1$ by $\nu(a,b):=(\nu a,b\nu^{-1})$ for $\nu\in\aut (X)$, and its stabilizer under the map is isomorphic to $\Hom_{\cc_{\Z/2}(\ca)}(C_{A_1},C_{B_2}^*)$, which is zero. So
$|\cs_1|= |\aut (X)| |\Ext^1_{\cc_{\Z/2}(\ca)}( C_{A_1},C_{B_2}^*)_X|$, and then
$$|\Ext^1_{\cc_{\Z/2}(\ca)}( C_{A_1},C_{B_2}^*)_X|=\frac{|\aut(B_2)||\aut(A_1)|}{|\aut(X)|}.$$

Clearly, $\aut(X)=\{(s,t)\in \aut (A_1)\times \aut(B_2) \mid t g=gs\}$, and so
$$|\Ext^1_{\cc_{\Z/2}(\ca)}( C_{A_1},C_{B_2}^*)_X|=\frac{|\aut(B_2)||\aut(A_1)|}{|\{(s,t)\in\aut (A_1)\times \aut (B_2)\mid t g=gs\}|}=|O_g|.$$
Note that there is a natural one to one correspondence between $O=\{O_g\mid g\in \Hom_\ca(A_1,B_2)\}$ and $\{[X]\in\Iso(\cc_{\Z/2}(\ca))\mid|\Ext^1_{\cc_{\Z/2}(\ca)}( C_{A_1},C_{B_2}^*)_X|\neq0\}$. The lemma now follows from \eqref{eq:Mdec2}.
\end{proof}

Let $A,B\in\ca$.
For $X,Y\in \ca,\delta,\widetilde{\delta}\in K_0(\ca)$, we set
\begin{align*}
\cu^{A,B}_{X,Y,\delta}:=&\big\{M=(\xymatrix{ A\ar@<0.5ex>[r]^{u}& B\ar@<0.5ex>[l]^{v}  })\in\cc_{\Z/2}(\ca) \mid H^0(M)\cong X,H^1(M)\cong Y, \widehat{\Im v}=\delta\big\},
\\
\cv^{A,B}_{X,Y,\widetilde{\delta}}:=&\big\{M=(\xymatrix{ A\ar@<0.5ex>[r]^{u}& B\ar@<0.5ex>[l]^{v}  })\in\cc_{\Z/2}(\ca) \mid H^0(M)\cong X,H^1(M)\cong Y, \widehat{\Im u}=\widetilde{\delta}\big\}.
\end{align*}

\begin{lemma}\label{left hand side of equality}
Let $A,B\in\ca$. For any $X,Y\in \ca,\delta,\widetilde{\delta}\in K_0(\ca)$,
we have
\begin{align}
\label{eq:U}
|\cu^{A,B}_{X,Y,\delta}|
=&\sum_{\stackrel{[A_1],[B_2],}{ \,[A_2],\widehat{A_2}=\delta}}\sum_{\tiny\begin{array}{cc}O_g:\,g \in \Hom_\ca(A_1,B_2),\\\,[\ker g]=[X] , [\coker g]=[Y]\end{array}}\frac{|\Ext^1_{\ca}(A_1,A_2)_A|}{|\Hom_{\ca}(A_1,A_2)|}
 \frac{|\Ext^1_{\ca} (A_2,B_2)_B|}{|\Hom_{\ca}(A_2,B_2)|}\\\notag
 &
 \frac{|\aut(A)||\aut(B)||O_g|}{|\aut(A_2)||\aut (A_1)||\aut (B_2)|};
\\
\label{eq:V}
|\cv^{A,B}_{X,Y,\widetilde{\delta}}|
=&\sum_{\stackrel{[\widetilde{A}_2],[\widetilde{B}_1]}{
\,[\widetilde{A}_1],\widehat{\widetilde{A}_1}=\widetilde{\delta}}} \sum_{\tiny\begin{array}{cc}O_f:\, f\in \Hom_\ca(\widetilde{B_1},\widetilde{A}_2),\\\,
[\ker f]=[Y] , \, [\coker f]=[X]\end{array}}  \frac{|\Ext^1_{\ca}(\widetilde{A}_1,\widetilde{A}_2)_A|}{|\Hom_{\ca}(\widetilde{A}_1,\widetilde{A}_2)|}
\frac{|\Ext^1_{\ca} (\widetilde{B}_1,\widetilde{A}_1)_B|}{|\Hom_{\ca}(\widetilde{B}_1,\widetilde{A}_1)|}\\\notag
&
\frac{|\aut(A)||\aut(B)||O_f|}{|\aut(\widetilde{A}_1)||\aut (\widetilde{A}_2)||\aut (\widetilde{B}_1)| }.
 \end{align}
\end{lemma}
\begin{proof}
We only prove \eqref{eq:U} since \eqref{eq:V} is similar. We have

\begin{eqnarray*}
\cu^{A,B}_{X,Y,\delta}
&=& \bigsqcup_{\stackrel{[A_1],[B_2],}{ \,[A_2],\widehat{A_2}=\delta}} \{ M=(\xymatrix{ A\ar@<0.5ex>[r]^{u}& B\ar@<0.5ex>[l]^{v}  })\in \cu^{A,B}_{X,Y,\delta} \mid\coker v\cong A_1,\ker v\cong B_2,\Im v\cong A_2\}\\
&=& \bigsqcup_{\stackrel{[A_1],[B_2],}{ \,[A_2],\widehat{A_2}=\delta}}\{\big((a_1,a_2),(b_1,b_2),u\big)\in \big(\Hom_{\ca}(A_2,A)\times \Hom_{\ca}(A,A_1)\big)\times\big(\Hom_{\ca}(B_2,B)\\
&&\qquad\qquad\times \Hom_{\ca}(B,A_2)\big) \times \Hom_{\ca}(A,B)\mid  0\longrightarrow  A_2\xrightarrow{a_1} A\xrightarrow{a_2} A_1\longrightarrow0, \\
&&\qquad\qquad0\longrightarrow B_2\xrightarrow{b_1} B\xrightarrow{b_2} A_2\longrightarrow0 \mbox{ are exact}, u a_1=0,b_2u=0 \}/ H_1,
\end{eqnarray*}
where $H_1=(\aut (A_1)\times \aut(A_2)\times \aut(B_2))$ and the group action of $H_1$ is defined as follows:
for $(s_{A_1},s_{A_2},s_{B_2})\in H_1$, and $\big((a_1,a_2),(b_1,b_2),u\big)$,
\begin{equation*}
(s_{A_1},s_{A_2},s_{B_2})\cdot\big((a_1,a_2),(b_1,b_2),u\big):=\big((a_1s_{A_2}^{-1}, s_{A_1}a_2),(b_1s_{B_2}^{-1},s_{A_2}b_2),u\big).
\end{equation*}
For $\big((a_1,a_2),(b_1,b_2),u\big)$, since $ua_1=0,b_2u=0$, there exists a unique $g\in \Hom_\ca(A_1,B_2)$ such that
$u= b_1 ga_2$. Note that $H^0(M)\cong \ker g$ and $H^1(M)\cong \coker g$.
Then
\begin{eqnarray*}
\cu^{A,B}_{X,Y,\delta}= \bigsqcup_{\stackrel{[A_1],[B_2],}{\,[A_2],\widehat{A_2}=\delta}} \bigsqcup_{\tiny\begin{array}{cc}O_g:\, g\in\Hom_{\ca}(A_1,B_2), \\\,[\ker g]=[X] , [\coker g]=[Y]\end{array}} \cu_{O_g}/H_1,
\end{eqnarray*}
where 
\begin{eqnarray*}
\cu_{O_g}&:=&\{ \big((a_1,a_2),(b_1,b_2),g'\big)\in  \big(\Hom_{\ca}(A_2,A)\times \Hom_{\ca}(A,A_1)\big)\times\big(\Hom_{\ca}(B_2,B)\\
&&\times \Hom_{\ca}(B,A_2)\big)\times \Hom_{\ca}(A_1,B_2)\mid 0\longrightarrow  A_2\xrightarrow{a_1} A\xrightarrow{a_2} A_1\longrightarrow0, \\
&&0\longrightarrow B_2\xrightarrow{b_1} B\xrightarrow{b_2} A_2\longrightarrow0 \mbox{ are exact}, g'\in O_g \},
\end{eqnarray*}
and the group action of $H_1$ on $\cu_{O_g}$ is defined as follows:
for $((a_1,a_2),(b_1,b_2),g')\in \cu_{O_g}$, and $(s_{A_1},s_{A_2},s_{B_2})\in H_1$,
\begin{equation*}
\big(s_{A_1},s_{A_2},s_{B_2}\big)\cdot\big((a_1,a_2),(b_1,b_2),g'\big):=\big((a_1s_{A_2}^{-1}, s_{A_1}a_2),(b_1s_{B_2}^{-1},s_{A_2}b_2),s_{B_2}g's_{A_1}^{-1}\big).
\end{equation*}

Furthermore,
\begin{eqnarray*}
\cu_{O_g}&=&\{(a_1,a_2)\in  (\Hom_{\ca}(A_2,A)\times \Hom_{\ca}(A,A_1))\mid  0\longrightarrow  A_2\xrightarrow{a_1} A\xrightarrow{a_2} A_1\longrightarrow0\mbox{ is exact}\}\\
&&\times \{(b_1,b_2)\in (\Hom_{\ca}(B_2,B)\times \Hom_{\ca}(B,A_2))\mid 0\longrightarrow B_2\xrightarrow{b_1} B\xrightarrow{b_2} A_2\longrightarrow0 \mbox{ is exact}\}\\
&&\times O_g.
\end{eqnarray*}
Thus,
\begin{eqnarray*}
|\cu_{O_g}|= \frac{|\Ext^1_{\ca}(A_1,A_2)_A||\aut(A)|}{|\Hom_{\ca}(A_1,A_2)|} \frac{|\Ext^1_{\ca}(A_2,B_2)_B||\aut(B)|}{|\Hom_{\ca}(A_2,B_2)|} |O_g|.
\end{eqnarray*}
Clearly, the group action of $H_1$ on $\cu_{O_g}$ is free. So
\begin{eqnarray*}
|\cu_{O_g}/H_1|=\frac{|\Ext^1_{\ca}(A_1,A_2)_A|}{|\Hom_{\ca}(A_1,A_2)|}
  \frac{|\Ext^1_{\ca} (A_2,B_2)_B|}{|\Hom_{\ca}(A_2,B_2)|} \frac{|\aut(A)||\aut(B)||O_g|}{|\aut(A_2)||\aut (A_1)||\aut (B_2)|}.
 \end{eqnarray*}
and then our desired result follows.
\end{proof}

Now we can prove the main result in this section.

\begin{theorem}
\label{theorem semi-derived hall algebra isomorphic to Drinfeld double}
Let $\ca$ be a hereditary abelian $\K$-linear category. Then the twisted semi-derived Ringel-Hall algebra $\cs\cd\ch_{tw}(\ca)$ is isomorphic to the Drinfeld double of $\ch^e_{tw}(\ca)$.
\end{theorem}
\begin{proof}
It follows from Proposition \ref{cor basis of twisted semi-derived hall algebra of A} 
that
the multiplication map $m:a\otimes b\mapsto I^e_+(a)*I^e_-(b)$ defines an isomorphism of vector spaces
$$m:\ch^e_{tw}(\ca)\otimes_\C \ch^e_{tw}(\ca)\xrightarrow{\sim}\cs\cd\ch_{tw}(\ca).$$
Since the Drinfeld double of the twisted extended Ringel-Hall algebra $\ch^e_{tw}(\ca)$ is isomorphic to the left hand side as a vector space, the map $m$ also induces its isomorphism to $\cs\cd\ch_{tw}(\ca)$.

The proof is reduced to checking the equation (\ref{equation drinfeld double 2}) for the elements  of the basis of $\ch^e_{tw}(\ca)$.
Let us write it in the current setting as
\begin{equation}
\label{equation last identity}
\sum \varphi(a_{(2)},b_{(1)}) I_+^e(a_{(1)})*I_-^e(b_{(2)})=\sum \varphi(a_{(1)},b_{(2)}) I_-^e(b_{(1)})*I_+^e(a_{(2)}).
\end{equation}

Recall that the skew-Hopf pairing is $\varphi([M]*k_\alpha,[N]*k_\beta)=\sqrt{(\alpha,\beta)}|\aut(M)|\delta_{[M],[N]}$ for $[M],[N]\in\Iso(\ca)$, $\alpha,\beta\in K_0(\ca)$.

First, we prove \eqref{equation last identity} when $(a,b)=([A],[B])$.

Since
$$\Delta([A])=\sum_{[A_1],[A_2]}\sqrt{\langle \widehat{A_1},\widehat{A_2}\rangle}\frac{|\Ext^1_{\ca}(A_1,A_2)_A|}{|\Hom_{\ca}(A_1,A_2)|}\frac{|\aut(A)|}{|\aut(A_1)||\aut(A_2)|}([A_1]*k_{\widehat{A_2}})\otimes[A_2],$$
and
$$\Delta([B])=\sum_{[B_1],[B_2]}\sqrt{\langle \widehat{B_1},\widehat{B_2}\rangle}\frac{|\Ext^1_{\ca}(B_1,B_2)_B|}{|\Hom_{\ca}(B_1,B_2)|}\frac{|\aut(B)|}{|\aut(B_1)||\aut(B_2)|}([B_1]*k_{\widehat{B_2}})\otimes[B_2],$$
the equation (\ref{equation last identity}) in this case becomes
\begin{align}
\label{equation case 4}
&\sum_{[A_1],[A_2],[B_1],[B_2]} \varphi([A_2],[B_1]*k_{\widehat{B_2}})\sqrt{\langle \widehat{A_1},\widehat{A_2}\rangle} \frac{|\Ext^1_{\ca}(A_1,A_2)_A|}{|\Hom_{\ca}(A_1,A_2)|} \frac{|\aut(A)|}{|\aut(A_1)||\aut(A_2)|}
\\\notag
&\sqrt{\langle \widehat{B_1},\widehat{B_2}\rangle}\frac{|\Ext^1_{\ca}(B_1,B_2)_B|}{|\Hom_{\ca}(B_1,B_2)|}\frac{|\aut(B)|}{|\aut(B_1)||\aut(B_2)|}[C_{A_1}]*[K_{A_2}]*[C_{B_2}^*]\\\notag
=&\sum_{[A_1],[A_2],[B_1],[B_2]} \varphi([A_1]*k_{\widehat{A_2}},[B_2])\sqrt{\langle \widehat{A_1},\widehat{A_2}\rangle} \frac{|\Ext^1_{\ca}(A_1,A_2)_A|}{|\Hom_{\ca}(A_1,A_2)|} \frac{|\aut(A)|}{|\aut(A_1)||\aut(A_2)|}\\\notag
&\sqrt{\langle \widehat{B_1},\widehat{B_2}\rangle}\frac{|\Ext^1_{\ca}(B_1,B_2)_B|}{|\Hom_{\ca}(B_1,B_2)|}\frac{|\aut(B)|}{|\aut(B_1)||\aut(B_2)|}[C_{B_1}^*]*[K_{B_2}^*]*[ C_{A_2}].
\end{align}

We compute
\begin{align*}
&\mbox{LHS of }(\ref{equation case 4})\\
=&\sum_{[A_1],[A_2],[B_2]} |\aut(A_2)|\sqrt{\langle \widehat{A_1},\widehat{A_2}\rangle} \frac{|\Ext^1_{\ca}(A_1,A_2)_A|}{|\Hom_{\ca}(A_1,A_2)|} \frac{|\aut(A)|}{|\aut (A_1)||\aut (A_2)|}
\\
&\sqrt{\langle \widehat{A_2},\widehat{B_2}\rangle}\frac{|\Ext^1_{\ca} (A_2,B_2)_B|}{|\Hom_{\ca}(A_2,B_2)|}\frac{|\aut(B)|}{|\aut (A_2)||\aut (B_2)|}[C_{A_1}]*[K_{A_2}]*[C_{B_2}^*]
\\
=&\sum_{[A_1],[A_2],[B_2]} |\aut(A_2)|\sqrt{\langle \widehat{A_1},\widehat{A_2}\rangle} \frac{|\Ext^1_{\ca}(A_1,A_2)_A|}{|\Hom_{\ca}(A_1,A_2)|} \frac{|\aut(A)|}{|\aut (A_1)||\aut (A_2)|}\\
&\sqrt{\langle \widehat{A_2},\widehat{B_2}\rangle}\frac{|\Ext^1_{\ca} (A_2,B_2)_B|}{|\Hom_{\ca}(A_2,B_2)|}\frac{|\aut(B)|}{|\aut (A_2)||\aut (B_2)|}\frac{1}{\sqrt{(\widehat{A_2},\widehat{B_2})}}
[C_{A_1}]*[C_{B_2}^*]*[K_{A_2}]
\\
=&\sum_{\tiny\begin{array}{cc}[A_1],[A_2],[B_2]\\ O_g:\,g\in \Hom_\ca (A_1,B_2)\end{array}} |\aut(A_2)|\sqrt{\langle \widehat{A_1},\widehat{A_2}\rangle} \frac{|\Ext^1_{\ca}(A_1,A_2)_A|}{|\Hom_{\ca}(A_1,A_2)|} \frac{|\aut(A)|}{|\aut (A_1)||\aut (A_2)|}\\
&\sqrt{\langle \widehat{A_2},\widehat{B_2}\rangle}\frac{|\Ext^1_{\ca} (A_2,B_2)_B|}{|\Hom_{\ca}(A_2,B_2)|}\frac{|\aut(B)|}{|\aut (A_2)||\aut (B_2)|}\frac{1}{\sqrt{(\widehat{A_2},\widehat{B_2})}}\frac{\sqrt{\langle \widehat{\coker g},\widehat{\Im g}\rangle}}{\sqrt{\langle \widehat{\ker g},\widehat{\Im g}\rangle}} \\
&|O_g|[C_{\ker g}\oplus C^*_{\coker g}]*[K_{\Im g}^*]*[K_{A_2}]
\end{align*}
by using \eqref{eq:A1B2}. Then we have
\begin{align*}
&\mbox{LHS of }(\ref{equation case 4})\\
=&\sum_{ \stackrel{ \delta\in K_0(\ca),}{ \,[X],[Y]}}\Big(\sum_{\stackrel{[A_1],[B_2],}{
 \,[A_2],\widehat{A_2}=\delta }} \sum_{\tiny\begin{array}{cc}O_g:\,g \in \Hom_\ca (A_1,B_2),\\\,[\ker g]=[X] , [\coker g]=[Y]\end{array}}|\aut(A_2)|\sqrt{\langle \widehat{A_1},\widehat{A_2}\rangle} \frac{|\Ext^1_{\ca}(A_1,A_2)_A|}{|\Hom_{\ca}(A_1,A_2)|}
\\
& \frac{|\aut(A)|}{|\aut (A_1)||\aut (A_2)|} \sqrt{\langle \widehat{A_2},\widehat{B_2}\rangle}\frac{|\Ext^1_{\ca} (A_2,B_2)_B|}{|\Hom_{\ca}(A_2,B_2)|}\frac{|\aut(B)|}{|\aut (A_2)||\aut (B_2)|}\frac{|O_g|}{\sqrt{(\widehat{A_2},\widehat{B_2})}}\\
&\frac{\sqrt{\langle \widehat{Y},\widehat{A}-\delta-\widehat{X}\rangle}}{\sqrt{\langle \widehat{X}, \widehat{A}-\delta-\widehat{X}\rangle}}\Big) [C_{X}\oplus C^*_{Y}]*K_{\widehat{A}-\delta-\widehat{X}}^**K_{\delta},
\end{align*}
since $\widehat{\Im g}=\widehat{A_1}-\widehat{\ker g}=\widehat{A}-\widehat{A_2}-\widehat{\ker g}$ if the coefficients are nonzero.
Furthermore, if the coefficients are nonzero, then
\begin{align*}
&\sqrt{\langle \widehat{A_1},\widehat{A_2}\rangle} \sqrt{ \langle \widehat{A_2},\widehat{B_2}\rangle} \frac{1}{\sqrt{(\widehat{A_2},\widehat{B_2})}} \frac{\sqrt{\langle \widehat{Y},\widehat{A}-\delta-\widehat{X}\rangle}}{\sqrt{\langle \widehat{X},\widehat{A}-\delta-\widehat{X}\rangle}}\\
=&\sqrt{ \frac{\langle \widehat{A_1},\widehat{A_2}\rangle}{\langle \widehat{B_2},\widehat{A_2}\rangle}}  \frac{\sqrt{\langle \widehat{Y},\widehat{A}-\delta-\widehat{X}\rangle}}{\sqrt{\langle \widehat{X},\widehat{A}-\delta-\widehat{X}\rangle}}
\\
=&\sqrt{ \frac{\langle \widehat{\ker g}+\widehat{\Im g}\widehat{,A_2}\rangle}{\langle \widehat{\coker g}+\widehat{\Im g},\widehat{A_2}\rangle}}  \frac{\sqrt{\langle \widehat{Y},\widehat{A}-\delta-\widehat{X}\rangle}}{\sqrt{\langle \widehat{X},\widehat{A}-\delta-\widehat{X}\rangle}}\\
=&\sqrt{ \frac{\langle\widehat{ X},\delta\rangle}{\langle \widehat{Y},\delta\rangle}}  \frac{\sqrt{\langle \widehat{Y},\widehat{A}-\delta-\widehat{X}\rangle}}{\sqrt{\langle \widehat{X},\widehat{A}-\delta-\widehat{X}\rangle}}.
\end{align*}
Together with \eqref{eq:U}, we obtain that the left hand side of (\ref{equation case 4}) is equal to
\begin{eqnarray*}
&&\sum_{ \stackrel{ \delta\in K_0(\ca),}{ \,[X],[Y]}}\Big(\sum_{\stackrel{[A_1],[B_2],}{
 \,[A_2],\widehat{A_2}=\delta }} \sum_{\tiny\begin{array}{cc}O_g:\,g \in \Hom_\ca (A_1,B_2),\\\,[\ker g]=[X] , [\coker g]=[Y]\end{array}}\sqrt{ \frac{\langle \widehat{X},\delta\rangle}{\langle \widehat{Y},\delta\rangle}}  \frac{\sqrt{\langle \widehat{Y},\widehat{A}-\delta-\widehat{X}\rangle}}{\sqrt{\langle \widehat{X},\widehat{A}-\delta-\widehat{X}\rangle}}\frac{|\Ext^1_{\ca}(A_1,A_2)_A|}{|\Hom_{\ca}(A_1,A_2)|}
 \\
&&\frac{|\Ext^1_{\ca} (A_2,B_2)_B|}{|\Hom_{\ca}(A_2,B_2)|}\frac{|\aut(B)||\aut(A)||O_g|}{|\aut(A_2)||\aut(A_1)||\aut(B_2)|}\Big) [C_{X}\oplus C^*_{Y}]*K_{\widehat{A}-\delta-\widehat{X}}^**K_{\delta}\\
&=&\sum_{ \stackrel{ \delta\in K_0(\ca),}{ \,[X],[Y]}}\Big( \sqrt{ \frac{\langle \widehat{X},\delta\rangle}{\langle \widehat{Y},\delta\rangle}}  \frac{\sqrt{\langle \widehat{Y},\widehat{A}-\delta-\widehat{X}\rangle}}{\sqrt{\langle \widehat{X},\widehat{A}-\delta-\widehat{X}\rangle}}|\cu^{A,B}_{X,Y,\delta}|\Big) [C_{X}\oplus C^*_{Y}]*K_{\widehat{A}-\delta-\widehat{X}}^**K_{\delta}.
\end{eqnarray*}

Similarly, the right hand side of (\ref{equation case 4}) is equal to
\begin{align*}
&\mbox{RHS of }(\ref{equation case 4})\\
=&\sum_{[\widetilde{A}_1],[\widetilde{A}_2],[\widetilde{B}_1]} |\aut(\widetilde{A}_1)|\sqrt{\langle \widehat{\widetilde{A}_1},\widehat{\widetilde{A}_2}\rangle} \frac{|\Ext^1_{\ca}(\widetilde{A}_1,\widetilde{A}_2)_A|}{|\Hom_{\ca}(\widetilde{A}_1,\widetilde{A}_2)|} \frac{|\aut(A)|}{|\aut (\widetilde{A}_1)||\aut (\widetilde{A}_2)|}\\
&\sqrt{\langle \widehat{\widetilde{B}_1},\widehat{\widetilde{A}_1}\rangle}\frac{|\Ext^1_{\ca} (\widetilde{B}_1,\widetilde{A}_1)_B|}{|\Hom_{\ca}(\widetilde{B}_1,\widetilde{A}_1)|}
\frac{|\aut(B)|}{|\aut (\widetilde{B}_1)||\aut (\widetilde{A}_1)|}[C_{\widetilde{B_1}}^*]* [K_{ \widetilde{A}_1}^*]*[C_{\widetilde{A}_2}]\\
=&\sum_{\stackrel{[\widetilde{A}_1],[\widetilde{A}_2],[\widetilde{B}_1]}{
O_f:f\in\Hom_\ca(\widetilde{B}_1,\widetilde{A}_2)}} |\aut(\widetilde{A}_1)|\sqrt{\langle \widehat{\widetilde{A}_1},\widehat{\widetilde{A}_2}\rangle} \frac{|\Ext^1_{\ca}(\widetilde{A}_1,\widetilde{A}_2)_A|}{|\Hom_{\ca}(\widetilde{A}_1,\widetilde{A}_2)|} \frac{|\aut(A)|}{|\aut (\widetilde{A}_1)||\aut (\widetilde{A}_2)|}\\
&\sqrt{\langle \widehat{\widetilde{B}_1},\widehat{\widetilde{A}_1}\rangle}\frac{|\Ext^1_{\ca} (\widetilde{B}_1,\widetilde{A}_1)_B|}{|\Hom_{\ca}(\widetilde{B}_1,\widetilde{A}_1)|}\frac{|\aut(B)|}{|\aut (\widetilde{B}_1)||\aut (\widetilde{A}_1)|}
\frac{|O_f|}{\sqrt{(\widehat{\widetilde{A}_1},\widehat{\widetilde{A}_2})}}\frac{\sqrt{\langle \widehat{\coker f},\widehat{\Im f}\rangle}}{\sqrt{\langle \widehat{\ker f},\widehat{\Im f}\rangle}} \\
&[C^*_{\ker f}\oplus C_{\coker f}]*[K_{\Im f}]*[K_{\widetilde{A}_1}^*]
\end{align*}
by \eqref{eq:B1A2}. Then we have
\begin{align*}
&\mbox{RHS of }(\ref{equation case 4})
\\
=&\sum_{ \stackrel{ \widetilde{\delta}\in K_0(\ca),}{\,[X],[Y]}} \sum_{\stackrel{[\widetilde{A}_2],[\widetilde{B}_1]}
{\,[\widetilde{A}_1],\widehat{\widetilde{A}_1}=\widetilde{\delta}}} \sum_{\tiny\begin{array}{cc}O_f:\, f\in \Hom_\ca(\widetilde{B_1},\widetilde{A}_2),\\\,
[\ker f]=[Y] , [\coker f]=[X]\end{array}}|\aut(\widetilde{A}_1)|\sqrt{\langle \widehat{\widetilde{A}_1},\widehat{\widetilde{A}_2}\rangle} \frac{|\Ext^1_{\ca}(\widetilde{A}_1,\widetilde{A}_2)_A|}{|\Hom_{\ca}(\widetilde{A}_1,\widetilde{A}_2)|} \\
& \frac{|\aut(A)|}{|\aut (\widetilde{A}_1)||\aut (\widetilde{A}_2)|}\sqrt{\langle \widehat{\widetilde{B}_1},\widehat{\widetilde{A}_1}\rangle}
\frac{|\Ext^1_{\ca} (\widetilde{B}_1,\widetilde{A}_1)_B|}{|\Hom_{\ca}(\widetilde{B}_1,\widetilde{A}_1)|}
\frac{|\aut(B)|}{|\aut (\widetilde{B}_1)||\aut (\widetilde{A}_1)|}
\frac{|O_f|}{\sqrt{(\widehat{\widetilde{A}_1},\widehat{\widetilde{A}_2})}}\\
&\frac{\sqrt{\langle \widehat{X}, \widehat{A}-\widetilde{\delta}-\widehat{X}\rangle}}{\sqrt{\langle \widehat{Y},\widehat{A}-\widetilde{\delta}-\widehat{X}\rangle}} [C^*_{Y}\oplus C_{X}]*K_{\widehat{A}-\widetilde{\delta}-\widehat{X}}*K_{\widetilde{\delta}}^*
\\
=&\sum_{\stackrel{ \widetilde{\delta}\in K_0(\ca),}{ \,[X],[Y]}} \sum_{\stackrel{[\widetilde{A}_2],[\widetilde{B}_1]}{
\,[\widetilde{A}_1],\widehat{\widetilde{A}_1}=\widetilde{\delta}}} \sum_{\tiny\begin{array}{cc}O_f:\, f\in \Hom_\ca(\widetilde{B_1},\widetilde{A}_2),\\\,
[\ker f]=[Y],\, [\coker f]=[X]\end{array}} \sqrt{\frac{\langle \widehat{Y},\widetilde{\delta}\rangle}{\langle \widehat{X},\widetilde{\delta}\rangle}} \sqrt{\frac{\langle \widehat{X},\widehat{A}-\widetilde{\delta}-\widehat{X}\rangle}{\langle \widehat{Y},\widehat{A}-\widetilde{\delta}-\widehat{X}\rangle}} \\
&  \frac{|\Ext^1_{\ca}(\widetilde{A}_1,\widetilde{A}_2)_A|}{|\Hom_{\ca}(\widetilde{A}_1,\widetilde{A}_2)|}
\frac{|\Ext^1_{\ca} (\widetilde{B}_1,\widetilde{A}_1)_B|}{|\Hom_{\ca}(\widetilde{B}_1,\widetilde{A}_1)|}
\left.\frac{|\aut(A)||\aut(B)||O_f|}{|\aut(\widetilde{A}_1)||\aut (\widetilde{A}_2)||\aut (\widetilde{B}_1)|}\right)\\
&[C^*_{Y}\oplus C_{X}]*K_{\widehat{A}-\widetilde{\delta}-\widehat{X}}*K_{\widetilde{\delta}}^*\\
=& \sum_{ \stackrel{ \widetilde{\delta}\in K_0(\ca),}{ \,[X],[Y]}} \left(\sqrt{\frac{\langle \widehat{Y},\widetilde{\delta}\rangle}{\langle \widehat{X},\widetilde{\delta}}} \sqrt{\frac{\langle \widehat{X},\widehat{A}-\widetilde{\delta}-\widehat{X}\rangle}{\langle \widehat{Y},\widehat{A}-\widetilde{\delta}-\widehat{X}\rangle}}|\cv^{A,B}_{X,Y,\widetilde{\delta}}|\right)[C^*_{Y}\oplus C_{X}]*K_{\widehat{A}-\widetilde{\delta}-\widehat{X}}*K_{\widetilde{\delta}}^*.
\end{align*}
By definitions, for any $[X],[Y]\in\Iso(\ca)$ and $\delta,\widetilde{\delta}\in K_0(\ca)$, if $\delta+\widetilde{\delta}=\widehat{A}-\widehat{X}$ in $K_0(\ca)$, then $\cu^{A,B}_{X,Y,\delta}=\cv^{A,B}_{X,Y,\widetilde{\delta}}$. Therefore, $\mbox{LHS of }(\ref{equation case 4})=\mbox{RHS of }(\ref{equation case 4})$, and then the identity (\ref{equation case 4}) holds.

Secondly, we prove that (\ref{equation last identity}) holds for the general case $(a,b)=([A]*k_\alpha,[B]* k_\beta)$.
Recall that
\begin{align*}
&\Delta([A]*k_\alpha)=\sum_{[A_1],[A_2]}\sqrt{\langle \widehat{A_1},\widehat{A_2}\rangle}\frac{|\Ext^1_{\ca}(A_1,A_2)_A|}{|\Hom_{\ca}(A_1,A_2)|}\frac{|\aut(A)|}{|\aut(A_1)||\aut(A_2)|}([A_1]*k_{\widehat{A_2}+\alpha})\otimes[A_2]*k_\alpha,
\end{align*}
\begin{align*}
&\Delta([B]*k_\beta)=\sum_{[B_1],[B_2]}\sqrt{\langle \widehat{B_1},\widehat{B_2}\rangle}\frac{|\Ext^1_{\ca}(B_1,B_2)_B|}{|\Hom_{\ca}(B_1,B_2)|}\frac{|\aut(B)|}{|\aut(B_1)||\aut(B_2)|}([B_1]*k_{\widehat{B_2}+\beta})\otimes[B_2]*k_\beta.
\end{align*}
Then in this case, the left hand side of (\ref{equation last identity}) is equal to
\begin{align*}
&\sum_{[A_1],[A_2],[B_1],[B_2]} \varphi([A_2]*k_\alpha,[B_1]*k_{\widehat{B_2}+\beta})\sqrt{\langle \widehat{A_1},\widehat{A_2}\rangle} \frac{|\Ext^1_{\ca}(A_1,A_2)_A|}{|\Hom_{\ca}(A_1,A_2)|} \frac{|\aut(A)|}{|\aut(A_1)||\aut(A_2)|}\\
&\sqrt{\langle \widehat{B_1},\widehat{B_2}\rangle}\frac{|\Ext^1_{\ca}(B_1,B_2)_B|}{|\Hom_{\ca}(B_1,B_2)|}\frac{|\aut(B)|}{|\aut(B_1)||\aut(B_2)|}[C_{A_1}]*K_{\widehat{A_2}+\alpha}*[C_{B_2}^*]*K_\beta^*
\\
=&\sum_{[A_1],[A_2],[B_1],[B_2]} \varphi([A_2],[B_1]*k_{\widehat{B_2}})\sqrt{(\alpha,\widehat{B_2}+\beta)\langle \widehat{A_1},\widehat{A_2}\rangle} \frac{|\Ext^1_{\ca}(A_1,A_2)_A||\aut(A)|}{|\Hom_{\ca}(A_1,A_2)||\aut(A_1)||\aut(A_2)|}\\
&\sqrt{\langle \widehat{B_1},\widehat{B_2}\rangle}\frac{|\Ext^1_{\ca}(B_1,B_2)_B|}{|\Hom_{\ca}(B_1,B_2)|}\frac{|\aut(B)|}{|\aut(B_1)||\aut(B_2)|}\frac{1}{\sqrt{(\alpha,\widehat{B_2})}}[C_{A_1}]*[K_{A_2}]
*[C_{B_2}^*]*K_\alpha*K_\beta^*
\end{align*}
\begin{align*}
=&\sum_{[A_1],[A_2],[B_1],[B_2]} \varphi([A_2],[B_1]*k_{\widehat{B_2}})\sqrt{(\alpha,\beta)}\sqrt{\langle \widehat{A_1},\widehat{A_2}\rangle} \frac{|\Ext^1_{\ca}(A_1,A_2)_A|}{|\Hom_{\ca}(A_1,A_2)|} \frac{|\aut(A)|}{|\aut(A_1)||\aut(A_2)|}\\
&\sqrt{\langle \widehat{B_1},\widehat{B_2}\rangle}\frac{|\Ext^1_{\ca}(B_1,B_2)_B|}{|\Hom_{\ca}(B_1,B_2)|}\frac{|\aut(B)|}{|\aut(B_1)||\aut(B_2)|}[C_{A_1}]*[K_{A_2}]
*[C_{B_2}^*]*K_\alpha*K_\beta^*.
\end{align*}
From the equality (\ref{equation case 4}) proved above, we obtain that
\begin{align*}
&\mbox{LHS of }(\ref{equation last identity})\\
=&\sum_{[A_1],[A_2],[B_1],[B_2]} \varphi([A_1]*k_{\widehat{A_2}},[B_2])\sqrt{(\alpha,\beta)}\sqrt{\langle \widehat{A_1},\widehat{A_2}\rangle} \frac{|\Ext^1_{\ca}(A_1,A_2)_A|}{|\Hom_{\ca}(A_1,A_2)|} \frac{|\aut(A)|}{|\aut(A_1)||\aut(A_2)|}\\
&\sqrt{\langle \widehat{B_1},\widehat{B_2}\rangle}\frac{|\Ext^1_{\ca}(B_1,B_2)_B|}{|\Hom_{\ca}(B_1,B_2)|}\frac{|\aut(B)|}{|\aut(B_1)||\aut(B_2)|}[C_{B_1}^*]*[K_{B_2}^*]*[ C_{A_2}]*K_\alpha*K_\beta^*\\
\\
=&\sum_{[A_1],[A_2],[B_1],[B_2]} \varphi([A_1]*k_{\widehat{A_2}},[B_2]*k_\beta)\frac{\sqrt{\langle \widehat{A_1},\widehat{A_2}\rangle}}{\sqrt{(\widehat{A_2},\beta)}} \frac{|\Ext^1_{\ca}(A_1,A_2)_A|}{|\Hom_{\ca}(A_1,A_2)|} \frac{|\aut(A)|}{|\aut(A_1)||\aut(A_2)|}\\
&\sqrt{\langle \widehat{B_1},\widehat{B_2}\rangle}\frac{|\Ext^1_{\ca}(B_1,B_2)_B|}{|\Hom_{\ca}(B_1,B_2)|}\frac{|\aut(B)|}{|\aut(B_1)||\aut(B_2)|}
\sqrt{(\widehat{A_2},\beta)}[C_{B_1}^*]*[K_{B_2}^*]*K_\beta^**[ C_{A_2}]*K_\alpha\\
=&\sum_{[A_1],[A_2],[B_1],[B_2]} \varphi([A_1]*k_{\widehat{A_2}},[B_2]*k_\beta)\sqrt{\langle \widehat{A_1},\widehat{A_2}\rangle} \frac{|\Ext^1_{\ca}(A_1,A_2)_A|}{|\Hom_{\ca}(A_1,A_2)|} \frac{|\aut(A)|}{|\aut(A_1)||\aut(A_2)|}\\
&\sqrt{\langle \widehat{B_1},\widehat{B_2}\rangle}\frac{|\Ext^1_{\ca}(B_1,B_2)_B|}{|\Hom_{\ca}(B_1,B_2)|}\frac{|\aut(B)|}{|\aut(B_1)||\aut(B_2)|}
[C_{B_1}^*]*K_{\widehat{{B_2}}+\beta}^**[ C_{A_2}]*K_\alpha\\
=&\mbox{RHS of }(\ref{equation last identity}).
\end{align*}
The proof is completed.
\end{proof}

The following is a corollary of Theorem \ref{theorem semi-derived hall algebra isomorphic to Drinfeld double}.
\begin{corollary}
\label{cor:drinfeld double}
Let $\ca$ be a hereditary abelian $\K$-linear category. Then the subspace $\ch_{tw}^e(\ca)\otimes \ch_{tw}^e(\ca)$ (not completed) is a subalgebra of the Drinfeld double of $\ch_{tw}^e(\ca)$, which is isomorphic to $\cs\cd\ch_{tw}(\ca)$.
\end{corollary}
\begin{proof}
Denote by $\Phi$ the isomorphism from the completed Drinfeld double of $\ch_{tw}^e(\ca)$ to the completed $\cs\cd\ch_{tw}(\ca)$ in Theorem \ref{theorem semi-derived hall algebra isomorphic to Drinfeld double}.
For any $[A_i],[B_i]\in\Iso(\ca)$, $\alpha_i,\beta_i\in K_0(\ca)$,  we obtain that
\begin{eqnarray*}
&&\Phi\big(([A_1]*k_{\alpha_1}\otimes [B_1]*k_{\beta_1})*([A_2]*k_{\alpha_2}\otimes [B_2]*k_{\beta_2})\big)
\\
&=&\Phi\big([A_1]*k_{\alpha_1}\otimes [B_1]*k_{\beta_1})*\Phi([A_2]*k_{\alpha_2}\otimes [B_2]*k_{\beta_2}\big)\\
&=&[C_{A_1}]*K_{\alpha_1}*[C_{B_1}^*]*K_{\beta_1}^**[C_{A_2}]*K_{\alpha_2}*[C_{B_2}^*]*K_{\beta_2}^*.
\end{eqnarray*}
Then by Proposition \ref{cor basis of twisted semi-derived hall algebra of A}, we have
$$[C_{A_1}]*K_{\alpha_1}*[C_{B_1}^*]*K_{\beta_1}^**[C_{A_2}]*K_{\alpha_2}*[C_{B_2}^*]*K_{\beta_2}^*= \sum_{\stackrel{\alpha,\beta\in K_0(\ca)}{ \,[A],[B]\in\Iso(\ca)}} a_{\alpha,\beta,[A],[B]} [C_A]*[C_B^*]* K_\alpha* K_\beta^*,$$
with finitely many nonzero terms.
By applying the inverse of $\Phi$, we obtain that
$$([A_1]*k_{\alpha_1}\otimes [B_1]*k_{\beta_1})*([A_2]*k_{\alpha_2}\otimes [B_2]*k_{\beta_2})= \sum_{\stackrel{\alpha,\beta\in K_0(\ca)}{ \,[A],[B]\in\Iso(\ca)}} a_{\alpha,\beta,[A],[B]} b_{\alpha,\beta,[A],[B]}[A]*k_{\alpha}\otimes [B]*k_\beta,$$
for some $b_{\alpha,\beta,[A],[B]}\in\C$.
Then $([A_1]*k_{\alpha_1}\otimes [B_1]*k_{\beta_1})*([A_2]*k_{\alpha_2}\otimes [B_2]*k_{\beta_2})\in \ch_{tw}^e(\ca)\otimes \ch_{tw}^e(\ca)$, and so $\ch_{tw}^e(\ca)\otimes \ch_{tw}^e(\ca)$ is closed under taking the multiplication of the Drinfeld double  of $\ch_{tw}^e(\ca)$.
Therefore, $\ch_{tw}^e(\ca)\otimes \ch_{tw}^e(\ca)$ is a subalgebra of the Drinfeld double  of $\ch_{tw}^e(\ca)$, which is isomorphic to $\cs\cd\ch_{tw}(\ca)$.
\end{proof}

\begin{remark}
I. Burban and O. Schiffmann \cite{BS1} proved that $\ch^e_{tw}(\ca)\otimes\ch^e_{tw}(\ca)$ (not completed) is an associative algebra satisfying the conditions (D1)-(D3) for $\ca$ the category of coherent sheaves over a weighted projective line. T. Cramer \cite{Cr} generalized their result to any hereditary abelian category $\ca$ satisfying the so-called \emph{finitary conditions} (see \cite[Definition 1]{Cr}). Corollary \ref{cor:drinfeld double} generalizes these results to arbitrary hereditary abelian categories.
\end{remark}

The following corollary follows from Theorem \ref{theorem semi-derived hall algebra isomorphic to Drinfeld double}.
\begin{corollary}
\label{cor: coherent sheaves}
Let $\ca$ be the category of finite-dimensional representations of a quiver $Q$ which may have oriented cycles, or the category of coherent sheaves on a smooth projective curve or a weighted projective line over $\K$. Then
the twisted semi-derived Ringel-Hall algebra
$\cs\cd\ch_{tw}(\ca)$ is isomorphic to the Drinfeld double of $\ch^e_{tw}(\ca)$.
\end{corollary}

\begin{corollary}[\cite{Y}]
Let $\ca$ be a hereditary abelian $\K$-linear category with enough projective objects. Then Bridgeland's Hall algebra
$\cd\ch(\ca)$ is isomorphic to the Drinfeld double of $\ch^e_{tw}(\ca)$.
\end{corollary}
\begin{proof}

For such an $\ca$, by \cite[Corollary 9.21]{Gor13}, the twisted semi-derived Hall algebra defined there is isomorphic to
$\cs\cd\ch_{tw}(\ca)$. Together with
\cite[Proposition 9.23]{Gor13}, we have that the twisted semi-derived Ringel-Hall algebra
$\cs\cd\ch_{tw}(\ca)$ is isomorphic to Bridgeland's Hall algebra $\cd\ch(\ca)$. Therefore, from Theorem \ref{theorem semi-derived hall algebra isomorphic to Drinfeld double}, we also obtain that $\cd\ch(\ca)$ is isomorphic to the Drinfeld double of $\ch^e_{tw}(\ca)$.
\end{proof}

\subsection{The reduced version}

As in \cite{Br,Gor13}, we define the \emph{reduced} version of the $\cs\cd\ch_{tw}(\ca)$ by setting $[K]=1$ whenever $K$ is an acyclic complex, invariant under the shift functor:
$$\cs\cd\ch_{tw,red}(\ca):= \cs\cd\ch_{tw}(\ca)/\big([K]-1:K\in\cc_{\Z/2,ac}(\ca),K\cong K^*\big).$$
Furthermore, also as in \cite{Gor13}, this is the same as setting $K_\alpha*K_{\alpha}^*=1,\forall \alpha\in K_0(\ca)$.

\begin{definition}[see e.g. \cite{X,Jo,Cr,BS1,BS}]
Let $\ca$ be a hereditary abelian $\K$-linear category.
The \emph{double Ringel-Hall algebra} $D\ch(\ca)$ is the quotient of the Drinfeld double of $\ch_{tw}^e(\ca)$ by the ideal generated by the elements $k_\alpha\otimes1-1\otimes k_\alpha^{-1}$.
\end{definition}

Then we obtain the following corollary.

\begin{corollary}
\label{cor:SDHDH}
For a hereditary abelian category $\ca$, $\cs\cd\ch_{tw,red}(\ca)$ is isomorphic to the double Ringel-Hall algebra $D\ch(\ca)$.
\end{corollary}
\begin{proof}
It follows from Theorem \ref{theorem semi-derived hall algebra isomorphic to Drinfeld double}.
\end{proof}

\begin{remark}
The quantum loop algebra ${\bf U}_v(\cl \fg)$ was defined as a generalization of the Drinfeld's new realization
of the quantum affine algebra to the loop algebra of any Kac-Moody algebra $\fg$. It has been shown by
Burban and Schiffmann \cite{BS1,BS,Sch,Sch2} that the Ringel-Hall algebra of the category of coherent sheaves on a weighted projective line $\X$ is closely
related to the quantum loop algebra ${\bf U}_v(\cl \fg)$, for some $\fg$ with a star-shaped Dynkin diagram.
Based on their work, R. Dou, Y. Jiang and J. Xiao \cite{DJX} established a homomorphism of algebras from  Drinfeld's presentation of ${\bf U}_v(\cl \fg)$ to the double Ringel-Hall algebra of $\Coh(\X)$.

Using Corollary \ref{cor:SDHDH}, we can use $\cs\cd\ch_{tw,red}(\Coh(\X))$ to realize the Drinfeld's presentation of ${\bf U}_v(\cl \fg)$, for some $\fg$ with a star-shaped Dynkin diagram.
\end{remark}

\begin{remark}
Let ${\bf U}_v(\fg)$ be a quantum generalized Kac-Moody algebra \cite{Bor,K}. Kang and Schiffmann \cite{KS} used $\ch_{tw}({\rm rep}^{\rm nil}_\K(Q))$ to  realize one  half of ${\bf U}_v(\fg)$, where ${\rm rep}^{\rm nil}_\K(Q)$ is the category of nilpotent finite-dimensional representations over some quiver $Q$ (with loops). It follows that the double Ringel-Hall algebra of ${\rm rep}^{\rm nil}_\K(Q)$ can be used to realize the whole ${\bf U}_v(\fg)$. Using Corollary \ref{cor:SDHDH}, we can use $\cs\cd\ch_{tw,red}({\rm rep}^{\rm nil}_\K(Q))$ to realize the whole quantum generalized Kac-Moody algebra ${\bf U}_v(\fg)$.
\end{remark}

For a quiver with at least one oriented cycle, the category $\ca$ of its finite-dimensional representations over $\K$ does not have enough projective objects.
However, its Drinfeld double of the twisted extended Ringel-Hall algebra or equivalently its semi-derived Ringel-Hall algebra is still interesting (see e.g. \cite{Sch,DDF}).


\section{Hereditary categories with tilting objects}
\label{sec:tilting}
In this section, we prove the property of tilting invariance for the semi-derived Ringel-Hall algebras.
\subsection{Tilting objects}
Let $\ca$ be a Krull-Schmidt abelian $\K$-linear category with finite morphism and extension spaces. Let $(\ct,\cf)$ be a pair of full subcategories in $\ca$. Following \cite{Di}, $(\ct,\cf)$ is called a \emph{torsion pair} of $\ca$ if the following conditions are satisfied:
\begin{itemize}
\item[(i).] $\Hom_{\ca}(T,F)=0$ for all $T\in\ct$ and $F\in\cf$.
\item[(ii).] For any $X\in\ca$, there exists a short exact sequence
$$0\longrightarrow T_X\longrightarrow X\longrightarrow F_X\longrightarrow 0$$
such that $T_X\in\ct$ and $F_X\in\cf$.
\end{itemize}
For a torsion pair $(\ct,\cf)$, the following conditions are always satisfied:
\begin{itemize}
\item[(iii).] If $X\in\ca$ and $\Hom_{\ca}(T,X)=0$ for all $T\in\ct$, then $X\in\cf$.
\item[(iv).] If $X\in\ca$ and $\Hom_{\ca}(X,F)=0$ for all $F\in\cf$, then $X\in\ct$.
\end{itemize}

If $(\ct,\cf)$ is a torsion pair, then $\ct$ is called the \emph{torsion class}, and $\cf$ is called the \emph{torsion free class}. $T\in\ct$ is called a \emph{torsion object} while $F\in\cf$ is called a \emph{torsion free object}. Clearly, $\ct$ and $\cf$ are closed under extensions, $\ct$ is closed under factor objects and $\cf$ is closed under subobjects.

Following \cite{HRS}, we say that a torsion class $\ct$ is a \emph{tilting torsion class} if $\ct$ cogenerates $\ca$ (that is, for any $X\in\ca$, there exists a monomorphism $\mu_X:X\rightarrow T_X$ to  an object $T_X\in\ct$).

For an object $A\in\ca$, we denote by $\add A$ the smallest additive full subcategory of $\ca$ containing $A$, that is, the full subcategory of $\ca$ whose objects are the finite
direct sums of direct summands of the object $A$;  
and denote by $\Fac A$ the full subcategory of $\ca$ formed by the epimorphic images of objects in $\add A$.
\begin{definition}[\cite{HRS}]\label{definition of tilting objects}
Let $T\in\ca$. Then $T$ is called a {tilting object} if there exists a torsion pair $(\ct,\cf)$ satisfying the following conditions:
\begin{itemize}
\item[(i).] $\ct$ is a tilting torsion class.
\item[(ii).] $\ct=\Fac T$.
\item[(iii).] $\Ext^i_\ca(T,X)=0$ for $X\in\ct$ and $i>0$, so $T$ is Ext-projective in $\ct$.
\item[(iv).] If $Z\in\ct$ satisfies $\Ext^i_\ca(Z,X)=0$ for all $X\in\ct$ and $i>0$, then $Z\in\add T$.
\item[(v).] If $\Ext^i_\ca(T,X)=0$ for $i\geq 0$ and $X\in\ca$, then $X=0$.
\end{itemize}
\end{definition}
This concept was introduced in \cite{HRS} to obtain a common treatment of both tilted algebras \cite{HRi} and canonical algebras \cite{R1,R3,GL}.
\begin{lemma}[\cite{HRS}]\label{lemma tilting object HRS}
Let $\ca$ be an abelian $\K$-linear category and $(\ct,\cf)$ a torsion pair.
\begin{itemize}
\item[(i).] If $\ct$ is given by a tilting object $T$, then $\Ext^i_\ca(T,-)=0$ for $i\geq2$, i.e., $\pd_\ca (T)\leq1$.

\item[(ii).] If $\ct$ is given by a tilting object $T$ and $X\in\ca$ satisfies $\Ext^1_\ca(T,X)=0$, then $X\in\ct$.

\item[(iii).] Let $T$ be a tilting object in $\ca$. Then $\cd^b(\ca)$ and $\cd^b(\Lambda)$ are derived equivalent, where $\Lambda=\End_\ca(T)^{op}$.
\end{itemize}
\end{lemma}

For a  hereditary abelian $\K$-linear category with a tilting object, D. Happel gave the following characterization.
\begin{theorem}[\cite{Ha}]
Let $\ch$ be a connected hereditary abelian $\K$-linear category with a tilting object. Then $\ch$ is derived equivalent to $\mod H$ for some finite-dimensional hereditary $\K$-algebra $H$ or derived equivalent to $\Coh (\X)$ for some weighted projective line $\X$.
\end{theorem}


\subsection{$T$-resolutions}

Before going on, we recall some definitions from \cite{AS}.
For any Krull-Schmidt additive category $\cb$ (not necessarily hereditary) with finite-dimensional homomorphism spaces, let $\cx$ be a subcategory of $\cb$ closed under isomorphisms and direct summands. A morphism $g:B\rightarrow C$ is called \emph{right minimal} if every $\alpha\in\End(B)$ such that $g\alpha=g$ is an automorphism.
For any $C\in\cb$, a morphism $f:X\rightarrow C$ is a \emph{right $\cx$-approximation} of $C$ if $\Hom_\cb(\cx,f)$ is surjective. 
A right $\cx$-approximation $h:X\rightarrow C$ is a \emph{minimal right $\cx$-approximation} of $C$ if additionally $h$ is a right minimal morphism. An object $C\in\cb$ has a minimal right $\cx$-approximation if it has a right $\cx$-approximation. Any two minimal right $\cx$-approximations $f_i:X_i\rightarrow C$, $i=1,2$ are isomorphic in the sense that there is an isomorphism $g:X_1\rightarrow X_2$ such that $h_1=h_2g$.
The subcategory $\cx$ is \emph{contravariantly finite} in $\cb$ if every $C$ in $\cb$ has a right $\cx$-approximation. 

Again we will use freely the notions of left $\cx$-approximations, minimal left $\cx$-approximations, and $\cx$ being covariantly finite in $\cb$, which are the duals of the notions given above. We say that $\cx$ is \emph{functorially finite} if it is both covariantly finite and contravariantly finite. It is well known that
$\add T$ is functorially finite for any object $T\in\cb$.

\begin{lemma}[Wakamatsu's Lemma] Let $\cx$ be a subcategory of an abelian category $\cb$ which is closed under extensions. For any $C\in\cb$, we have the following.
\begin{itemize}
\item[(i).] If $0\rightarrow Y\rightarrow X\stackrel{f}{\rightarrow}C $ is exact with $f$ being a right minimal $\cx$-approximation of $C$, then $\Ext^1_\cb(\cx,Y)=0$.
\item[(ii).] If $C\stackrel{g}{\rightarrow}X\rightarrow Z\rightarrow0$ is exact with $g$ being a left minimal $\cx$-approximation of $C$, then $\Ext^1_\cb(Z,\cx)=0$.
\end{itemize}
\end{lemma}

We turn to consider tilting objects in $\ca$ now. 

\begin{proposition}
\label{lemma resolution of torsionable object}
Let $T$ be a tilting object in a hereditary abelian $\K$-linear category $\ca$ and $(\ct,\cf)$ be the corresponding torsion pair. Then we have the following.
\begin{itemize}
\item[(i).] For any $M\in\ct$, there is an exact sequence
$$0\longrightarrow T^1\longrightarrow T^0{\longrightarrow} M\longrightarrow0$$
with $T^0,T^1\in\add T$.

\item[(ii).] 
For any $C\in\ca$, there is an exact sequence
$$0\longrightarrow C{\longrightarrow} X\longrightarrow T^2\longrightarrow0$$
with 
$X\in\ct,T^2\in\add T$.
\end{itemize}
\end{proposition}

\begin{proof}
(i). Since $\add T$ is a functorially finite subcategory and $\ct=\Fac(T)$, there is an epimorphism $f:T^0\rightarrow M$ with $T^0\in\add T$ such that $f$ is a right minimal $\add T$-approximation of $M$. Then we obtain a short exact sequence
$$0\longrightarrow \ker f\longrightarrow T^0\xrightarrow{f} M\longrightarrow0.$$
From Wakamatsu's Lemma, we obtain that $\Ext^1_{\ca}(T,\ker f)=0$, and so $\ker f\in \ct$ by Lemma \ref{lemma tilting object HRS} (ii).
For any $X\in\ct$, we obtain a long exact sequence
$$\cdots\longrightarrow0=\Ext^1_{\ca}(T^0, X)\longrightarrow \Ext^1_{\ca}(\ker f,X)\longrightarrow \Ext^2_\ca(M,X)=0\longrightarrow\cdots.$$
So $\Ext^1_{\ca}(\ker f,X)=0$, which implies $\Ext^i_{\ca}(\ker f,\ct)=0$ for any $i>0$ since $\ca$ is hereditary. Then $\ker f\in \add T$ by Definition \ref{definition of tilting objects} (iv).

(ii). By Definition \ref{definition of tilting objects}, we know that $\ct$ is a cogenerator for $\ca$. For any $C\in\ca$, there exists a monomorphism $a:C\rightarrow X^C$ with $X^C\in\ct$, which can be completed to a short exact sequence
$$0\longrightarrow C\stackrel{a}{\longrightarrow} X^C\stackrel{b}{\longrightarrow} Y^C\longrightarrow0.$$
Since $\ct$ is closed under quotients, we have $Y^C\in\ct$.
From (i),  there exists a short exact sequence
$$0\longrightarrow T^1\longrightarrow T^0\longrightarrow Y^C\longrightarrow0$$
with $T^0,T^1\in\add T$.
So we obtain the following pullback diagram:
\[\xymatrix{ & T^1\ar@{=}[r] \ar@{.>}[d] &T^1\ar[d] \\
C\ar@{.>}[r] \ar@{=}[d] & X \ar@{.>}[r] \ar@{.>}[d] & T^0\ar[d] \\
C\ar[r] & X^C\ar[r] &Y^C  }\]
It follows from the short exact sequence in the second column that $X\in\ct$ since $T^1,X^C\in\ct$ and $\ct$ is closed under extensions.
So we obtain a short exact sequence $0\rightarrow C{\rightarrow} X\rightarrow T^0\rightarrow0$ with $X\in\ct$, $T^0\in\add T$.
\end{proof}

\begin{proposition}\label{proposition tilting resolution of bounded complexes}
Let $T$ be a tilting object in a hereditary abelian $\K$-linear category $\ca$ and $(\ct,\cf)$ be the corresponding torsion pair. For any $A\in \cc^b(\ca)$,  there exists a short exact sequence of complexes
$$0\longrightarrow A{\longrightarrow} X\longrightarrow T_2\longrightarrow0,$$
with $T_2$ acyclic and $X\in\cc^b(\ct)$, $T_2\in\cc^b(\add T)$.
\end{proposition}
\begin{proof}
The proof is similar to \cite[Lemma 4.1]{Ke1}.
By the hypothesis and Proposition \ref{lemma resolution of torsionable object} (ii), for any $C\in\ca$, there exist $X^C\in\ct,Y^C\in\add T$ such that
$0\rightarrow C\rightarrow X^C\rightarrow Y^C\rightarrow0$ is exact. Without loss of generality, we can assume that $A$ is of the form
$$A:\cdots\longrightarrow0\longrightarrow A^0\stackrel{d^0}{\longrightarrow} A^1\stackrel{d^1}{\longrightarrow} \cdots \longrightarrow A^n\longrightarrow0\longrightarrow\cdots,\,\,A^p\in\ca.$$
We construct the resolution step by step.

Take a short exact sequence $0\rightarrow A^0\stackrel{l^0}{\rightarrow} X^0\stackrel {\pi^0}{\rightarrow} Y^0\rightarrow0$ with $X^0\in\ct$, $Y^0\in \add T$.
Then 
we obtain the following pushout diagram on the left:
\[ \xymatrix{A^0 \ar[r]^{d^0} \ar[d]^{l_0} & A^1\ar[r]^{d^1} \ar[d]^{s^1} &A^2 & &  A^1\ar@{=}[r] \ar[d]^{s^1} &A^1\ar[d]^{l^1=t^1s^1} &\\
X^0\ar[r] \ar[d]^{\pi_0}\ar[r]^{i^0} &Z^1 \ar@{.>}[ur]_{j^1}\ar[d]^{p^1}  \ar[dr]^{t^1} & && Z^1\ar[r]^{t^1} \ar[d]^{p^1} & X^1 \ar[r] \ar@{.>}[d]^{\pi^1} &T^1 \ar@{=}[d] \\
Y^0 \ar@{=}[r]& Y^0 &X^1 & &  Y^0\ar@{.>}[r] &Y^1 \ar@{.>}[r]& T^1}\]
where $t_1:Z^1\rightarrow X^1$ fits into a short exact sequence $0\rightarrow Z^1\xrightarrow{t^1} X^1\rightarrow T^1\rightarrow0$ with
$X^1\in\ct,T^1\in\add T$. Since $d^1d^0=0$, there exists a unique morphism $j^1:Z^1\rightarrow A^2$ such that $j^1s^1=d^1$ and $j^1i^0=0$.
Consider the above pushout diagram on the right. 

We have $Y^1\cong Y^0\oplus T^1 \in \add T$ since
$\Ext^1_\ca(T^1,Y^0)=0$. 
 In this way, we obtain the following commutative diagram
\[\xymatrix{ A^0 \ar[r]^{d^0} \ar[d]^{l^0} &A^1 \ar[r]^{d^1} \ar[d]^{l^1} &A^2 \\
X^0 \ar[r]^{t^1i^0} \ar[d]^{\pi^0}& X^1 \ar[d]^{\pi^1}  &\\
Y^0\ar[r] &Y^1  &  }\]

Now consider the following pushout diagram on the left:
\[\xymatrix{ Z^1\ar[r]^{j^1} \ar[d]^{t^1}& A^2\ar[d]^{s^2} \ar[r]^{d^2}&A^3   &    A^2\ar@{=}[r] \ar[d]^{s^2} & A^2\ar[d]^{l^2=t^2s^2}& \\
X^1\ar[r]^{i^1} \ar[d]& Z^2\ar[d]^{p^2}\ar@{.>}[ur]_{j^2} && Z^2\ar[r]^{t^2} \ar[d]^{p^2}&X^2\ar[r]\ar[d] &T^2\ar@{=}[d] \\
T^1\ar@{=}[r]&T^1&  &  T^1\ar[r]& Y^2\ar[r] &T^2 }\]
Since $0=d^2d^1=d^2j^1s^1$ and $s^1$ is injective, we obtain $d^2j^1=0$, and then by the property of the pushout, there exists a unique morphism $j^2:Z^2\rightarrow A^3$ such that $ j^2s^2=d^2$ and $j^2i^1=0$.
For $Z^2$, there exists a short exact sequence $0\rightarrow Z^2\stackrel{t^2}{\rightarrow} X^2\rightarrow T^2\rightarrow0$ with $X^2\in\ct$ and $T^2\in\add T$. Then we obtain the above pushout diagram on the right. So $Y^2\cong T^1\oplus T^2\in\add T$.

In the same manner, we construct $ X^3,\dots ,X^{n-1},X^n\in \ct$ and $Y^3,\dots, Y^{n-1},Y^n\in\add T$. Then we have $Z^n\xrightarrow{t^n} X^n\rightarrow T^n$, $j^n=0:Z^n\rightarrow0$. %
Let $X^{n+1}:=T^n$ and $Y^{n+1}:=T^n$.

In this way, we obtain
two complexes
\begin{align*}
X= \cdots\longrightarrow0\longrightarrow X^0\longrightarrow \cdots \longrightarrow X^n\longrightarrow X^{n+1}\longrightarrow 0\longrightarrow\cdots,
\\
T_2= \cdots\longrightarrow0\longrightarrow Y^0\longrightarrow \cdots \longrightarrow Y^n\longrightarrow Y^{n+1}\longrightarrow 0\longrightarrow\cdots,
\end{align*}
which give rise to a short exact sequence of complexes: $0\rightarrow A\rightarrow X\rightarrow T_2\rightarrow0$.

One can verify that $A\rightarrow X$ is a quasi-isomorphism and $T_2$ is acyclic.
\end{proof}


\begin{lemma}\label{lemma tilting resolution of bounded torsion compelxes}
Let $T$ be a tilting object in a hereditary abelian $\K$-linear category $\ca$ and $(\ct,\cf)$ be the corresponding torsion pair. For any $X\in \cc^b(\ct)$, there exists a short exact sequence of complexes
$$0\longrightarrow T_1\longrightarrow T_0{\longrightarrow} X\longrightarrow0,$$
with $T_1$ acyclic and $T_0,T_1\in\cc^b(\add T)$.
\end{lemma}
\begin{proof}
We know that $\ct$ is closed under taking extensions, $\add T\subset\ct$, and $T$ is projective in $\ct$. Thanks to Proposition \ref{lemma resolution of torsionable object}, the same proof as in \cite[Lemma 4.1]{Ke1} applies here.
\end{proof}

\begin{proposition}
\label{theorem resolution of 2 complexes by T}
Let $T$ be a tilting object in a hereditary abelian $\K$-linear category $\ca$ and $(\ct,\cf)$ be the corresponding torsion pair. Then we have the following.
\begin{itemize}
\item[(i).] For any $M\in\cc_{\Z/2}(\ca)$, there exists a short exact sequence of complexes
$$0\longrightarrow M\stackrel{f_1}{\longrightarrow} X\stackrel{f_2}{\longrightarrow} T_2\longrightarrow0,$$
with $T_2$ acyclic and $X\in\cc_{\Z/2}(\ct)$, $T_2\in\cc_{\Z/2}(\add T)$.

\item[(ii).] For any $N\in \cc_{\Z/2}(\ct)$, there exists a short exact sequence of complexes
$$0\longrightarrow T_1\stackrel{g_1}{\longrightarrow}T_0\stackrel{g_2}{\longrightarrow} N\longrightarrow0,$$
with $T_1$ acyclic and $T_0,T_1\in\cc_{\Z/2}(\add T)$.
\end{itemize}
\end{proposition}

\begin{proof}
(i). For any $M=\xymatrix{M^0\ar@<0.5ex>[r]^{f^0} & M^1\ar@<0.5ex>[l]^{f^1}}\in\cc_{\Z/2}(\ca)$, by Lemma \ref{Galois functor is dense for hereditary categories}, there exists a short exact sequence
$$0\longrightarrow K_{\ker f^0}\longrightarrow \pi(U)\longrightarrow M\longrightarrow0$$
with $U\in\cc^b(\ca)$.
Proposition \ref{proposition tilting resolution of bounded complexes} yields an exact sequence
$$0\longrightarrow U{\longrightarrow} W\longrightarrow Y\longrightarrow0$$
with $Y$ acyclic and $W\in\cc^b(\ct)$, $Y\in\cc^b(\add T)$.
Then we obtain the following pushout diagram
\[\xymatrix{ & K_{\ker f^0} \ar[r] \ar@{=}[d] & \pi(U) \ar[r] \ar[d] & M\ar@{.>}[d] &\\
 & K_{\ker f^0} \ar[r]& \pi(W)\ar@{.>}[r] \ar[d] &X  \ar@{.>}[d] &(*) \\
& &  \pi(Y)\ar@{=}[r] &\pi(Y)   &       }\]
From the exact sequence in the second row, we have $X\in \cc_{\Z/2}(\ct)$ since $\ct$ is closed under quotients. So 
$$0\longrightarrow M \longrightarrow X \longrightarrow \pi(Y)\longrightarrow0$$
is the desired short exact sequence.

(ii). For any $N=\xymatrix{N^0\ar@<0.5ex>[r]^{g^0} & N^1\ar@<0.5ex>[l]^{g^1}}\in \cc_{\Z/2}(\ct)$, by Lemma \ref{Galois functor is dense for hereditary categories},  there exists a short exact sequence
$$0\longrightarrow N\longrightarrow \pi(V)\longrightarrow K_{\coker g^0}\longrightarrow0,$$
with $V\in\cc^b(\ca)$.
From Proposition \ref{proposition tilting resolution of bounded complexes}, we can assume that $V\in\cc^b(\ct)$. It follows that $K_{\coker g^0}\in\cc_{\Z/2}(\ct)$ since $\ct$ is closed under quotients. From Lemma \ref{lemma tilting resolution of bounded torsion compelxes}, there exists a short exact sequence
$$0\longrightarrow Y\longrightarrow X{\longrightarrow} V\longrightarrow0$$
with $Y$ acyclic and $X,Y\in\cc^b(\add T)$.
So we obtain the following pullback diagram
\[ \xymatrix{ \pi(Y)\ar@{=}[r] \ar@{.>}[d] &\pi(Y)\ar[d] & &\\
 Z\ar@{.>}[r] \ar@{.>}[d]  &\pi(X)\ar[d]\ar[r] & K_{\coker g^0}\ar@{=}[d]& (**)\\
 N\ar[r] &\pi(V) \ar[r] &K_{\coker g^0} & }   \]

From the short exact sequence in the first column, we obtain that $Z\in\cc_{\Z/2}(\ct)$
since $\ct$ is closed under extensions.
By applying $\Hom_\ca(-,\ct)$ to the short exact sequences in the second row, we obtain that
$Z\in\cc_{\Z/2}(\add T)$. So $0\rightarrow \pi(Y)\rightarrow Z\rightarrow N\rightarrow 0$ is our desired short exact sequence.
\end{proof}

For any $M\in\cc_{\Z/2}(\ca)$, the two short exact sequences in Proposition \ref{theorem resolution of 2 complexes by T} (i) and (ii) as a whole is called a \emph{$T$-resolution} of $M$.

\subsection{The localization of $\ch(\cc_{\Z/2}(\add T))$}

\begin{lemma}[cf. \cite{Br}]
\label{Extpvanish}
Let $T$ be a tilting object of $\ca$. For any $K,T_0\in\cc_{\Z/2}(\add T)$ with $K$ acyclic, we have
$$\Ext^{p}_{\cc_{\Z/2}(\ca)}(T_0,K)=0,\quad \Ext^{p}_{\cc_{\Z/2}(\ca)}(K,T_0)=0,\text{ for any } p\geq1.$$
\end{lemma}

\begin{proof}
Assume $K=\xymatrix{K^0\ar@<0.5ex>[r]^{d^0} & K^1\ar@<0.5ex>[l]^{d^1}}$ with $K^0,K^1\in\add T$. We have the short exact sequences
\begin{align*}
0\longrightarrow \ker d^0 \longrightarrow K^0\longrightarrow \Im d^0\longrightarrow0 \text{ and } 0\longrightarrow \ker d^1 \longrightarrow K^0\longrightarrow \Im d^1\longrightarrow0
\end{align*}
with $\Im d^0\cong \ker d^1$ and $\Im d^1\cong \ker d^0$ since $K$ is acyclic. Since $\ct$ is closed under taking quotients, we obtain that $\Im d^0,\Im d^1\in\ct$. For any $X\in\ct$, by applying $\Hom_{\ca}(-,X)$ to the above short exact sequences, we have $\Ext^i_{\ca}(\Im d^0,X)=0=\Ext^i_\ca(\Im d^1,X)$ for any $i>0$. Then $\Im d^0,\Im d^1\in\add T$ by Definition \ref{definition of tilting objects}.
It follows that $K^0\cong \Im d^0\oplus \Im d^1\cong K^1$, and then
$K\cong K_{\Im d^0}\oplus K_{\Im d^1}^*$ in $\cc_{\Z/2}(\add T)$.

The remaining proof is completely similar to Lemma \ref{lemma extension 2 zero}, and hence is omitted.
\end{proof}

\begin{lemma}
\label{theorem extension isomorphism}
For any hereditary abelian $\K$-linear category $\ca$, let $T$ be a tilting object of $\ca$ and $(\ct,\cf)$ be its corresponding torsion pair. Then for any $P\in\cc_{\Z/2}(\add T)$, $X\in\cc_{\Z/2}(\ct)$, we have
$$\Ext^p_{\cc_{\Z/2}(\ca)}(P,X)=0,\,\,\mbox{ for any } p>0,$$
if either $P$ or $X$ is acyclic.
\end{lemma}

\begin{proof}
From Proposition \ref{theorem resolution of 2 complexes by T}, we obtain a short exact sequence
\begin{equation}
\label{eq:shortexact}
0\longrightarrow T_1\stackrel{g_1}{\longrightarrow}T_0\stackrel{g_2}{\longrightarrow} X\longrightarrow0,
\end{equation}
with $T_0\in\cc_{\Z/2}(\add T), T_1\in\cc_{\Z/2,ac}(\add T)$.
By applying $\Hom_{\cc_{\Z/2}(\ca)}(P,-)$ to \eqref{eq:shortexact}, the desired formula follows from Lemma \ref{Extpvanish}.
\end{proof}

\begin{corollary}
\label{corollary existence of minimal T-resolution}
Let $T$ be a tilting object in a hereditary abelian $\K$-linear category $\ca$ and $(\ct,\cf)$ be the corresponding torsion pair. Then $\cc_{\Z/2}(\ct)$ is a contravariantly finite
subcategory of $\cc_{\Z/2}(\ca)$ and $\cc_{\Z/2}(\add T)$ is a covariantly finite subcategory of $\cc_{\Z/2}(\ct)$.

Furthermore, we have the following.
\begin{itemize}
\item[(i).] For any $M\in\cc_{\Z/2}(\ca)$, if $f_1:M\rightarrow Z$ is a minimal left $\cc_{\Z/2}(\ct)$-approximation, then $f_1$ is injective, and $\coker f_1\in \cc_{\Z/2,ac}(\add T)$.

\item[(ii).] For any $N\in\cc_{\Z/2}(\ct)$, if $g_1:W\rightarrow N$ is a minimal right $\cc_{\Z/2}(\add T)$-approximation, then $g_1$ is surjective, and $\ker g_1\in\cc_{\Z/2,ac}(\add T)$.
\end{itemize}
\end{corollary}
\begin{proof}
For any $M\in\cc_{\Z/2}(\ca)$, Proposition \ref{theorem resolution of 2 complexes by T} (i) shows that there is a short exact sequence
\begin{equation}\label{equation corollary coresolution}
0\longrightarrow M\xrightarrow{u_1} X\xrightarrow{u_2} T_2\longrightarrow0
\end{equation}
with $X\in\cc_{\Z/2}(\ct)$, $T_2\in\cc_{\Z/2,ac}(\add T)$.
Since $\Ext^1_{\cc_{\Z/2}(\ca)}(T_2, U)=0$ for any $U\in \cc_{\Z/2}(\ct)$ by Lemma \ref{theorem extension isomorphism}, we obtain that
$u_1$ is a left $\cc_{\Z/2}(\ct)$-approximation. Thus, $\cc_{\Z/2}(\ct)$ is a contravariantly finite
subcategory of $\cc_{\Z/2}(\ca)$.

Similarly, we can prove that $\cc_{\Z/2}(\add T)$ is a covariantly finite subcategory of $\cc_{\Z/2}(\ct)$.

For (i), if $f_1:M\rightarrow Z$ is a minimal left $\cc_{\Z/2}(\ct)$-approximation, let $f_2:Z\rightarrow Z'$ be the cokernel of $f_1$.
From (\ref{equation corollary coresolution}), there exist $s_1:Z\rightarrow X$ and $s_2:X\rightarrow Z$ such that $u_1s_1=f_1$ and $s_2u_1=f_1$.
We have the following commutative diagram
\[\xymatrix{ M\ar[r]^{f_1} \ar[d]^1 &Z\ar[r]^{f_2} \ar[d]^{s_1} & Z'\ar@{.>}[d]^{t_1}\\
M\ar[r]^{u_1} \ar[d]^1 & X\ar[r]^{u_2} \ar[d]^{s_2} & T_2\ar@{.>}[d]^{t_2}\\
M\ar[r]^{f_1}& Z\ar[r]^{f_2} & Z' }\]
So there exist $t_1,t_2$ such that $t_1f_2 =u_2s_1$ and $t_2u_2=f_2s_2$.
Since $f_1$ is a left minimal morphism, we obtain that $s_2s_1$ is an isomorphism, and then $t_2t_1$ is also an isomorphism.
Thus, $Z'$ is a direct summand of $T_2$, and so $Z'\in\cc_{\Z/2,ac}(\add T)$.

The proof of (ii) is similar, and hence is omitted.
\end{proof}
For any
$M\in\cc_{\Z/2}(\ca)$, it admits a $T$-resolution:
$$0\longrightarrow M\stackrel{f_1}{\longrightarrow} X\stackrel{f_2}{\longrightarrow} T_2\longrightarrow0,\mbox{ and } 0\longrightarrow T_1\stackrel{g_1}{\longrightarrow}T_0\stackrel{g_2}{\longrightarrow} X\longrightarrow0,$$
as in Proposition \ref{theorem resolution of 2 complexes by T}. If $f_1$ is a minimal left $\cc_{\Z/2}(\ct)$-approximation, and $g_2$ is a minimal right $\cc_{\Z/2}(\add T)$-approximation, then this $T$-resolution is called  a \emph{minimal $T$-resolution}. Corollary \ref{corollary existence of minimal T-resolution} implies that a minimal $T$-resolution always exists. Clearly, the minimal $T$-resolution is unique up to isomorphisms.

\begin{proposition}
\label{lemma hall algebra of tilting objects}
Let $\ca$ be a hereditary abelian $\K$-linear category with a tilting object $T$. Let $T_{\Z/2}$ be the subset of $\ch(\cc_{\Z/2}(\add T))$ formed by all $a[K]$, where $a\in\Q^\times$, $K\in\cc_{\Z/2,ac}(\add T)$.
Then
$T_{\Z/2}$ is a multiplicatively closed, right Ore and right reversible subset of $\ch(\cc_{\Z/2}(\add T))$. Moreover, the right localization of $\ch(\cc_{\Z/2}(\add T))$ with respect to $T_{\Z/2}$ exists, and will be denoted by
$\ch(\cc_{\Z/2}(\add T))[T_{\Z/2}^{-1}]$.
\end{proposition}

\begin{proof}
For any complexes $K\in\cc_{\Z,ac}(\add T)$ and $M\in\cc_{\Z/2}(\add T)$, Lemma \ref{theorem extension isomorphism} yields that
$$[K]\diamond [M]=\frac{1}{\langle [K],[M]\rangle}[K\oplus M], \quad [M]\diamond [K]=\frac{1}{\langle [M],[K]\rangle}[K\oplus M].$$
So $$[K]\diamond [M]=\frac{\langle [M],[K]\rangle}{\langle [K],[M]\rangle}[M]\diamond [K].$$
Then one can check that $T_{\Z/2}$ is a multiplicatively closed, right Ore and right reversible subset of $\ch(\cc_{\Z/2}(\add T))$.
\end{proof}

\subsection{Tilting invariance}
In this subsection, we prove that $\ch(\cc_{\Z/2}(\add T))[T_{\Z/2}^{-1}]$ is isomorphic to the semi-derived Ringel-Hall algebra $\cs\cd\ch(\ca)$ as algebras.

By Proposition \ref{theorem resolution of 2 complexes by T}, any complex $M\in\cc_{\Z/2}(\ca)$ admits a $T$-resolution:
\begin{equation}\label{equation torsion resolution}
0\longrightarrow M\xrightarrow{w}X\xrightarrow{v} T_2\longrightarrow0 \mbox{ and }
0\longrightarrow T_1\xrightarrow{u_1} T_0\xrightarrow{u_0} X\longrightarrow0.
\end{equation}

Now, we can prove the main result of this section.

\begin{theorem}\label{proposition isomorphism of semi derived hall algebras }
Let $\ca$ be a hereditary abelian $\K$-linear category with a tilting object $T$. Then there exists an algebra homomorphism
$$\phi:\ch(\cc_{\Z/2}(\add T)) \longrightarrow \ch(\cc_{\Z/2}(\ca)).$$
\begin{itemize}
\item[(i).] The homomorphism $\phi$ induces an algebra isomorphism
$$\widetilde{\phi}:\ch(\cc_{\Z/2}(\add T))[T_{\Z/2}^{-1}] \xrightarrow{\sim} \cs\cd\ch(\ca).$$

\item[(ii).] The inverse of $\widetilde{\phi}$ is given by
$$\widetilde{\psi}: [M]\mapsto\frac{1}{\langle [M],[T_1\oplus T_{2}]\rangle}[T_0]\diamond [T_1\oplus T_2]^{-1},$$
where $T_0,T_1,T_2$ are the objects in any $T$-resolution of $M$ of the form \emph{(\ref{equation torsion resolution})}.
\end{itemize}
\end{theorem}
\begin{proof}
Clearly, $\add T$ is a full and extension-closed subcategory of $\ca$. So $\cc_{\Z/2}(\add T)$ is a full subcategory of $\cc_{\Z/2}(\ca)$, which is closed under extensions. 
Therefore, there exists an algebra monomorphism $\phi: \ch(\cc_{\Z/2}(\add T)) \rightarrow \ch(\cc_{\Z/2}(\ca))$.
Since the full embedding $\cc_{\Z/2}(\add T)\rightarrow \cc_{\Z/2}(\ca)$ is exact and maps acyclic complexes to acyclic ones,  $\phi$ induces an algebra homomorphism
$$\widetilde{\phi}:\ch(\cc_{\Z/2}(\add T))[T_{\Z/2}^{-1}] \longrightarrow (\ch(\cc_{\Z/2}(\ca))/I_{\Z/2})[S_{\Z/2}^{-1}].$$

We define $$\psi:\ch(\cc_{\Z/2}(\ca))\longrightarrow \ch(\cc_{\Z/2}(\add T))[T_{\Z/2}^{-1}]$$ by
$$[M]\mapsto \frac{1}{\langle [M],[T_1\oplus T_{2}]\rangle}[T_0]\diamond [T_1\oplus T_2]^{-1},$$
where $T_0,T_1,T_2$ are the objects in any $T$-resolution of $M$ of the form (\ref{equation torsion resolution}).
We claim that $\psi$ is independent of the $T$-resolution of $M$.

Let us prove the claim. For any minimal $T$-resolution
$$0\longrightarrow M\xrightarrow{w'}X'\xrightarrow{v'} T'_2\longrightarrow0,
\mbox{ and }0\longrightarrow T'_1\xrightarrow{u_1'} T'_0\xrightarrow{u_0'} X'\longrightarrow0,$$
there is a commutative
diagram
\[\xymatrix{M\ar[r]^{w'} \ar@{=}[d] & X' \ar[r]^{v'} \ar[d]^{g} & T'_2 \ar[d]^f& \\
M \ar[r]^{w} &X\ar[r]^{v} &T_2&}\]
which is a both pushout and pullback diagram. Clearly, $g$ and $f$ are sections with isomorphic cokernels. We denote each of these cokernels by $T_2''$. Then
$T_2''\in\cc_{\Z/2,ac}(\add T)$, and
$$T_2\cong T_2'\oplus T_2'',\quad X\cong X'\oplus T_2''.$$

Similarly, there exists the following commutative diagram
\[\xymatrix{T_1'\ar[r]^{u_1'} \ar[d]^{h_1} & T_0' \ar[r]^{u_0'} \ar[d]^{h_0} &X' \ar[d]^{g}\\
T_1 \ar[r]^{u_1} &T_0\ar[r]^{u_0} &X}\]
Then $h_0$ and $h_1$ are sections. Denote by $T_1''$ and $T_0''$ the cokernels of $h_1$ and $h_0$, respectively.
So
$$T_0\cong T_0'\oplus T_0'',\mbox{ and } T_1\cong T_1'\oplus T_1''.$$
Furthermore, by the above commutative diagram, there exists a short exact sequence of acyclic complexes: $0\rightarrow T_1''\rightarrow T_0''\rightarrow T_2''\rightarrow0$, which yields that
$T_0''\cong T_1''\oplus T_2''$.

Then we obtain in $\ch(\cc_{\Z/2}(\add T))[T_{\Z/2}^{-1}]$ that
\begin{eqnarray*}
&&\frac{1}{\langle [M],[T_1\oplus T_2]\rangle} [T_0]\diamond [T_1\oplus T_2]^{-1}\\
&=& \frac{1}{\langle [M],[T_1'\oplus T_1''\oplus T'_2\oplus T_2'']\rangle} [T_0'\oplus T_0'']\diamond [T_1'\oplus T_1''\oplus T'_2\oplus T_2'']^{-1}\\
&=&\frac{1}{\langle [M],[T_1'\oplus T_1''\oplus T'_2\oplus T_2'']\rangle} [T_0'\oplus T_1''\oplus T_2'']\diamond [T_1'\oplus T_1''\oplus T'_2\oplus T_2'']^{-1}\\
&=&\frac{1}{\langle [M],[T_1'\oplus T_1''\oplus T'_2\oplus T_2'']\rangle}  \frac{\langle [T_0'], [T_1''\oplus T_2'']\rangle}{\langle [T_1'\oplus T_2''],[T_1''\oplus T_2'']\rangle}
[T_0']\diamond [T_1''\oplus T_2''] \diamond [T_1''\oplus T_2'']^{-1}\diamond [T_1'\oplus T_2']^{-1}\\
&=&\frac{1}{\langle [M],[T_1'\oplus T_1''\oplus T'_2\oplus T_2'']\rangle} \langle [M],  [T_1''\oplus T_2''] \rangle[T_0']\diamond[T_1'\oplus T_2']^{-1}\\
&=&\frac{1}{\langle [M],[T_1'\oplus T'_2]\rangle}[T_0']\diamond[T_1'\oplus T_2']^{-1}.
\end{eqnarray*}
So $\psi$ is independent of the $T$-resolution of $M$.

For any short exact sequence
$$0\longrightarrow K\xrightarrow{a^1} L\xrightarrow{a^2} M\longrightarrow0$$
in $\cc_{\Z/2}(\ca)$ with $K$ acyclic, we have the following $T$-resolutions:
\begin{eqnarray}
&&\label{equation T-resolution of K}0\longrightarrow K\stackrel{b^1}{\longrightarrow} X^K\stackrel{b^2}{\longrightarrow} T_2^K\longrightarrow0, \mbox{ and } 0\longrightarrow T_1^K\stackrel{c^1}{\longrightarrow} T_0^K\stackrel{c^2}{\longrightarrow} X^K\longrightarrow0,\\
&&\label{equation T-resolution of M}0\longrightarrow M\stackrel{d^1}{\longrightarrow} X^M\stackrel{d^2}{\longrightarrow} T_2^M\longrightarrow0,\mbox{ and }0\longrightarrow T_1^M\stackrel{e^1}{\longrightarrow} T_0^M\stackrel{e^2}{\longrightarrow} X^M\longrightarrow0.
\end{eqnarray}
Note that $X^K,T_0^K$ are acyclic complexes.
Applying $\Hom_{\cc_{\Z/2}(\ca)}(-,K)$ to the short exact sequence $0\rightarrow M\stackrel{d^1}{\rightarrow} X^M\stackrel{d^2}{\rightarrow} T_2^M\rightarrow0$, Proposition \ref{proposition extension 2 zero} yields a surjective map $\Ext^1_{\cc_{\Z/2}(\ca)}(X^M,K)\rightarrow \Ext^1_{\cc_{\Z/2}(\ca)}(M,K)$. 
Thus, there exists $X'$ such that the following diagram commutes:
\[\xymatrix{ K\ar@{=}[r] \ar[d]^{a^1}& K \ar@{.>}[d]^{f_1}  &\\
L\ar@{.>}[r]^{g_1} \ar[d]^{a^2} & X' \ar@{.>}[r]^{g_2} \ar@{.>}[d]^{f_2} & T_2^M \ar@{=}[d]\\
M\ar[r]^{d^1} &X^M \ar[r]^{d^2} & T_2^M}\]
Furthermore, we have the following pushout diagram:
\[ \xymatrix{K\ar[r] \ar[d]^{f_1} & X^K \ar[r] \ar@{.>}[d]^{j_1} & T_2^K\ar@{=}[d]\\
X'\ar@{.>}[r]^{h_1}\ar[d]^{f_2} &X^L \ar@{.>}[r]^{h_2} \ar@{.>}[d]^{j_2} & T_2^K  \\
X^M\ar@{=}[r] &X^M& }\]
From it, we also obtain the following pushout diagram:
\[\xymatrix{ L\ar@{=}[r] \ar[d]^{g_1}  &L\ar[d]^{h_1g_1}& \\
X'\ar[r]^{h_1} \ar[d]^{f_1} & X^L \ar[r]^{h_2} \ar@{.>}[d] & T_2^K \ar@{=}[d]\\
T_2^M\ar@{.>}[r] &W \ar@{.>}[r]& T_2^K}\]
Then $0\rightarrow T_2^M\rightarrow W\rightarrow T_2^K\rightarrow0$ is split since $\Ext^1_{\cc_{\Z/2}(\ca)}(T_2^K,T_2^M)=0$, and so $W\cong T_2^K\oplus T_2^M$. Thus,
$0\rightarrow L\rightarrow X^L\rightarrow T_2^K\oplus T_2^M\rightarrow0$ is exact with $T_2^K\oplus T_2^M\in\cc_{\Z/2,ac}(\add T)$.

We obtain the following commutative diagram of short exact sequences:
\[ \xymatrix{   K \ar[r]^{a^1} \ar[d]^{b^1} &L\ar[r]^{a^2} \ar[d]^{h_1g_1} & M\ar[d]^{d^1} \\
X^K\ar[r]^{j_1} \ar[d]^{b^2} & X^L \ar[r]^{j_2} \ar[d] & X^M \ar[d]^{d^2}\\
T_2^K \ar[r] & T_2^K\oplus T_2^M\ar[r] & T_2^M } \]
Since $X^K$ is acyclic, Lemma \ref{theorem extension isomorphism} shows that $\Ext^1_{\cc_{\Z/2}(\ca)}(T_0^M,X^K)=0$. So we can
construct the following commutative diagram whose  middle row is short exact:
\[ \xymatrix{ T_1^K \ar[r]^{c^1} \ar@{.>}[d]  & T_0^K \ar[r]^{c^2} \ar@{.>}[d] & X^K \ar[d]^{j_1} \\
T_1^K\oplus T_1^M \ar@{.>}[r] \ar@{.>}[d] & T_0^K\oplus T_0^M \ar@{.>}[r] \ar@{.>}[d] & X^L \ar[d]^{j_2} \\
T_1^M\ar[r]^{e^1} & T_0^M \ar[r]^{e^2} & X^M  }\]
Therefore,
\begin{eqnarray*}
\psi([L])&=& \frac{1}{\langle [L],[T_1^K\oplus T_1^M\oplus T_2^K\oplus T_2^M]\rangle} [T_0^K\oplus T_0^M]\diamond [T_1^K\oplus T_1^M\oplus T_2^K\oplus T_2^M]^{-1}\\
&=& \psi([K\oplus M]).
\end{eqnarray*}
In this way, $\psi$ induces a linear map
$$\ch(\cc_{\Z/2}(\ca))/J_{\Z/2} \longrightarrow \ch(\cc_{\Z/2}(\add T))[T_{\Z/2}^{-1}].$$

For any complex $M$ and  acyclic complex $K$, we assume that their $T$-resolutions are as in (\ref{equation T-resolution of K})--(\ref{equation T-resolution of M}).
Then
\begin{align*}
&\psi([M])\diamond \psi([K])\\
=&\frac{1}{\langle [M],[ T_1^M\oplus T_2^M]\rangle\langle [K], [T_1^K\oplus T_2^K]\rangle} [T_0^M]\diamond [T_1^M\oplus T_2^M]^{-1} \diamond [T_0^K]\diamond [T_1^K\oplus T_2^K]^{-1}\\
=&\frac{\langle [T_1^M\oplus T_2^M],[T_0^K]\rangle}{\langle [M], [T_1^M\oplus T_2^M]\rangle\langle [K], [T_1^K\oplus T_2^K]\rangle\langle [T_0^K],[T_1^M\oplus T_2^M]\rangle}\\
& [T_0^M]\diamond  [T_0^K]\diamond [T_1^M\oplus T_2^M]^{-1} \diamond[T_1^K\oplus T_2^K]^{-1}\\
=&\frac{\langle [T_1^M\oplus T_2^M],[T_0^K]\rangle\langle [T_1^K\oplus T_2^K],[T_1^M\oplus T_2^M]\rangle}{\langle [M], [T_1^M\oplus T_2^M]\rangle\langle [K], [T_1^K\oplus T_2^K]\rangle\langle [T_0^K],[T_1^M\oplus T_2^M]\rangle\langle[ T_0^M],[T_0^K]\rangle} \\
&[T_0^M\oplus T_0^K]\diamond [T_1^M\oplus T_2^M\oplus T_1^K\oplus T_2^K]^{-1}\\
=&\frac{1}{\langle [M],[K]\rangle\langle [M\oplus K], [T_1^M\oplus T_2^M\oplus T_1^K\oplus T_2^K]\rangle} [T_0^M\oplus T_0^K]\diamond [T_1^M\oplus T_2^M\oplus T_1^K\oplus T_2^K]^{-1}\\
=&\frac{1}{\langle [M],[K]\rangle}\psi([M\oplus K])\\
=&\psi([M]\diamond[K]),
\end{align*}
where $[M]\diamond [K]$ is defined in  $\ch(\cc_{\Z/2}(\ca))/J_{\Z/2}$ since it is an $\A_{\Z/2,ac}(\ca)$-bimodule.
Similarly, we can prove that $\psi([K]\diamond[M])=\psi([K])\diamond \psi([M])$.

Furthermore, for any $K\in\cc_{\Z/2,ac}(\ca)$, $\psi([K])$ is invertible in $\ch(\cc_{\Z/2}(\add T))[T_{\Z/2}^{-1}]$, and so
$\psi$ induces a morphism of $\T_{\Z/2,ac}(\ca)$-bimodules
$$\widetilde{\psi}:\cm_{\Z/2}(\ca) \longrightarrow \ch(\cc_{\Z/2}(\add T))[T_{\Z/2}^{-1}].$$
In fact,
$$\widetilde{\psi}(s_1^{-1}\otimes a \otimes s_2^{-1})=\psi(s_1)^{-1}\psi(a)\psi(s_2)^{-1},\mbox{ for any }s_1^{-1}\otimes a \otimes s_2^{-1}\in \cm_{\Z/2}(\ca).$$

By Proposition \ref{proposition ideal}, the natural projection $\widehat{\Upsilon}:\ch(\cc_{\Z/2}(\ca))\rightarrow \ch(\cc_{\Z/2}(\ca))/J_{\Z/2}$ induces
an isomorphism $\cs\cd\ch(\ca)\xrightarrow{\sim}\cm_{\Z/2}(\ca)$.
Then there is a linear map from $\cs\cd\ch(\ca)$ to $\ch(\cc_{\Z/2}(\add T))[T_{\Z/2}^{-1}]$, which is also denoted by $\widetilde{\psi}$.

Clearly, $\widetilde{\psi}\circ\widetilde{\phi}=\id$. On the other hand, for any $M\in\cc_{\Z/2}(\ca)$, we have
\begin{eqnarray*}
[M]&=&\frac{1}{\langle [M] ,[T_2^M]\rangle}[X^M]\diamond [T_2^M]^{-1}\\
&=&\frac{1}{\langle [M] ,[T_2^M]\rangle\langle [X^M],[T_1^M]\rangle}[T_0^M]\diamond [T_1^M]^{-1}\diamond [T_2^M]^{-1}\\
&=&\frac{\langle [T_2^M],[T_1^M]\rangle}{\langle [M] ,[T_2^M]\rangle\langle [X^M],[T_1^M]\rangle }[T_0^M]\diamond [T_1^M\oplus T_2^M]^{-1}\\
&=&\frac{1}{\langle [M] ,[T_1^M\oplus T_2^M]\rangle }[T_0^M]\diamond [T_1^M\oplus T_2^M]^{-1}
\end{eqnarray*}
in $\cs\cd\ch(\ca)$.
So $\widetilde{\phi}\circ\widetilde{\psi}=\id$. It follows that $\widetilde{\phi}$ is an isomorphism with $\widetilde{\psi}$ being its inverse. The proof is completed.
\end{proof}

As a corollary, we prove the tilting invariance of the semi-derived Ringel-Hall algebras. 

\begin{corollary}\label{corollary invariance of derived equivalence}
Let $\ca$ be a hereditary abelian $\K$-linear category with a tilting object $T$. If $\Lambda=\End(T)^{op}$ is a hereditary algebra, then
$\cs\cd\ch(\ca)\simeq \cs\cd\ch(\mod \Lambda)$. In particular, the Drinfeld double of $\ch^e_{tw}(\ca)$ is isomorphic to that of $\ch^e_{tw}(\mod\Lambda)$.
\end{corollary}
\begin{proof}
Since $\ca$ and $\mod\Lambda$ are hereditary categories, it follows from Theorem \ref{proposition isomorphism of semi derived hall algebras } that
$$\cs\cd\ch(\ca)\cong \ch(\cc_{\Z/2}(\add T))[T_{\Z/2}^{-1}].$$
Let $\cp$ be the exact category of $\mod\Lambda$ formed by the finitely generated projective modules.
Let $R_{\Z/2}$ be the subset of $\ch(\cc_{\Z/2}(\add \Lambda))$ formed by all $a[K]$, where $a\in \Q^\times$, $K\in\cc_{\Z/2,ac}(\cp)$.
By the same arguments as in the proof of Proposition \ref{lemma hall algebra of tilting objects}, the right localization of $\ch(\cc_{\Z/2}(\cp))$ with respect to $R_{\Z/2}$ exists, and will be denoted by $\ch(\cc_{\Z/2}(\cp))[R_{\Z/2}^{-1}]$.
Theorem \ref{proposition isomorphism of semi derived hall algebras } implies that
$$\cs\cd\ch(\mod\Lambda)\cong\ch(\cc_{\Z/2}(\cp))[R_{\Z/2}^{-1}].$$
Since $\add T$ is equivalent to $\cp$ as exact categories, we have
$$\ch(\cc_{\Z/2}(\add T))[T_{\Z/2}^{-1}]\cong \ch(\cc_{\Z/2}(\cp))[R_{\Z/2}^{-1}].$$
By combining all these isomorphisms, we obtain that $\cs\cd\ch(\ca)\simeq \cs\cd\ch(\mod \Lambda)$.

The last statement follows from Theorem \ref{theorem semi-derived hall algebra isomorphic to Drinfeld double}.
\end{proof}

\subsection{Relations to Bridgeland's Hall algebras}
In this subsection, we consider the relations among the semi-derived Ringel-Hall algebras, the $\Z/2$-graded semi-derived Hall algebras defined by Gorsky, and Bridgeland's Hall algebras.

We recall the definition of $\Z/2$-graded semi-derived Hall algebra from \cite{Gor13}.
Let $\ce$ be an exact category satisfying the following conditions:
\begin{itemize}
\item[(C1)] ${\ce}$ is essentially small, idempotent complete and linear over some ground field $\K$.
\item[(C2)] For each pair of objects $A,B\in\ce$ and for each $p>0$, we have
$|\Ext^p_\ce(A,B)|<\infty$, $|\Hom_\ce(A,B)|<\infty$.
\item[(C3)] For each pair of objects $A,B\in\ce$, there exists $N>0$ such that for all $p>N$, we have
$\Ext^p_\ce(A,B)=0$.
\item[(C4)] ${\ce}$ has enough projectives, and each object has a finite projective resolution.
\end{itemize}

Let $\cp$ denote the full subcategory of the projective objects of ${\ce}$. Denote by $\widetilde{\ce}$ the closure with respect to extensions and quasi-isomorphism
classes of the full subcategory of all stalk complexes inside $\cc_{\Z/2}(\ce)$.
Define the {\em left relative Grothendieck monoid} $M_0'(\widetilde{{\ce}})$ as the free monoid generated by the set $\Iso(\widetilde{{\ce}})$, modulo the following set of relations:
$$\{ [L]=[K\oplus M]\mid \exists\text{ an exact sequence }0\rightarrow K\rightarrow L\rightarrow M\rightarrow0\mbox{ with } K\in\cc_{\Z/2,ac}({\ce})\}.$$
The quantum affine space $\A_{\Z/2,ac}({\ce})$ and the quantum torus $\T_{\Z/2,ac}({\ce})$ are defined in the same way as in Subsection \ref{subsec:quantum torus}.

Consider the $\Q$-vector space $\cm'_{\Z/2}({\ce})$ with a basis parametrized by the elements of $M_0'(\widetilde{{\ce}})$.
Define an $\A_{\Z/2,ac}({\ce})$-bimodule structure on $\cm'_{\Z/2}({\ce})$ by setting
$$[K]\diamond[M]:=\frac{1}{\langle [K],[M]\rangle}[K\oplus M],\quad [M]\diamond [K]:=\frac{1}{\langle [M],[K]\rangle} [M\oplus K]$$
for $K\in\cc_{\Z/2,ac}({\ce}),M\in\widetilde{{\ce}}$; cf. \eqref{definition of bimodule}. Then
$$\cm_{\Z/2}({\ce}):= \T_{\Z/2,ac}({\ce})\otimes_{\A_{\Z/2,ac}({\ce})}\cm'_{\Z/2}({\ce})\otimes_{\A_{\Z/2,ac}({\ce})}\T_{\Z/2,ac}({\ce})$$
is a bimodule over the quantum torus $\T_{\Z/2,ac}({\ce})$; cf. \eqref{def:bimodule}.

\begin{definition}[\cite{Gor13}]
\label{def:semiderivedGor}
We endow $\cm_{\Z/2}({\ce})$ with the following multiplication: for $[L],[M]\in\Iso(\widetilde{{\ce}})$,
\begin{equation}
[L]\diamond [M]=\frac{1}{\langle [A_L],[L]\rangle} [A_L]^{-1}\diamond \sum_{X\in\Iso(\widetilde{{\ce}})}(\frac{|\Ext^1_{\widetilde{{\ce}}}(P_L,M)_X|}{|\Hom_{\widetilde{{\ce}}}(P_L,M)|}[X]),
\end{equation}
where $A_L\rightarrowtail P_L{\twoheadrightarrow}L$ is an arbitrary conflation with $A_L\in \cc_{\Z/2,ac}({\ce}),P_L\in\widetilde{\cp}$. We call the resulting algebra the
$\Z/2$-graded semi-derived Hall algebra $\cs\cd\ch_{\Z/2}({\ce})$.
\end{definition}
It is proved in \cite{Gor13} that $\cs\cd\ch_{\Z/2}({\ce})$ is an associative unital algebra. Similarly to \eqref{eq:twistedprod}, one can define the twisted $\Z/2$-graded semi-derived Hall algebra (over $\C$) for $\ce$, which is denoted by $\cs\cd\ch_{\Z/2,tw}({\ce})$.
For the exact category $\add T$, its $\Z/2$-graded semi-derived Hall algebra satisfies
$$\cs\cd\ch_{\Z/2}(\add T)\cong \ch(\cc_{\Z/2}(\add T))[T_{\Z/2}^{-1}],$$
where $T_{\Z/2}$ is the set of all classes of acyclic $\Z/2$-graded complexes in $\cc_{\Z/2}(\add T)$; see \cite[Corollary 9.18]{Gor13}.

Let $\Lambda:=\End(T)^{op}$, and let $\cd\ch(\Lambda)$ be its Bridgeland's Hall algebra \cite{Br}. In order to avoid confusion, in the following, we change the notation of our version of semi-derived Ringel-Hall algebra to be $\cm\ch(\ca)$; correspondingly, we use $\cm\ch_{tw}(\ca)$ to denote its twisted version.
\begin{corollary}\label{corollary isomorphic of hall algebras}
Let $\ca$ be a hereditary abelian $\K$-linear category with a tilting object $T$ and $\Lambda=\End(T)^{op}$. Then there exist algebra isomorphisms
$$\cd\ch(\Lambda)\simeq\cs\cd\ch_{\Z/2,tw}(\mod\Lambda ) \xleftarrow{\sim}\cs\cd\ch_{\Z/2,tw}(\add T) \xrightarrow{\sim}\cm\ch_{tw}(\ca).$$
\end{corollary}
\begin{proof}
The first two isomorphisms follow from \cite[Corollary 9.18-9.19]{Gor13}, and the last one follows from Theorem \ref{proposition isomorphism of semi derived hall algebras } since $\cs\cd\ch_{\Z/2}(\add T)\cong \ch(\cc_{\Z/2}(\add T))[T_{\Z/2}^{-1}]$.
\end{proof}

\begin{corollary}\label{corollary semi derived Ringel-Hall algebra of canonical algebra isomorphic to weighted projective line}
Let $\ca$ be the category of coherent sheaves on a weighted projective line over $\K$, and $T$ be the tilting object with $\Lambda=\End(T)^{op}$ the canonical algebra. Then there exists an algebra isomorphism
$$\cd\ch(\mod\Lambda)\simeq\cs\cd\ch_{\Z/2,tw}(\mod \Lambda)\xrightarrow{\sim} \cm\ch_{tw}(\ca).$$
\end{corollary}

\begin{corollary}
Let $\ca$ be the category of coherent sheaves on a weighted projective line over $\K$. Let $\Lambda$ be the canonical algebra of the same type, i.e., there exists a tilting object $T$ in $\ca$ such that $\Lambda\cong \End(T)^{op}$. Then
the twisted semi-derived Hall algebra
$\cs\cd\ch_{\Z/2,tw}(\mod\Lambda)$ and also Bridgeland's Hall algebra $\cd\ch(\mod \Lambda)$ are isomorphic to the Drinfeld double of $\ch^e_{tw}(\ca)$.
\end{corollary}
\begin{proof}
It follows from Corollary \ref{corollary semi derived Ringel-Hall algebra of canonical algebra isomorphic to weighted projective line} and Theorem \ref{theorem semi-derived hall algebra isomorphic to Drinfeld double}.
\end{proof}

\end{document}